\documentclass[12pt]{amsart}
\usepackage{amsmath,amsthm,latexsym,amsfonts,amssymb}
\usepackage{datetime}

\newtheoremstyle{slthm}
{9pt}
{5pt}
{\slshape}
{}
{\bfseries}
{.}
{.5em}
{\thmname{#1}\thmnumber{ #2}\thmnote{ (#3)}}

\newtheoremstyle{prcl}
{9pt}
{5pt}
{\slshape}
{}
{\bfseries}
{.}
{.5em}
{\thmname{#3}\thmnumber{ #2}}

\newtheoremstyle{prblm}
{9pt}
{5pt}
{\rm}
{}
{\bfseries}
{.}
{.5em}
{\thmname{#3}\thmnumber{ #2}}

\theoremstyle{slthm}
\newtheorem{thm}{Theorem}[section]
\newtheorem{lemma}[thm]{Lemma}
\newtheorem{prop}[thm]{Proposition}
\newtheorem{cor}[thm]{Corollary}
\newtheorem{fact}[thm]{Fact}

\theoremstyle{definition}

\newtheorem{df}[thm]{Definition}

\newtheorem{nrmk}[thm]{Remark}
\newtheorem{nrmks}[thm]{Remarks}
\newtheorem{expl}[thm]{Example}
\newtheorem{expls}[thm]{Examples}

\theoremstyle{remark}
\newtheorem*{rmk}{Remark}
\newtheorem*{rmks}{Remarks}

\theoremstyle{prcl}
\newtheorem*{prclaim}{Proclaim}

\theoremstyle{prblm}

\newenvironment{renumerate}
{\renewcommand{\labelenumi}{\rm (\roman{enumi})}
	\begin{enumerate}}
	{\end{enumerate}}

\newcounter{flexnummark}

\DeclareMathOperator{\eh}{eh}
\DeclareMathOperator{\lm}{lm}

\DeclareMathOperator{\level}{level}

\DeclareMathOperator{\pp}{\mathfrak{p}}
\DeclareMathOperator{\mm}{\mathfrak{m}}
\DeclareMathOperator{\nn}{\mathfrak{n}}

\DeclareMathOperator{\im}{Im}
\DeclareMathOperator{\re}{Re}
\DeclareMathOperator{\supp}{supp}
\DeclareMathOperator{\ord}{ord}

\DeclareMathOperator{\sign}{sgn}

\DeclareMathOperator{\bexp}{\textbf{exp}}
\DeclareMathOperator{\blog}{\textbf{log}}
\DeclareMathOperator{\compclass}{cc}

\newcommand{\rest}[1]{\!\!\upharpoonright_{#1}}

\newcommand{\into}{\longrightarrow}

\renewcommand{\tilde}{\widetilde}
\renewcommand{\bar}{\overline}

\def\Ind#1#2{#1\setbox0=\hbox{$#1x$}\kern\wd0\hbox to 0pt{\hss$#1\mid$\hss}
	\lower.9\ht0\hbox to 0pt{\hss$#1\smile$\hss}\kern\wd0}

\def\Notind#1#2{#1\setbox0=\hbox{$#1x$}\kern\wd0\hbox to 0pt{\mathchardef
		\nn=12854\hss$#1\nn$\kern1.4\wd0\hss}\hbox to
	0pt{\hss$#1\mid$\hss}\lower.9\ht0 \hbox to
	0pt{\hss$#1\smile$\hss}\kern\wd0}

\newcommand{\set}[1]{\left\{#1\right\}}

\newcommand{\NN}{\mathbb{N}}
\newcommand{\ZZ}{\mathbb{Z}}

\newcommand{\RR}{\mathbb{R}}

\newcommand{\CC}{\mathbb{C}}

\newcommand{\curly}[1]{\mathcal{#1}}
\newcommand{\A}{\curly{A}}

\newcommand{\C}{\curly{C}}
\newcommand{\D}{\curly{D}}
\newcommand{\E}{\curly{E}}
\newcommand{\F}{\curly{F}}

\renewcommand{\H}{\curly{H}}
\newcommand{\I}{\curly{I}}

\newcommand{\K}{\curly{K}}

\newcommand{\M}{\curly{M}}

\renewcommand{\P}{\curly{P}}

\renewcommand{\S}{\curly{S}}

\newcommand{\U}{\curly{U}}

\renewcommand{\aa}{\mathbf{a}}

\newcommand{\ff}{\mathbf{f}}
\renewcommand{\gg}{\mathbf{g}}
\newcommand{\hh}{\mathbf h}

\renewcommand{\mm}{\mathbf{m}}
\renewcommand{\nn}{\mathbf{n}}
\renewcommand{\pp}{\mathbf{p}}

\newcommand{\la}{\curly{L}}

\DeclareMathOperator{\Ranexp}{\RR_{an,exp}}
\DeclareMathOperator{\Hanexp}{\H_{an,exp}}

\DeclareMathOperator{\Lanexp}{\la_{an,exp}}

\newcommand{\hplus}{\H^{>0}}

\newcommand{\Ps}[2]{\mathbb{#1}\left[\!\left[#2\right]\!\right]}

\newcommand{\As}[2]{#1\left[\!\left[#2\right]\!\right]}

\newcommand{\Gs}[2]{#1\left(\!\left(#2\right)\!\right)}

\allowdisplaybreaks[2]
\numberwithin{equation}{section}

\title {Ilyashenko algebras based on transserial asymptotic expansions}

\author {Zeinab Galal, Tobias Kaiser and Patrick Speissegger}

\address {Laboratoire IRIF \\
	Universit\'e Paris-Diderot\ --\ Paris 7 \\
	Case 7014 \\
	75205 PARIS Cedex 13, France}

\email {zgalal@irif.fr}

\address{Universit\"at Passau \\
Fakult\"at f\"ur Informatik und Mathematik \\
Innstr. 33 \\
94032 Passau \\
Germany}
\email{tobias.kaiser@uni-passau.de}

\address {Department of Mathematics and Statistics, McMaster University, 1280
Main Street West, Hamilton, Ontario L8S 4K1, Canada}

\email {speisseg@math.mcmaster.ca}

\date{\today\ at \currenttime}



\thanks{Supported by NSERC of Canada grant RGPIN 261961 and the Zukunftskolleg of the University of Konstanz.}

\begin{document}

\begin{abstract}
	We construct a Hardy field that contains Ilyashenko's class of germs at $+\infty$ of almost regular functions found in \cite{Ilyashenko:1991fk} as well as all $\log$-$\exp$-analytic germs.  This implies non-oscillatory behavior of almost regular germs with respect to all $\log$-$\exp$-analytic germs.  In addition, each germ in this Hardy field is uniquely characterized by an asymptotic expansion that is an LE-series as defined by van den Dries et al. \cite{MR1848569}.  As these series generally have support of order type larger than $\omega$, the notion of asymptotic expansion itself needs to be generalized.  
\end{abstract}

\maketitle
\markboth{Zeinab Galal, Tobias Kaiser and Patrick Speissegger}{Ilyashenko algebras}

\section*{Introduction}

The purpose of this paper is to extend Ilyashenko's construction in \cite{Ilyashenko:1991fk} of the class of germs at $+\infty$ of almost regular functions to obtain a Hardy field containing them.  In addition, each germ in this Hardy field is uniquely characterized by an asymptotic expansion that is an LE-series as defined by van den Dries et al. \cite{MR1848569} and a transseries as defined by van der Hoeven \cite{Hoeven:2006qr}.  As these series generally have support of order type larger than $\omega$, the notion of asymptotic expansion itself needs to be generalized.  This can be done naturally in the context of a quasianalytic algebra, leading to our definition of \textit{quasianalytic asymptotic algebra}, or \textit{qaa algebra} for short.  Any qaa algebra constructed by generalizing Ilyashenko's construction will be called an \textbf{Ilyashenko algebra}.  

The Hardy field $\H = \Hanexp$ of all unary germs at $+\infty$ of unary functions definable in the o-minimal structure $\Ranexp$ is an example of an Ilyashenko field; see van den Dries and Miller \cite{Dries:1994eq} and van den Dries et al. \cite{Dries:1994tw}.  The third author's paper \cite{MR3744892} contains a first attempt at constructing an Ilyashenko field $\F$ containing Ilyashenko's almost regular germs.  The implied non-oscillatory properties of its germs were used in Belotto et al.'s recent solution \cite{belotto2018} of the strong Sard conjecture.  However, this field $\F$ does not contain $\H$; the Ilyashenko field constructed here is a Hardy field that contains both $\F$ and $\H$, implying non-oscillatory behaviour with respect to all $\log$-$\exp$-analytic germs.  

Our main motivation for generalizing Ilyashenko's construction in this way is the conjecture that the class of almost regular germs generates an o-minimal structure over the field of real numbers.  This conjecture, in turn, might lead to locally uniform bounds on the number of limit cycles in subanalytic families of real analytic planar vector fields all of whose singularities are hyperbolic.  Establishing such uniform bounds for planar polynomial vector fields follows Roussarie's approach \cite{MR3134495} to Hilbert's 16th problem (part 2); see Ilyashenko \cite{Ilyashenko:2002fk} for an overview on the latter.  Our conjecture implies a generic instance of Roussarie's finite cyclicity conjecture \cite{Roussarie:1988fk}; see the third author's preprint \cite{Speissegger:2018} explaining this connection. In Kaiser et al. \cite{Kaiser:2009ud} we gave a positive answer to our conjecture in the special case where all singularities are, in addition, non-resonant.  (For a different approach to the general hyperbolic case, see Mourtada \cite{mourtada2009}.)  

The almost regular germs also play a role in the description of Riemann maps and solutions of Dirichlet's problem on semianalytic domains; see Kaiser \cite{MR2495077,MR2481954} for details.  Finally, in the spirit of the concluding remark of van den Dries et al. \cite{Dries:1997jl}, this paper provides a rigorous construction of a Hardy subfield of \'Ecalle's field of ``fonctions analysables'' \cite{MR1094378} that properly extends $\H$, and we do so without the use of ``acc\'el\'ero-sommation''; for more details on this, see the concluding remarks in Section \ref{final}.

We plan to eventually settle our o-minimality conjecture by adapting the procedure in \cite{Kaiser:2009ud}, which requires three main steps:
\begin{enumerate}
	\item extend the class of almost regular germs into an Ilyashenko field;
	\item construct corresponding algebras of germs of functions in several variables, such that the resulting system of algebras is stable under various operations (such as blowings-up, say);
	\item obtain o-minimality using a normalization procedure.
\end{enumerate}
While \cite{MR3744892} contains a first successful attempt at Step (1), Step (2) poses some challenges.  For instance, it is not immediately obvious what the nature of LE-series in several variables should be; they should at least be stable under all the operations required for Step (3).  They should also contain the series used in \cite{mourtada2009} to characterize parametric transition maps in the hyperbolic case, which use so-called \textit{\'Ecalle-Roussarie compensators} as monomials.

Our approach to this problem is to enlarge the set of monomials used in asymptotic expansions.  A first candidate for such a set of monomials is the set of all (germs of) functions definable in the \hbox{o-minimal} structure $\Ranexp$ (in any number of variables).  This set of germs is obviously closed under the required operations, because the latter are all definable, and it contains the \'Ecalle-Roussarie compensators.  However, it is too large to be meaningful for use as monomials in asymptotic expansions, as it is clearly not $\RR$-linearly independent (neither in the additive nor the multiplicative sense) and contains many germs that have ``similar asymptotic behavior'' such as, in the case of unary germs, belonging to the same archimedean class.  More suitable would be to find a minimal subclass $\la_n$ of all definable $n$-variable germs such that every definable $n$-variable germ is piecewise given by a convergent Laurent series (or, if necessary, a convergent \textit{generalized} Laurent series, see van den Dries and Speissegger \cite{Dries:1998xr}) in a finite tuple of germs in $\la_n$.

Thus, the purpose of this paper is to determine such a minimal set of monomials $\la = \la_1$ contained in the set $\H$ of all unary germs at $+\infty$ definable in $\Ranexp$, and to further adapt the construction in \cite{MR3744892} to corresponding generalized series in one variable.  Recalling that $\H$ is a Hardy field, we can summarize the results of this paper (Theorems \ref{Construction_thm} and \ref{Closure_thm} below) as follows:

\begin{prclaim}[Main Theorem]
	There is a multiplicative subgroup $\la$ of $\H$ such that the following hold:
	\begin{enumerate}
		\item no two germs in $\la$ belong to the same archimedean class;
		\item every germ in $\H$ is given by composing a convergent Laurent series with a tuple of germs in $\la$;
		\item the construction in \cite{MR3744892} generalizes, after replacing the finite iterates of $\,\log$ with germs in $\la$, to obtain a corresponding Ilyashenko field $\K$.
	\end{enumerate}
	The resulting Ilyashenko field $\K$ is a Hardy field extending $\H$ as well as the Ilyashenko field $\F$ constructed in \cite{MR3744892}.
\end{prclaim}

\begin{rmk}
	As mentioned earlier, the Ilyashenko field $\F$ constructed in \cite{MR3744892} does not extend $\H$.
\end{rmk}

We obtain this set $\la$ of monomials by giving an explicit description of the Hardy field $\H$ as the set of all \textit{convergent LE-series}, as suggested in \cite[Remark 6.31]{MR1848569}, with $\la$ being the corresponding set of \textit{convergent LE-monomials}; see Example 5.15 below.  The proof that the construction in \cite{MR3744892} generalizes to this set $\la$ relies heavily on our recent paper \cite{Kaiser:2017aa}; indeed, our construction here was the main motivation for \cite{Kaiser:2017aa}. In the next two sections, we give a more detailed overview of the definitions and results of this paper (Section \ref{results_section}) and their proof (Section \ref{outline_section}).

\section{Main definitions and results} 
\label{results_section}

We let $\C$ be the ring of all germs at $+\infty$ of continuous functions $f:\RR \into \RR$.  
A germ $f \in \C$ is \textbf{small} if $\lim_{x \to +\infty} f(x) = 0$ and \textbf{large} if $\lim_{x \to +\infty} |f(x)| = \infty$.  To compare elements of $\C$, we use the dominance relation $\prec$ as found in Aschenbrenner and van den Dries \cite[Section 1]{MR2130825}, defined by $f \prec g$ if and only if $g(x)$ is ultimately nonzero and $\lim_{x \to +\infty} f(x)/g(x) = 0$, or equivalently, if and only if $g(x)$ is ultimately nonzero and $f(x) = o(g(x))$ as $x \to +\infty$.  Thus, $f \preceq g$ if and only if $f(x) = O(g(x))$ as $x \to +\infty$, and we write $f \asymp g$ if and only if $f \preceq g$ and $g \preceq f$.  Note that the relation $\asymp$ is an equivalence relation on $\C$, and the corresponding equivalence classes are the \textbf{archimedean classes} of $\C$; we denote by $\Pi_\asymp:\C \into \C/_\asymp$ the corresponding projection map.

We denote by $\H \subseteq \C$ the Hardy field of all germs of unary
functions definable in $\Ranexp$.  
Below, we let $K$ be a commutative ring of characteristic 0 with unit $1$.

Recall from \cite{Dries:1998xr} that a \textbf{generalized power series} over $K$ is a power series $G = \sum_{\alpha \in [0,\infty)^n} a_\alpha X^\alpha$, where $X = (X_1, \dots, X_n)$, each $a_\alpha \in K$ and the \textbf{support} of $G$,
\[ \supp(G):= \set{\alpha \in [0,\infty)^n:\ a_\alpha \ne 0}, \] is contained in a cartesian product of well-ordered subsets of $\RR$.  Moreover, we call the support of $G$ \textbf{natural} (see Kaiser et al. \cite{Kaiser:2009ud}) if, for every $a > 0$, the intersection $[0,a) \cap \Pi_{X_i}(\supp(G))$ is finite, where $\Pi_{X_i}:\RR^n \into \RR$ denotes the projection on the $i$th coordinate.

Throughout this paper, we work with the following series: we fix a multiplicative $\RR$-vector subspace $M$ of $$\H^{>0}:= \set{h \in \H:\ h>0}.$$

\begin{df}\label{Ranexp_generalized_power_series}
	An \textbf{$M$-generalized Laurent series (over $K$)} is a series of the form $n \cdot G(m_0, \dots, m_k)$, where $k \in \NN$, $G(X_0, \dots, X_k)$ is a generalized power series with natural support, $m_0, \dots, m_k \in M$ are small and $n$ belongs to the $\RR$-multiplicative vector space $\langle m_0, \dots, m_k \rangle^\times$ generated by $m_0, \dots, m_k$.  In this situation, we say that $F$ has \textbf{generating monomials} $m_0, \dots, m_k$.
\end{df}

\begin{expl}\label{log_generalized_power_series}
	Every logarithmic generalized power series, as defined in \cite[Introduction]{MR3744892}, is an $L$-generalized Laurent series, where $$L:= \langle \exp, x, \log, \log_2, \dots \rangle^\times$$ is the multiplicative $\RR$-vector space generated by $\{\exp, x, \log, \log_2, \dots\}$ and $\log_i$ denotes the $i$-th compositional iterate of $\log$. 
\end{expl}

Every $M$-generalized Laurent series $F$ belongs to the ring $\Gs{K}{M}$ of \textit{generalized series}, as defined for instance in \cite{Dries:1997jl}.  
Correspondingly, we write $F = \sum_{m \in M} a_m m$, and the \textbf{support} of $F$ is the reverse-well ordered set $$\supp(F):= \set{m \in M:\ a_m \ne 0}.$$  We show in Section \ref{M-gps_section} that the set $\Gs{K}{M}^{\text{ls}}$ of all $M$-generalized Laurent series is a subring of $\Gs{K}{M}$ in general, and is a subfield if $K$ is a field.

\begin{nrmks}\label{order_type_expl}
	Let $F \in \Gs{K}{M}^{\text{ls}}$, and let $m_0, \dots, m_k \in M$ be small, $n \in \langle m_0, \dots, m_k \rangle^\times$ and a generalized power series $G$ with natural support be such that $F = n \cdot G(m_0, \dots, m_k)$.  
	\begin{enumerate}
		\item The support of $F$ is of reverse-order type at most $\omega^{k+1}$ (see Corollary \ref{cc_order-type} below).  
		For instance, the logarithmic generalized power series $$\sum_{m,n \in \NN} x^{-m} \exp^{-n}$$ has reverse-order type $\omega^2$.
		\item The latter is not a unique representation of $F$ as an $M$-generalized Laurent series: taking, say, $$H(X_0, \dots, X_k):= G\left(X_0^2, X_1, \dots, X_k\right),$$ we have $F = n \cdot H\left( \sqrt{m_0}, m_1, \dots, m_k\right)$ as well.
	\end{enumerate}
\end{nrmks}

To justify using $M$-generalized Laurent series as asymptotic expansions, we need some further notations.

\begin{df}
	\label{natural_df}
	\begin{enumerate}
		\item $M$ is an \textbf{asymptotic scale} if $m \asymp 1$ implies $m = 1$, for $m \in M$ (or, equivalently, if every archimedean class of $\H^{>0}$ has at most one representative in $M$).
		\item A set $S \subseteq M$ is called \mbox{\textbf{$M$-natural}} if, for all $a \in M$, the intersection $S \cap (a,+\infty)$ is finite.
	\end{enumerate}
\end{df}	

\begin{expls}
	\label{natural_expls}
	\begin{enumerate}
		\item For $k \in \NN$ we set $$L_k:= \langle \exp, x, \log, \dots, \log_{k-1} \rangle^\times \subset L.$$  It follows from basic calculus that $L$, and hence each $L_k$, is an asymptotic scale.  Every \textit{Dulac series} (see Ilyashenko and Yakovenko \cite[Section 24]{Ilyashenko:2008fk}) belongs to $\Gs{\RR}{L_1}$ and has $L_1$-natural support.
		\item Let $\la$ be the set of all \textit{principal monomials} of $\H$ as defined in \cite[Section 2]{Kaiser:2017aa}.  Since every archimedean class of $\H^{>0}$ has a unique representative in $\la$ \cite[Proposition 2.18(2)]{Kaiser:2017aa}, the latter is a maximal asymptotic scale.
		\item Two germs $g,h \in \C$ are \textbf{comparable} if there exist $r,s>0$ such that $|f|^{r} < |g| < |f|^{s}$ (see Rosenlicht \cite{Rosenlicht:1983sp}).  By Lemma \ref{pairwise_comparable_lemma} below, if $G(X_1, \dots, X_k)$ is a generalized power series with natural support and $m_1, \dots, m_k \in M$ are small and pairwise comparable and $n \in \langle m_1, \dots, m_k \rangle^\times$, the $M$-generalized Laurent series $nG(m_1, \dots, m_k)$ has $\langle m_1, \dots, m_k \rangle^\times$-natural support.
	\end{enumerate}
\end{expls}

We assume from now on that $M$ is an asymptotic scale.
	
\begin{df}\label{asymptotic_expansion}
	Let $f \in \C$ and $F = \sum a_m m \in \Gs{\RR}{M}$.  We say that $f$ has \textbf{asymptotic expansion} $F$ (at $+\infty$) if $\supp(F)$ is $M$-natural and
	\begin{equation*}\tag{$\ast$}
	f - \sum_{m \ge n} a_m m \prec n
	\end{equation*}
	for every $n \in M$.
\end{df}

\begin{expl}\label{asymptotic_expl}
	Every \textit{almost regular} $f \in \C$, in the sense defined in the introduction of \cite{MR3744892}, has an asymptotic expansion in $\Gs{\RR}{L_1}^{\text{ls}}$.
\end{expl}

We denote by $\C(M)$ the set of all $f \in \C$ that have an asymptotic expansion in $\Gs{\RR}{M}$.  By Lemmas \ref{asymptotic_ring_lemma} and \ref{unique_as_exp} below, $\C(M)$ is an $\RR$-algebra, every $f \in \C(M)$ has a unique asymptotic expansion in $T_M(f) \in \Gs{\RR}{M}$, and the map $T_M:\C(M) \into \Gs{\RR}{M}$ is an $\RR$-algebra homomorphism. 
In this paper, we are interested in the following kind of subalgebras of $\C(M)$:

\begin{df}
	\label{qa_df}
	We call a subalgebra $\K$ of $\C(M)$ \textbf{quasianalytic} if the restriction of $T_M$ to $\K$ is injective.
\end{df}

Note that, since $\Gs{\RR}{M}$ is a field, every quasianalytic subalgebra of $\C(M)$ is an integral domain.

We now want to extend the definition of asymptotic expansion to all series in $\Gs{\RR}{M}$, not just the ones with natural support.  However, for this generalization we cannot separate ``asymptotic expansion'' from ``quasianalyticity''; both need to be defined simultaneously in the context of a ring of germs, in the spirit of \cite[Definition 2]{MR3744892}.  

For $F = \sum a_m m \in \Gs{\RR}{M}$ and $n \in M$, we denote by $$F_n := \sum_{m \ge n} a_m m$$ the \textbf{truncation} of $F$ above $n$.  A subset $S \subseteq \Gs{\RR}{M}$ is \textbf{truncation closed} if, for every $F \in S$ and $n \in M$, the truncation $F_n$ belongs to $S$.  

\begin{expl}\label{truncation_closed_quotient}
	The set $T_M(\C(M))$ is truncation closed.
\end{expl}

\begin{df}\label{qaa_df}
	Let $\K \subseteq \C$ be an $\RR$-subalgebra and $T:\K \into \Gs{\RR}{M}$ be an $\RR$-algebra homomorphism.  The triple $(\K,M,T)$ is a \textbf{quasianalytic asymptotic algebra} (or \textbf{qaa algebra} for short) if
	\begin{renumerate}
		\item $T$ is injective;
		\item the image $T(\K)$ is truncation closed;
		\item for every $f \in \K$ and every $n \in M$, we have $$f - T^{-1}((Tf)_n) \prec n.$$
	\end{renumerate}
\end{df}

\begin{expl}
	\label{qaa_field_expl}
	Let $\la$ be the set of all principal monomials of $\H$ as defined in \cite[Section 2]{Kaiser:2017aa}.  We show in Corollary \ref{convergent_prop}(2) below that there is a field homomorphism $S_\la:\H \into \Gs{\RR}{\la}$ such that $(\H,\la,S_\la)$ is a qaa field.  The image of $S_\la$ is what we call the set of all convergent LE-series (or convergent transweries), see Section \ref{convergent_subsection}.
\end{expl}

Let $M' \subset \H^{>0}$ be another asymptotic scale, and let $(\K, M, T)$ and $(\K', M', T')$ be two qaa algebras.  We say that \textbf{$(\K, M, T)$ extends $(\K', M', T')$} if $\K'$ is a subalgebra of $\K$, $M'$ is a multiplicative $\RR$-vector subspace of $M$ and $T \rest{\K'} = T'$.

\begin{thm}[Construction]\label{Construction_thm}
Let $h$ be a finite tuple of small germs in $\la$.  
	\begin{enumerate}
		\item There exists a qaa field $\left(\K_h, \langle h \rangle^\times, T_h\right)$ such that $h \subseteq \K_h$.
		\item If $g$ is finite tuple of small germs in $\la$ and $h \subseteq g$, then the qaa field $(\K_g, \langle g \rangle^\times, T_g)$ extends $\left(\K_h, \langle h \rangle^\times, T_h\right)$.
	\end{enumerate}
\end{thm}

\begin{rmk}
		For general $f \in \K_h$, the series $T_h(f)$ is not convergent. 
\end{rmk}

In view of the Construction Theorem, we consider the set consisting of all qaa fields $\left(\K_h, \langle h \rangle^\times, T_h\right)$, for finite tuple $h$ of small germs in $\la$, partially ordered by the subset ordering on the tuples $h$, and we let $(\K,\la, T)$ be the direct limit of this partially ordered set.

\begin{thm}[Closure]\label{Closure_thm}
	\begin{enumerate}
		\item $(\K,\la, T)$ is a qaa field extending each $\left(\K_h, \langle h \rangle^\times, T_h\right)$.
		\item $(\K,\la,T)$ extends the qaa field $(\F,L,T)$ constructed in \cite[Theorem 3]{MR3744892}.
		\item $(\K,\la,T)$ extends the qaa field $(\H,\la,S_\la)$ of Example \ref{qaa_field_expl} above.
		\item $\K$ is closed under differentiation; in particular, $\K$ is a Hardy field.
	\end{enumerate}
\end{thm}

\section{Outline of proof and the Extension Theorem}
\label{outline_section}

The proof of the Construction Theorem proceeds by adapting the construction in \cite{MR3744892} to the more general setting here.  The role of \textit{standard quadratic domain} there is taken on by the following domains here: for $a \in \RR$, we set $$H(a):= \set{z \in \CC:\ \re z > a}.$$

\begin{df}\label{spd_df}
	A \textbf{standard power domain} is a set 
	$$U_C^\epsilon:= \phi_C^\epsilon(H(0)),$$ where $C>0$, $\epsilon \in (0,1)$ and $\phi_C^\epsilon:H(0) \into U_C^\epsilon$ is the biholomorphic map defined by $$\phi_C^\epsilon(z):= z + C(1+z)^\epsilon,$$ where $(\cdot)^\epsilon$ denotes the standard branch of the power function on $H(0)$ (see Section \ref{spd} for details).
\end{df}

Note that $\epsilon = \frac12$ corresponds to the standard quadratic domains of \cite{MR3744892}.  We use the following consequence of the Phrag\-m\'en-Lindel\"of principle \cite[Theorem 24.36]{Ilyashenko:2008fk}: 

\subsection*{Uniqueness Principle}
	\textsl{Let $U \subseteq \CC$ be a standard power domain and $\phi:U \into \CC$ be holomorphic.  
	If $\phi$ is bounded and
	\begin{equation*}
	\phi\rest{\RR}\ \prec\ \exp^{-n} \quad\text{for each } n \in \NN,
	\end{equation*}
	then $\phi = 0$.}
\medskip 

The Uniqueness Principle follows from \cite[Lemma 24.37]{Ilyashenko:2008fk}, because $x < \phi_C^\epsilon(x)$ for $x > 0$.  The reason for working with standard power domains in place of standard quadratic domains is technical; see the remark following Lemma \ref{comparison_mapping_4} below for details. \medskip

Recall that the construction in \cite{MR3744892} is for the tuples $$\left(\frac1\exp, \frac1{x}, \dots, \frac1{\log_k}\right) = \exp \circ (-x, -\log, \dots, -\log_{k+1}),$$  and it proceeds by induction on $k$.  
To understand how we can generalize this construction to more general sequences $h \in \D^{k+1}$, where $$\D:= \set{h \in \H^{>0}:\ h \prec 1}$$ is the set of all positive small germs, we let $$\I:= \set{h \in \H^{>0}:\ h \succ 1}$$ be the set of all \textbf{infinitely increasing} (i.e., positive large) germs in $\H$ and write $$h = \exp \circ (-f) = \frac1{\exp} \circ f,$$ where $f = (f_0, \dots, f_k)$ with each $f_i \in \I$.  In this situation, we shall also write  $M_f$ for the multiplicative $\RR$-vector subspace $\langle h \rangle^\times$.  

We first recall how the induction on $k$ works in \cite{MR3744892}: assuming the qaa field $(\F_{k-1},L,T_{k-1})$ has been constructed such that every germ in $\F_{k-1}$ has a complex analytic continuation on some standard quadratic domain, we ``right shift'' by $\log$, that is, we
\begin{itemize}
	\item [(i)$_{\text{\cite{MR3744892}}}$] set $\F_{k}':= \F_{k-1} \circ \log$ and define $T_{k}':\F_{k}' \into L$ by $T_{k}'(h \circ \log):= (T_{k-1} h) \circ \log$.
\end{itemize}
Note that, since $\log$ has a complex analytic continuation on any standard quadratic domain with image contained in every standard quadratic domain \cite[Example 3.13(2)]{Kaiser:2017aa}, the tuple $(\F'_{k}, L, T_{k}')$ is also a qaa field as defined in \cite{MR3744892} such that every germ in $\F'_{k}$ has a complex analytic continuation on some standard quadratic domain.  So we 
\begin{itemize}
	\item [(ii)$_{\text{\cite{MR3744892}}}$] let $\A_{k}$ be the $\RR$-algebra of all germs $h \in \C$ that have a bounded, complex analytic continuation on some standard quadratic domain $U$ and an asymptotic expansion $F = \sum h_mm \in \Gs{\F_{k}'}{L_0}$ that holds not only in $\RR$, but in all of $U$, and we set $$T_kh:= \sum (T_k'h_m) m \in \Gs{\RR}{L_k}.$$
\end{itemize}  
(Note that, in general, $T_kh$ is an $L$-series over $\RR$, but not an $L$-genera\-lized Laurent series over $\RR$; this observation was not explicitely mentioned in \cite{MR3744892}.)
The corresponding generalization of asymptotic expansion $(\ast)$ to allowing coefficients in $\F_{k}'$ works, because each germ in $\F_{k}'$ is polynomially bounded, and the quasianalyticity follows from the Uniqueness Principle.  Finally, since $\Gs{\RR}{L_k}$ is a field, the ring $\A_k$ is an integral domain, and we 
\begin{itemize}
	\item [(iii)$_{\text{\cite{MR3744892}}}$] let $\F_k$ be the fraction field of $\A_k$ and extend $T_k$ accordingly.
\end{itemize}  
We represent this construction by the schematic in Figure 1.

\begin{figure}
$$
\begin{matrix}
\RR & \xrightarrow{\text{(UP)}} & \begin{bmatrix} \F_0 \\ e^{-x} \circ (x) \end{bmatrix} \\\\
& \swarrow \circ\log \swarrow &  \\\\
\begin{bmatrix} \F_1' \\ e^{-x} \circ (\log) \end{bmatrix} & \xrightarrow{\text{(UP)}} & \begin{bmatrix} \F_1 \\ e^{-x} \circ (x,\log) \end{bmatrix} \\\\
& \swarrow \circ\log \swarrow &  \\\\
& \vdots & \\
& \xrightarrow{\text{(UP)}} & \begin{bmatrix} \F_{k-1} \\ e^{-x} \circ (x,\log, \dots, \log_{k-1}) \end{bmatrix} \\\\
& \swarrow \circ\log \swarrow &  \\\\
\begin{bmatrix} \F_k' \\ e^{-x} \circ (\log, \dots, \log_{k}) \end{bmatrix} & \xrightarrow{\text{(UP)}} & \begin{bmatrix} \F_k \\ e^{-x} \circ (x,\log, \dots, \log_{k}) \end{bmatrix} \\
\end{matrix}
$$
\caption{Schematic of the construction in \cite{MR3744892}: going horizontally from left to right represents one use of the Uniqueness Principle (UP) and adds $e^{-x}$ to the generating monomials on the left; going from the right to the next lower left represents a right shift by $\log$.}
\end{figure}

Throughout this construction, the following property of $L$ is used: 

\begin{df}
	\label{asymptotic_scale_df}
	Let $M$ be a multiplicative $\RR$-vector subspace of $\H^{>0}$.  We call $M$ a \textbf{strong asymptotic scale} if 
	
	\begin{enumerate}
		\item there is a basis $\{m_0, \dots, m_k\}$ of $M$ consisting of pairwise incomparable small germs;
		\item every $m \in M$ has a complex analytic continuation $\mm:U \into \CC$ on every standard power domain $U$;
		\item for every standard power domain $U$ and every $m,n \in M$, we have $m \prec n$ if and only if $\mm(z) = o(\nn(z))$ as $|z| \to \infty$ in $U$.
	\end{enumerate}
\end{df}

\begin{nrmk}\label{strong_asymptotic_rmk}
	If $M$ is a strong asymptotic scale, then $M$ is an asymp\-totic scale: let $\{m_0, \dots, m_k\}$ be a basis of $M$ consisting of pairwise incomparable small germs such that $m_0 < \cdots < m_k$, and set $m_{k+1}:= 1$.  Let $m \in M$ be such that $m \asymp 1$, let $\alpha_0, \dots, \alpha_k \in \RR$ be such that $$m = m_0^{\alpha_0} \cdots m_k^{\alpha_k},$$ and set $\alpha_{k+1}:= 1$.  Since the $m_i$ are pairwise incomparable, $m$ is comparable to $m_j$, where $j:= \min\{i=0, \dots, k+1:\ \alpha_i \ne 0\}$; hence $m \asymp 1$ implies $\alpha_0 = \cdots = \alpha_k = 0$, that is, $m=1$.
\end{nrmk}

The use of strong aymptotic scales is to extend the notion of asymptotic expansion to standard power domains, see \textbf{strong asymptotic expansions} in Section \ref{asymptotic_section}.

For some of the examples below, we let $\U$ be the set of all \textit{purely infinite} germs in $\H$, as defined in \cite[Section 2]{Kaiser:2017aa}.  Recall that $\la = \exp \circ\ \U$; in particular, two germs $f,g \in \U$ belong to the same archimedean class if and only if the germs $\exp \circ f$ and $\exp \circ g$ are comparable.

\begin{expls}
	\label{as_scale_spd}
	\begin{enumerate}
		\item $L$ is a strong asymptotic scale by \cite[Lemma 8]{MR3744892}.
		\item $\la$ is not a strong asymptotic scale: the germ $e^{-x} \circ x^2$ belongs to $\la$ and is bounded, but its complex analytic continuation on any standard power domain is unbounded.
		\item Not every tuple from $\la$ is a basis consisting of pairwise incomparable small germs: consider the germs $f_0:= x$, $f_1= x-\log$ and $f_2:= \log + \log\log$ in $\U$.  While $\{f_0, f_1, f_2\}$ is additively linearly independent, we have $f_0 \asymp f_1$.  However, $M_f$ has the basis $$e^{-x} \circ (x, \log, \log_2)$$ consisting of pairwise incomparable small germs, because $x$, $\log$ and $\log_2$ belong to distinct archimedean classes.
		\item We show in Lemma \ref{strong_basis_lemma} below that, if each $f_i$ belongs to $\U$, then the additive $\RR$-vector space $\langle f \rangle$ generated by the $f_i$ has a basis consisting of infinitely increasing germs belonging to pairwise distinct archimedean classes; hence, $M_f$ has a basis consisting of pairwise incomparable small germs.
	\end{enumerate}
\end{expls}



The most straightforward generalization of the construction in \cite{MR3744892} is to any sequence $f$ of the form $$f =  \left(g^{\circ 0}, g^{\circ 1}, \dots, g^{\circ k}\right),$$ where $k \in \NN$, $g \in \I$ belongs to a strictly smaller archimedean class than $x$, $g^{\circ i}$ denotes the $i$-th compositional iterate of $g$ and $M_f$ is an asymp\-totic scale on standard power domains, and such that the following holds:
\begin{itemize}
 	\item [$(\dagger)_1$] for every standard power domain $V$, the germ $g$ has a complex analytic continuation $\gg$ on some standard power domain $U$ such that $\gg(U) \subseteq V$.
\end{itemize}
The additional assumption $(\dagger)_1$ means that we can compose on the right (``right shift'') with $g$ in place of $\log$, as in the construction in \cite{MR3744892}.

In general, we assume that $k>0$ and $f_0 > f_1 > \cdots > f_k$ belong to $\I$ and that $M_f$ is an asymptotic scale with basis $e^{-x} \circ f$ consisting of pairwise incomparable small germs; this implies, in particular, that $f_0 \succ \cdots \succ f_k$.  In this situation, we aim to adapt the construction in \cite{MR3744892} as represented by the schematic pictured in Figure 2.  The ``right shifts'' are now by germs of the form $f_{k-i+1} \circ f_{k-i}^{-1}$---which still belong to $\H$ since they are definable---and the monomials at the $i$-th step are $e^{-x} \circ f^{\langle i \rangle}$, where we set $$f^{\langle i \rangle} := \left(x, f_{k-i+1} \circ f_{k-i}^{-1}, \dots, f_k \circ f_{k-i}^{-1}\right)$$ and $$f^{\langle i \rangle'} := \left(f_{k-i+1} \circ f_{k-i}^{-1}, \dots, f_k \circ f_{k-i}^{-1}\right).$$  In particular, we have $f^{\langle 0 \rangle} = (x)$, so that the first step in the construction yielding $\K_0 = \K_{f,0}$ is the same as the first  step of the construction in Figure 1, that is, $\K_0 = \F_0$.

\begin{figure}
	$$
	\begin{matrix}
	\RR & \xrightarrow{\text{(UP)}} & \begin{bmatrix} \K_{f,0} \\ e^{-x} \circ f^{\langle 0 \rangle} \end{bmatrix} \\ \\
	& \swarrow \circ \left(f_k \circ f_{k-1}^{-1}\right) \swarrow &  \\\\
	\begin{bmatrix} \K_{f,1}' \\ e^{-x} \circ f^{\langle 1 \rangle'} \end{bmatrix} & \xrightarrow{\text{(UP)}} & \begin{bmatrix} \K_{f,1} \\ e^{-x} \circ f^{\langle 1 \rangle} \end{bmatrix} \\\\
	& \swarrow \circ\left(f_{k-1} \circ f_{k-2}^{-1}\right) \swarrow &  \\\\
	& \vdots & \\
	& \xrightarrow{\text{(UP)}} & \begin{bmatrix} \K_{f,i-1} \\ e^{-x} \circ f^{\langle i-1 \rangle} \end{bmatrix} \\\\
	& \swarrow \circ \left(f_{k-i+1} \circ f_{k-i}^{-1} \right) \swarrow &  \\\\
	\begin{bmatrix} \K_{f,i}' \\ e^{-x} \circ f^{\langle i \rangle'} \end{bmatrix} & \xrightarrow{\text{(UP)}} & \begin{bmatrix} \K_{f,i} \\ e^{-x} \circ f^{\langle i \rangle} \end{bmatrix} \\
	& \vdots & \\
	& \xrightarrow{\text{(UP)}} & \begin{bmatrix} \K_{f,k-1} \\ e^{-x} \circ f^{\langle k-1 \rangle} \end{bmatrix} \\\\
	& \swarrow \circ \left(f_1 \circ f_{0}^{-1} \right) \swarrow &  \\\\
	\begin{bmatrix} \K_{f,k}' \\ e^{-x} \circ f^{\langle k \rangle'} \end{bmatrix} & \xrightarrow{\text{(UP)}} & \begin{bmatrix} \K_{f,k} \\ e^{-x} \circ f^{\langle k \rangle} \end{bmatrix} \\\\
	& \swarrow \circ f_0 \swarrow &  \\\\
	\begin{bmatrix} \K_f \\ e^{-x} \circ f \end{bmatrix} &
	\end{matrix}
	$$
	\caption{Schematic of the generalized construction: going horizontally from left to right represents one use of the Uniqueness Principle (UP) and adds $e^{-x}$ to the generating monomials on the left; going from the right to the next lower left represents a ``right shift'' by $f_{k-i+1} \circ f_{k-i}^{-1}$.  One final right shift by $f_0$ yields the desired qaa field $\K_f$.}
\end{figure}

To determine what additional conditions $f$ has to satisfy in order for this adaptation to go through at the $i$-th step, we assume that $M_{f^{\langle i-1 \rangle}}$ is a strong asymptotic scale with basis $e^{-x} \circ f^{\langle i-1 \rangle}$ consisting of pairwise incomparable germs, and that we have constructed $\K_{i-1} = \K_{f,i-1}$ such that every $h \in \K_{i-1}$ has an analytic continuation $\hh$ on some standard power domain.  (We shall omit the subscript ``$f$'' in $\K_{f,i}$ if clear from context.)  Provided that
\begin{itemize}
	\item [$(\dagger)_2$] for every standard power domain $V$, the germ $f_{k-i+1} \circ f_{k-i}^{-1}$ has a complex analytic continuation $\ff_{i,i-1}$ on some standard power domain $U$ such that $\ff_{i,i-1}(U) \subseteq V$, 
\end{itemize}
the set $M_{f^{\langle i \rangle'}}$ is also a strong asymptotic scale (because $\ff_{i,i-1}^{-1}$ maps standard power domains \textit{into} standard power domains) with  basis $e^{-x} \circ f^{\langle i \rangle'}$ consisting of pairwise incomparable small germs.  Therefore, we right shift by $f_{k-i+1} \circ f_{k-i}^{-1}$, that is, we
\begin{itemize}
	\item [(i)] set $\K_i' = \K_{f,i}':= \K_{i-1} \circ \left( f_{k-i+1} \circ f_{k-i}^{-1} \right)$ and define $T_i' = T_{f,i'}:\K_i' \into \Gs{\RR}{M_{f^{\langle i \rangle'}}}$ by $$T_i'\left(h \circ \left( f_{k-i+1} \circ f_{k-i}^{-1} \right)\right):= (T_{i-1} h) \circ \left( f_{k-i+1} \circ f_{k-i}^{-1} \right).$$
\end{itemize} 
Again by assumption $(\dagger)_2$, the triple $\left(\K'_i, M_{f^{\langle i \rangle'}}, T_i'\right)$ is a qaa field such that every germ in $\K'_i$ has a complex analytic continuation on some standard power domain.  So we 
\begin{itemize}
	\item [(ii)] let $\A_i$ be the set of all germs $h \in \C$ that have a bounded, complex analytic continuation on some standard power domain $U$ and a \textbf{strong asymptotic expansion} $$H = \sum h_m m \in \Gs{\K_i'}{M_{x}}$$ in $U$ (that is, this asymptotic expansion holds as $|z| \to \infty$ in $U$, see Section \ref{asymptotic_section} for details), and we set $$T_i (h):= \sum T_i'(h_m) m \in \Gs{\RR}{M_{f^{\langle i \rangle}}}.$$
\end{itemize}
As in \cite{MR3744892}, the corresponding generalization of asymptotic expansion $(\ast)$ in Definition \ref{asymptotic_expansion} to allowing coefficients in $\K_{i}'$ works, because each germ in $\K_{i}'$ has comparability class strictly smaller than that of $e^x$ and strictly larger than that of $e^{-x}$, and the quasianalyticity follows from the Uniqueness Principle.  
Finally, we 
\begin{itemize}
	\item [(iii)] let $\K_i$ be the fraction field of $\A_i$, and we extend $T_i$ accordingly.
\end{itemize}  
Iterating this construction leads to the schematic pictured in Figure 2.  The final step, a right shift by $f_0$, leads to the desired qaa field $(\K_f, M_f, T_f)$.  Note that, for this last step, we do not need any analytic continuation assumptions and, consequently, we do not expect analytic continuation of the germs in $\K_f$ on standard power domains.

The crucial additional assumption we need to make this work is $(\dagger)_2$ above, which we need for each $i$.  Requiring this condition to be inherited by all subtuples of $f$, we shall consider the following stronger assumption:
\begin{itemize}
	\item [$(\dagger)$] for $0 \le j < i \le k$ and every standard power domain $V$, the germ $f_i \circ f_j^{-1}$ has a complex analytic continuation on some standard power domain $U$ with image contained in $V$.
\end{itemize}

This leads us to the following condition on general tuples $f$: 

\begin{df}\label{admissible_def}
	We call the tuple $f$ \textbf{admissible} if $(\dagger)$ holds and $M_{f \circ f_0^{-1}}$ is a strong asymptotic scale with basis $e^{-x} \circ \left(f \circ f_0^{-1}\right)$ consisting of pairwise incomparable small germs.
\end{df}

Note that if $f$ is admissible and $g$ is a subtuple of $f$, then $g$ is admissible as well and, in this situation, the above construction shows  (see Proposition \ref{sub_tuple_prop} below) that $(\K_f, M_f, T_f)$ extends $(\K_g, M_g, T_g)$.

Since not every germ in $\I$ satisfies $(\dagger)_2$, not every tuple $f$ is admissible.  To figure out what tuples $f$ are admissible, recall from \cite{Kaiser:2017aa} that a germ $f \in \H$ is \textbf{simple} if $\eh(f) = \level(f)$, where $\eh(f)$ is the exponential height of $f$ as defined in \cite{Kaiser:2017aa} and $\level(f)$ is the level of $f$ as found in \cite{Marker:1997kn}.  In Section \ref{monomial_section}, we use Application 1.3 and Corollary 1.6 of \cite{Kaiser:2017aa} to establish the following: 

\begin{thm}[Admissibility]
	\label{adm_lemma}
	Assume the $f_i$ are simple and have pairwise distinct archimedean classes.  Then $f$ is admissible.
\end{thm}

Note that the Admissibility Theorem fails for non-simple germs in general:

\begin{expl}
	\label{adm_counterexpl}
  Consider the tuple $f = (f_0,f_1):=  \left(x,x+e^{-x^2}\right)$: while $f_0$ has a bounded complex analytic continuation on every standard power domain, the germ $f_1$ does not have a bounded complex analytic continuation on any standard power domain.  In fact, $M_f$ is not a strong asymptotic scale.
\end{expl}

Since every germ in $\U$ is simple \cite[Example 8.7]{Kaiser:2017aa}, we obtain the following from the Admissibility Theorem \ref{adm_lemma}:

\begin{cor}[Admissibility]
	\label{adm_princ_mon}
	If each $f_i$ belongs to $\U$, then every subtuple of $\langle f \rangle$ consisting of infinitely increasing germs belonging to pairwise disjoint archimedean classes is admissible. \qed
\end{cor}

Therefore, if each $f_i \in \U$, we obtain the qaa field $(\K_f,M_f,T_f)$ as follows:  by Example  \ref{as_scale_spd}(4), $\langle f \rangle$ has a basis $g$ consisting of infinitely increasing germs belonging to pairwise disjoint archimedean classes.  By the Admissibility Corollary \ref{adm_princ_mon}, our construction then produces the qaa field $\left(\K_{g}, M_{g}, T_{g}\right)$, and we set $$\K_f := \K_g \quad\text{and}\quad T_f:= T_g.$$  The resulting $\K_f$ is independent of the chosen basis $g$, see Proposition \ref{constr_indep}.

Finally, we show in Section \ref{closure_section} that the  direct limit $(\K,\la,T)$ is maximal in the following sense: if each $f_i$ belongs to $\U$, the qaa field $\left(\K_f, M_f, T_f\right)$ constructed here is extended by $(\K,\la,T)$; this implies, in particular, parts (1) and (2) of the Closure Theorem.  For part (3) of the latter, it suffices to verify that every germ given by a convergent $\la$-generalized power series in $\Gs{\RR}{M_f}$ belongs to $\K_f$.  The proof of part (4) of the Closure Theorem is adapted from the proof of \cite[Theorem 3(2)]{MR3744892}.

\section{Standard power domains}\label{spd}

This section summarizes some elementary properties of standard power domains and makes some related conventions.  Given two germs $f,g \in \C$, we set $f \sim g$ if $g(x) \ne 0$ for sufficiently large $x>0$ and $f(x)/g(x) \to 1$ as $x \to +\infty$.

\begin{lemma}
\label{sqd_lemma}
Let $C>0$ and $\epsilon \in (0,1)$.
\begin{enumerate}
\item The map $\phi_C^\epsilon$ is biholomorphic onto its image.
\item We have $\re\phi_C^\epsilon(ir) \sim C\cos\left(\epsilon \frac\pi2\right) r^\epsilon$ and $\im\phi_C^\epsilon(ir) \sim r$, as $r \to +\infty$ in $\RR$.
\item There exists a continuous $f_C^\epsilon:[C,+\infty) \into (0,+\infty)$ such that $$\im\phi_C^\epsilon(ir) = f_C^\epsilon(\re\phi_C^\epsilon(ir)) \quad\text{ for } r>0$$ and $f_C^\epsilon(r) \sim K_C^\epsilon r^{1/\epsilon}$ as $r \to +\infty$ in $\RR$, where $K_C^\epsilon$ is the constant $\left(C\cos\left(\epsilon\frac\pi2\right)\right)^{-1/\epsilon}$.
\end{enumerate}
\end{lemma}

\begin{proof}
	(1) Note that $$z + C(1+z)^\epsilon - (w + C(1+w)^\epsilon) = (1+z) + C(1+z)^\epsilon - ((1+w) + C(1+w)^\epsilon);$$ so it suffices to show that the map $\psi = \psi_C^\epsilon : H(0) \into \CC$ defined by $\psi(z):= z+C z^\epsilon$ is injective.  Note also that $\psi$ maps the first quadrant $H(0)^+$ into itself, the real line into the real line and the fourth quadrant $H(0)^-$ into itself.  So we let $z,w \in H(0)^+$ be distinct and show that $\psi(z) \ne \psi(w)$; the other cases are similar and left to the reader.  
	
	We may assume that $\re w \ge \re z$, and we let $\Gamma \subseteq H(0)^+$ be the straight segment connecting $z$ to $w$ and parametrized by the corresponding affine linear curve $\gamma:[0,1] \into H(0)^+$; note that $\gamma' = a + ib$ with $a,b \in \RR$ such that $a \ge 0$.
	
	For $\xi \in H(0)^+$, we have $$\psi'(\xi) = 1 + \frac{C\epsilon}{\xi^{1-\epsilon}},$$ and since $\xi^{1-\epsilon} \in S\left((1-\epsilon)\frac\pi2\right) \cap H(0)^+$, it follows that $$\re \psi'(\xi) > 1 \quad\text{and}\quad \im\psi'(\xi) < 0.$$  By the Fundamental Theorem of Calculus for holomorphic functions we have 
	\begin{align*} 
	\psi(w) - \psi(z) &= \int_S \psi'(\xi) d\xi \\
	&= \int_0^1 \left(a\re\psi'(\gamma(t)) - b\im\psi'(\gamma(t))\right) dt \\ &\qquad + i\int_0^1 \left(b\re\psi'(\gamma(t)) + a\im\psi'(\gamma(t))\right) dt.
	\end{align*} 
	Thus, if $b \ge 0$ then, since not both $a$ and $b$ are zero, it follows that $\re(\psi(w)-\psi(z)) > 0$. Arguing similarly if $b \le 0$, we  obtain that $\im(\psi(w)-\psi(z)) < 0$ in this case, and part (1) is proved.
	
	(2) By the generalized binomial theorem we have, for $r>1$, that
	\begin{equation*}
	\phi_C^\epsilon(ir) = ir + C(ir)^\epsilon \sum_{k=0}^\infty \begin{pmatrix}
	\epsilon \\ k
	\end{pmatrix} (ir)^{-k}.
	\end{equation*}
	Taking real and imaginary parts gives
	\begin{equation*}
	\re\phi_C^\epsilon(ir) = Cr^\epsilon \left(c_\epsilon K_\epsilon\left(\frac1r\right) - s_\epsilon L_\epsilon\left(\frac1r\right)\right),
	\end{equation*}
	where $c_\epsilon:= \cos\left(\epsilon\frac\pi2\right)$, $s_\epsilon:= \sin\left(\epsilon\frac\pi2\right)$, $$K_\epsilon(X):= \sum_{k=0}^\infty \begin{pmatrix}
	\epsilon \\ 2k
	\end{pmatrix} (-1)^k X^{2k}$$ and $$L_\epsilon(r):= \sum_{k=0}^\infty \begin{pmatrix}
	\epsilon \\ 2k-1
	\end{pmatrix} (-1)^k X^{2k+1},$$ and
	\begin{equation*}
	\im\phi_C^\epsilon(ir) = r + Cr^\epsilon \left(c_\epsilon L_\epsilon\left(\frac1r\right) + s_\epsilon K_\epsilon\left(\frac1r\right)\right).
	\end{equation*}
	Since the series $K_\epsilon$ and $L_\epsilon$ converge, part (2) follows.
	
	(3) Arguing as in part (1), we get $\re \left(\phi_C^\epsilon\right)'(\xi) > 0$ for $\xi \in H(-1)$ with $\im\xi > 0$.  Thus, the map $h:(0,+\infty) \into (C,+\infty)$ defined by $h(r):= \re\phi_C^\epsilon(ir)$ is injective and has a compositional inverse $g:(C,+\infty) \into (0,+\infty)$, and we define $f_C^\epsilon:[C,+\infty) \into [0,+\infty)$ by $$f_C^\epsilon(t):= \begin{cases} \im\phi_C^\epsilon(ig(t)) &\text{if } t > C, \\ 0 &\text{if } t = C. \end{cases}$$
	
	Now note that $p_{-1} \circ h \circ p_{-1}(t) = H(t)$, where $H \in \Ps{R}{X^*}$ is a convergent generalized power series as in \cite{Dries:1998xr} with leading monomial $\lm(H) = X^\epsilon/Cc_\epsilon$.  By \cite[Lemma 9.9]{Dries:1998xr}, the series $H$ has a compositional inverse $G \in \Ps{R}{X^*}$; it follows that $g = p_{-1} \circ G \circ p_{-1}$.  Since $\lm(H \circ G) = \lm(H) \circ \lm(G)$, and since $\im\phi_C^\epsilon(ir) = p_{-1} \circ I \circ p_{-1}(r)$ for some convergent $I \in \Ps{R}{X^*}$ with leading monomial $X$, it follows that $f_C^\epsilon(r) \sim  K_C^\epsilon r^{1/\epsilon}$ as $r \to +\infty$.
\end{proof}

From now on, we denote by $\phi_C^\epsilon$ the restriction of $\phi_C^\epsilon$ to the closed right half-plane $\bar{H(0)}$.

\subsection*{Convention.}  Given a standard power domain $U$ and a function $g:\RR \into \RR$ that has a complex analytic continuation on $U$, we shall denote this extension by the corresponding boldface letter $\gg$.  \medskip

For $A \subseteq \CC$ and $\epsilon > 0$, let $$T(A,\epsilon):= \set{z \in \CC:\ d(z,A) < \epsilon}$$ be the \textbf{$\epsilon$-neighbourhood} of $A$.

\begin{lemma}
\label{sqd_facts}
Let $C>0$ and $\epsilon \in (0,1)$.  The following inclusions hold as germs at $\infty$ in $H(0)$:
\begin{enumerate}
\item for $D>0$, $\epsilon' \in (\epsilon,1)$ and $\delta > 0$, we have $T\left(U_D^{\epsilon'},\delta\right) \subseteq U_C^\epsilon$;
\item for $D > C$ and $\delta > 0$, we have $T\left(U_D^\epsilon,\delta\right) \subseteq U_C^\epsilon$;
\item for $\nu>0$, we have $$\nu \cdot U_C^\epsilon \subseteq \begin{cases} U_{\nu C}^\epsilon &\text{if } \nu \le 1, \\ U_C^\epsilon &\text{if } \nu \ge 1; \end{cases}$$ 
\item we have $U_C^\epsilon + U_C^\epsilon \subseteq U_{C/2}^\epsilon$;  
\item for any standard power domain $U$,  there exists $a>0$ such that  $\blog\left(U_C^\epsilon\right) \cap H(a) \subseteq U \cap H(a)$.
\end{enumerate}
\end{lemma}

\begin{proof}
	(1) and (2) follow from Lemma \ref{sqd_lemma}(3).
	
	(3) follows from Lemma \ref{sqd_lemma}(3) and the equality $$\nu\cdot\left(x, \left(\frac x{C\cos\left(\epsilon\frac\pi2\right)}\right)^{1/\epsilon}\right) = \left(\nu x, \left(\frac{\nu x}{\nu^{1-\epsilon} C\cos\left(\epsilon\frac\pi2\right)}\right)^{1/\epsilon}\right)$$ in $\RR^2$.
	
	(4) Note that, for $a \in \CC$ with $\re a \ge 0$, the boundary of $a + U_C^\epsilon$ in $\set{z \in \CC:\ \im z \ge \im a}$, viewed as a subset of $\RR^2$, is the graph of a function $f_{a,C}^\epsilon:[C + \re a,+\infty) \into [\im a,+\infty)$ such that $$f_{a,C}^\epsilon(x) \sim \im a + K_C^\epsilon\left(x- \re a\right)^{1/\epsilon} \prec f_{C/2}^\epsilon(x),$$ which proves the claim.
	
	(5) Note that $\blog(\{\re z > 1\}) = \set{z \in H(0):\ |\arg z| < \pi/2}$.
\end{proof}

The following is the main reason for working with standard power domains.

\begin{lemma}
\label{exp_on_sqd}
Let $C>0$, $\epsilon \in (0,1)$ and set $K:= \frac{C\cos\left(\epsilon\frac\pi2\right)}{3^{\epsilon/2}}$.  There exists $k \in (0,1)$ depending on $C$ and $\epsilon$ such that $$k\exp\left(K|z|^\epsilon\right) \le \left|\bexp(z)\right| \le \exp(|z|)$$ for $z \in U_C^\epsilon$.
\end{lemma}

\begin{proof}
	For $r>0$, denote by $C_r$ the circle with center 0 and radius $r$.  Since $|\bexp(x+iy)| = \exp x$, the point in $U \cap C_r$ where $|\bexp z|$ is maximal is $z = r$.  On the other hand, the point $z(r)  = x(r)+iy(r)$ in $U_C^\epsilon \cap C_r$ where $|\bexp z|$ is smallest lies on the boundary of $U_C^\epsilon$, so that $y(r) = f_C^\epsilon(x(r))$.  It follows from Lemma \ref{sqd_lemma}(3) that $$r = \sqrt{x(r)^2 + f_C^\epsilon(x(r))^2} \sim x(r)^{1/\epsilon} \sqrt{x(r)^{2-2/\epsilon} + (K_C^\epsilon)^2}.$$  Hence $x(r) \ge Kr^\epsilon$ for all sufficiently large $r \in \RR$, as required.
\end{proof}

\section{Monomials on standard power domains}
\label{monomial_section}

We now use Corollary 1.6 and Application 1.3 of \cite{Kaiser:2017aa} to find out which tuples $f$ of germs in $\I$ are admissible.  

\begin{expl}
	\label{extension_expl}
	The restriction of $1/\bexp$ to any right half-plane $H(a)$ with $a \in \RR$ is bounded.  Hence $\exp \circ (-x)$ has a bounded complex analytic continuation on every standard power domain.
\end{expl}

Below, we denote by $\arg$ the standard argument on $\CC \setminus (-\infty,0]$.  Recall that, for $\alpha \in (0,\pi]$, we set $$S(\alpha):= \set{z \in \CC:\ |\arg z| < \alpha)}.$$
We will need to work, for $f \in \H$, with both the \textit{exponential height} $\eh(f)$ and the \textit{level} $\level(f)$.  The former measures the logarithmic-exponential complexity of $f$; roughly speaking, if $f$ is unbounded, then $\eh(\exp \circ f) = \eh(f)+1$, while if $f$ is bounded, then $\eh(\exp \circ f) = \eh(f)$ (see \cite[Section 2]{Kaiser:2017aa} for details).  The latter measures the exponential order of growth of the germ $f$; we refer the reader to Marker and Miller \cite{Marker:1997kn} for details.  The level extends to all $\log$-$\exp$-analytic germs in an obvious manner, see \cite[Section 3]{Kaiser:2017aa}.  

\begin{rmks}
	\begin{enumerate}
		\item The two quantities are not equal in general: we have $\level(x+e^{-x}) = 0 \ne 1 = \eh(x+e^{-x})$.  
		\item The map $\level:(\H, \circ) \into (\ZZ,+)$ is a group homomorphism; in particular, for $f \in \H$ and $g \in \I$, we have $\level(f \circ g^{-1}) = \level(f) - \level(g)$.  In contrast, the map $\eh:\H \into \ZZ$ is not a group homomorphism, and the definition of $\eh$ gives no bounds on $\eh(f \circ g^{-1})$ in terms of $\eh(f)$ and  $\eh(g)$.
	\end{enumerate}
\end{rmks}

As in \cite[Section 2]{Kaiser:2017aa}, we set $$\H_{\le 0} := \set{ f \in \H:\ \eh(f) \le 0}.$$

\begin{fact}[Corollary 2.16 in \cite{Kaiser:2017aa}]
	\label{0-eh_cor}
	\begin{enumerate}
		\item Let $f \in \H$.  Then $\eh(f \circ \exp) = \eh(f) + 1$ and $\eh(f \circ \log) = \eh(f) -1$.
		\item Let $f \in \H$ be infinitely increasing.  Then $$\level(f) \le \eh(f).$$
		\item The set $\H_{\le 0}$ is a differential subfield of $\H$.
	\end{enumerate}
\end{fact}

The second fact we need gives an upper bound for exponential height in the situation of the second remark above:

\begin{fact}[Application 1.3 in \cite{Kaiser:2017aa}] 
	\label{eh_application}
	Let $f \in \I$ and $g \in \H$.  Then $$\eh\left(g \circ f^{-1}\right) \le \max\{\eh(g) + \eh(f) - 2\level(f), \eh(f) - \level(f)\}.$$
\end{fact}

The third fact summarizes analytic continuation properties of germs in $\H$ of low enough exponential height and level.  Given a domain $\Omega \subseteq \CC$ and a map $f:\Omega \into \CC$, we call $f$ \textbf{half-bounded} if $f$ or $1/f$ is bounded.  We denote by $\arg(z)$ the standard argument of $z \in \CC \setminus (-\infty,0]$.

\begin{fact}[Corollary 1.6 in \cite{Kaiser:2017aa}]
	\label{ccc}
	Let $f \in \H$ be such that $\eh(f) \le 0$.
	\begin{enumerate}
		\item There are $a \ge 0$ and a half-bounded complex analytic continuation $\ff:H(a) \into \CC$ of $f$.   
		\item Assume in addition that $f \in \I$.  Then 
		\begin{enumerate}
			\item $|\ff(z)| \to \infty$ as $|z| \to \infty$, for $z \in H(a)$;
			\item if $f \prec x^2$, then $\ff(H(a)) \subseteq \CC \setminus (-\infty,0]$, $\ff:H(a) \into \ff(H(a))$ is biholomorphic and we have $$\sign(\arg\ff(z)) = \sign(\arg z) = \sign(\im z) =
			\sign(\im\ff(z))$$ for $z \in H(a)$;
			\item if $\eh(f) < 0$ then, for every $\alpha > 0$, there exists $b \ge a$ such that $\ff(H(a)) \cap H(b) \subseteq S(\alpha)$. \qed
		\end{enumerate}
	\end{enumerate}
\end{fact}

From Fact \ref{ccc}, we immediately get the following:

\begin{cor}
	\label{eh_minus_one}
	Let $f \in \I$ be such that $\eh(f) \le -1$.  Then there exists $a \ge 0$ and a complex analytic continuation $\ff:H(a) \into \CC$ of $f$ such that, for all standard power domains $U,V \subseteq H(0)$, there exists $b \ge a$ with $\ff(U \cap H(b)) \subseteq V$.  In particular, $\bexp \circ (-\ff)$ is a bounded on every standard power domain.
\end{cor}

\begin{proof}
	By Fact \ref{0-eh_cor}(2), we have $\level(f) \le \eh(f) \le -1$, while $\level(x^2) = 0$; hence $f \prec x^2$.  So by Fact \ref{ccc}(2b), there are $a \ge 0$ and a complex analytic continuation $\ff:H(a) \into \CC$ such that $\ff(H(a)) \subseteq \CC \setminus (-\infty,0]$ and $\ff:H(a) \into \ff(H(a))$ is biholomorphic.  Let $V$ be a standard power domain, and let $c \ge 0$ be such that $S(\pi/4) \cap H(c) \subseteq V$.  By Facts \ref{ccc}(2ac), there is $d \ge c$ such that $\ff(H(a)) \cap H(d) \subseteq S(\pi/4) \cap H(c)$.  By Fact \ref{ccc}(2a) again, there is a $b \ge a$ such that $\ff(H(b)) \subseteq \ff(H(a)) \cap H(d)$.  Hence $\ff(U \cap H(b)) \subseteq V$ for any standard power domain $U$, as claimed.
\end{proof}

However, for $f \in \I$ with $\eh(f) = 0$, things are still not clear: while $f = x$ works by Example \ref{extension_expl}, the germ $f = x^2$ does not: if $U' \subseteq H(0)$ is a standard power domain, the set of squares of elements of $U'$ is not contained in any right half-plane $H(a)$ with $a \in \RR$, so the complex analytic continuation $\bexp \circ (-z^2)$ on $U'$ is unbounded.  Arguing similarly, we see that the germ $\exp \circ (-x^r)$ has a bounded complex analytic continuation on some standard power domain if and only if $r \le 1$ (and in this case, it has a bounded complex analytic continuation on all standard power domains).

What about the general $f \in \I$ with $\eh(f) = 0$?  While we do not fully characterize all such $f$ for which $\exp\circ (-f)$ has a bounded complex analytic continuation to some standard power domain, we do give a sufficient condition in Corollary \ref{admissible_prop} below that suffices for our purposes.  

%

To determine which of these germs satisfy $(\dagger)_1$, we also need a notion for studying asymptotic behavior on standard power domains: given  $U \subseteq H(0)$ and $\phi,\psi:U \into \CC$, we write $$\phi \preceq_U \psi \quad\text{iff}\quad \left|\frac{\phi(z)}{\psi(z)}\right| \text{ is bounded in } U$$ and $$\phi \prec_U \psi \quad\text{iff}\quad \lim_{|z| \to \infty, z \in U} \frac{\phi(z)}{\psi(z)} = 0.$$  Correspondingly, we write $\phi \asymp_U \psi$ if both $\phi \preceq_U \psi$ and $\psi \preceq_U \phi$.

We start with an easy case where dominance is preserved on the right half-plane: 

\begin{lemma}
\label{comparison_mapping_2}
Let $f,g \in \H$ be such that $\eh(f), \eh(g) \le 0$, and let $\ff$ and $\gg$ be corresponding complex analytic continuations obtained  from Fact \ref{ccc}.  If $f \prec g$, then there exists $a \ge 0$ such that $\ff \prec_{H(a)} \gg$.  
\end{lemma}

\begin{proof}
	By Fact \ref{0-eh_cor}(3), the germ $h:= \frac gf$ satisfies $\eh\left(h\right) \le 0$; let $\hh$ be a complex analytic continuation of $h$ obtained  from Fact \ref{ccc}.  Now choose $a \ge 0$ such that $\ff$, $\gg$ and $\hh$ are defined on $H(a)$.  By assumption, either $h$ or $-h$ belongs to $\I$; hence $|\hh(z)| \to \infty$ as $|z| \to \infty$ in $H(a)$ by Fact \ref{ccc}(2a).  Since $\ff/\gg = 1/\hh$ in $H(a)$ by the holomorphic identity theorem, the lemma is proved.
\end{proof}

\begin{cor}
	\label{unit_mapping}
	Let $f \in \H$ be such that $\eh(f) \le 0$, and let $\ff$ be a complex analytic continuation of $f$ obtained  from Fact \ref{ccc}.  If we have $\,\lim_{x \to +\infty} f(x) = c \in \RR$, then there exists $a \ge 0$ such that  $\lim_{|z| \to \infty} \ff(z) = c$ in $H(a)$.
\end{cor}

\begin{proof}
	Set $g:= f-c$;  then $\eh(g) \le 0$ by Fact \ref{0-eh_cor}(3) and $g \prec 1$, so that $\gg \prec 1$ in $H(a)$, for some $a \ge 0$, by Lemma \ref{comparison_mapping_2}.
\end{proof}

Recall that, for $\alpha \in [0,\pi]$ we denote by $$S(\alpha) = \set{z \in \CC:\ |\arg z| < \alpha}$$ the sector of opening $\alpha$ and bisecting line $(0,+\infty)$.

\begin{lemma}
\label{comparison_mapping_3}
Let $f \in \I$ be such that $f>0$ and $\eh(f) \le 0$, and let $\ff$ be a complex analytic continuation of $f$ obtained  from Fact \ref{ccc}.  Assume there exists $\epsilon \in (0,1)$ such that and $f \prec x^\epsilon$.  Then $$\ff(H(a)) \subseteq S(\epsilon\cdot\pi/2)$$ for some $a \ge 0$.
\end{lemma}

\begin{proof}
	The assumptions imply that $h:= \frac{x^\epsilon}f \le x^\epsilon \prec x^2$ belongs to $\I$, and $\eh(h) \le 0$ by Fact \ref{0-eh_cor}(3).  So by Fact \ref{ccc}(2b) and the holomorphic identity theorem, for some $a \ge 0$ and all $z \in H(a)$ with $\arg z > 0$, we have
	\begin{equation*}
	0 < \arg \frac{z^\epsilon}{\ff(z)} = \epsilon \arg z - \arg \ff(z),
	\end{equation*}
	so that $\arg\ff(z) < \epsilon \frac{\pi}2$, as claimed.  We argue similarly if $\arg z < 0$.
\end{proof}

\begin{cor}
	\label{comparison_mapping_cor}
	Let $f \in \I$ be such that $f>0$ and $\eh(f) \le 0$, and let $\ff$ be a complex analytic continuation of $f$ obtained from Fact \ref{ccc}.  Assume there exists $\epsilon \in (0,1)$ such that and $f \prec x^\epsilon$.  Then there exists $a \ge 0$ such that
	\begin{enumerate}
		\item for $z \in H(a)$ we have $|\ff(z)| \le |z|^\epsilon$ and $$\re\ff(z) \ge \cos\left(\epsilon \cdot \frac\pi2\right) |\ff(z)|;$$
		\item if $f \in \I$ and $U,V$ are standard power domains, there exists $b \ge a$ such that $\ff(U \cap H(b)) \subseteq V$.
	\end{enumerate}
\end{cor}

\begin{proof}
	From Lemma \ref{comparison_mapping_2} with $g = x^\epsilon$ we get $\ff \prec_{H(a)} \pp_\epsilon$, for some $a \ge 0$, and from Lemma \ref{comparison_mapping_3} we get $\re\ff(z)) \ge \cos\left(\epsilon \cdot \frac\pi2\right) \cdot |\ff(z)|$, for some $a \ge 0$ and all $z \in H(a)$; this proves part (1).  Part (2) follows from Lemma \ref{comparison_mapping_3} and Fact \ref{ccc}(2a), first choosing $c \ge 0$ such that $w \in V$ for all $w \in S(\epsilon \cdot \pi/2)$ with $|w| > c$, then choosing $b \ge a$ such that $|\ff(z)| \ge c$ for $z \in H(b)$.
\end{proof}

\begin{lemma}
\label{comparison_mapping_4}
Let $f \in \I$ be such that $f>0$ and $\eh(f) \le 0$, and let $\ff$ be a complex analytic continuation of $f$ obtained from Fact \ref{ccc}.  Assume that $f \prec x$ but $f \succ x^{\epsilon}$ for all $\epsilon \in (0,1)$.  Then for every standard power domain $V$, there exist $a\ge 0$ and standard power domains $U_1$ and $U_2$ such that $\ff(U_1 \cap H(a)) \subseteq V$ and $\ff(V \cap H(a)) \subseteq U_2$.
\end{lemma}

\begin{rmk}
	We do not know if this lemma remains true with ``quadratic'' in place of ``power''; this is the technical reason for working with standard power domains instead of standard quadratic domains.
\end{rmk}

\begin{proof}
	Let $U = U_C^\delta \subseteq \CC$ be a standard power domain, with $\delta \in (0,1)$ and $C>0$, and let $\delta' \in (0,\delta)$; we claim that there exists $a \ge 0$ such that $\ff(U \cap H(a)) \subseteq U_D^{\delta'}$, where $\left(D \cos\left(\delta'\frac\pi2\right)\right)^{\delta'} = \frac{\cos\left((\delta-\delta') \frac\pi2\right) }{C\cos\left(\delta\frac\pi2\right)}$, which then proves the lemma.  
	
	Set $g:= \frac f{x}$; by assumption and Fact \ref{0-eh_cor}(3), $1/g \in \I$, $\eh(1/g) \le 0$ and $1/g \prec x^\epsilon$ for all $\epsilon \in (0,1)$.  By Fact \ref{ccc}(2a,b), the germ $g$ has a complex analytic continuation $\gg:H(a) \into \CC$, for some $a\ge 0$, such that $\im\gg(z) < 0$ for $z \in U$ with $\im z > 0$.  By Lemma \ref{sqd_lemma}(2) and after shrinking $C$ a bit if necessary, we have $\re z \ge C\cos\left(\delta \frac\pi2\right)|z|^{\delta}$ for sufficiently large $z \in U$.  By Corollary \ref{comparison_mapping_cor}(1) with $\epsilon = \delta - \delta'$, we also have for sufficiently large $z \in U$ with $\im z > 0$, that $|\gg(z)| \ge |z|^{\delta'-\delta}$ and, because $\re \frac1\gg = \frac1{\gg^2}\re \gg$, that $\re \gg(z) \ge \cos\left((\delta-\delta') \frac\pi2\right)  |\gg(z)|$. 
	Since $f = x g$ and $\im\gg(z) < 0$, it follows that
	\begin{align*}
	\re\ff(z) &= (\re z)(\re\gg(z)) - (\im z)(\im\gg(z)) \\
	&> (\re z)(\re\gg(z)) \\
	&\ge (\re z)\cdot \cos\left((\delta-\delta') \frac\pi2\right)  |\gg(z)| \\
	&\ge C\cos\left(\delta \frac\pi2\right)\cos\left((\delta-\delta') \frac\pi2\right)  |z|^{\delta'} .
	\end{align*}
	On the other hand, since $\gg(z) \to 0$ as $|z| \to \infty$ in $U$, we have for sufficiently large $z \in U$ with $\im z > 0$ that
	\begin{multline*}
	0 < \im\ff(z) = (\re z)(\im\gg(z)) + (\im z)(\re\gg(z)) \\ \le (\im z)(\re\gg(z)) \le |z|;
	\end{multline*}
	arguing similarly for $\im z < 0$, we get $|\im\ff(z)| \le |z|$ for all sufficiently large $z \in U$.
	Hence $|\im\ff(z)| < \left(\frac{\re\ff(z)}{C\cos\left(\delta \frac\pi2\right)\cos\left((\delta-\delta') \frac\pi2\right)}\right)^{1/\delta'}$ for such $z \in U$, and the claim is proved by Lemma \ref{sqd_lemma}(2).
\end{proof}

\begin{lemma}
\label{comparison_mapping}
Let $f,g \in \I$ be such that $\eh(f),\eh(g) \le 0$, and let $\ff$ and $\gg$ be corresponding complex analytic continuations obtained  from Fact \ref{ccc}.  Assume that $f \asymp g \prec x^2$.  Let also $U, V \subseteq \CC$ be standard power domains and set $s:= f_{U,\im}$ and $t:= g_{V,\im}$ so that $U = U_s$ and $V = U_t$, and assume that $\gg(U) \subseteq V$.  Then there exist $a,b,c>0$ such that $\ff(U \cap H(c)) \subseteq U_{m_a \circ t \circ m_b}$.
\end{lemma}

\begin{proof}
	Let $c>0$ be such that $f \sim cg$.  If $f = cg$, then the conclusion follows directly from Lemmas \ref{sqd_facts}(3) and \ref{sqd_lemma}(3), so we assume $f \ne cg$.  We distinguish two cases:
	
	\subsection*{Case 1: $cg -f$ is bounded.}  Then there exists $d \in \RR$ such that $\epsilon:= cg-d-f \in \D$ or $-\epsilon \in \D$.  Hence either $1/\epsilon$ or $-1/\epsilon$ is infinitely increasing; by Fact \ref{0-eh_cor}(3) both are of exponential height at most 0.  By Fact \ref{ccc}(2a), it follows that $|\epsilon(z)| \to 0$ as $|z| \to \infty$ in $U_s$.
	
	Since $f = cg - d - \epsilon$, it follows that, for large enough $z \in U_s$ with $\im z > 0$, we have
	$$
	\im \ff(z) = c \im \gg(z) - \im \epsilon(z) \le c \im \gg(z) + 1
	$$
	and
	$$
	\re \ff(z) = c \re \gg(z) - d - \re \epsilon(z) \ge  c \re \gg(z) - (d + 1).
	$$
	Since $|\gg(z)| \to \infty$ as $|z| \to \infty$ by Fact \ref{ccc}(2a) and since, for $z \in U_t$, we have $\re z \to +\infty$ as $|z| \to \infty$, it follows that $\re \ff(z) \ge \frac c2 \re\gg(z) \to +\infty$ as $|z| \to \infty$ in $U_s$.
	A similar inequality as the first one above holds if $\im z < 0$, so that $$|\im \ff(z)| \le c|\im \gg(z)| + 1$$ for sufficiently large $z \in U_s$.
	By assumption $|\im \gg(z)| < t(\re\gg(z))$ for $z \in U_s$.  Since $t$ is infinitely increasing, we get for sufficiently large $z \in U_s$ that 
	$$|\im \ff(z)| \le c|\im \gg(z)| + 1 \le 2c\cdot t(\re\gg(z)) \le 2c\cdot t\left(\frac2c\re\ff(z)\right),$$ so we can take $a:= 2c$ and $b:= \frac2c$.
	
	\subsection*{Case 2: $cg - f$ is unbounded.} We assume here that $cg-f > 0$; the case $f-cg > 0$ is handled similarly and left to the reader.  By Fact \ref{ccc}(2b), since $\eh(cg-f) \le 0$ by fact \ref{0-eh_cor}(3) and $cg-f \prec f \prec x^2$, we have for large enough $z \in U_s$ with $\im z > 0$ that $$\im (\ff(z)) > 0 \quad\text{and}\quad \im(c\gg(z)-\ff(z)) > 0.$$  Since $\im (c\gg(z)) = \im\ff(z) + \im(c\gg(z)-\ff(z))$, it follows that $$\im\ff(z) < c \im \gg(z)$$ for such $z$.  Arguing similarly for $z \in U_s$ with $\im z < 0$, we conclude that $$|\im\ff(z)| \le c|\im \gg(z)|$$ for sufficiently large $z \in U_s$ (the inequality not being strict to account for the real $z$).  
	
	On the other hand, the germ $1/\delta$ belongs to $\I$ as well, where $\delta:= \frac{cg-f}{cg}$ has exponential height at most 0 by Fact \ref{0-eh_cor}(3).  Since $1/\delta \le cg \prec x^2$, by Fact \ref{ccc}(2b) again, we have $\im\delta(z) < 0$ for sufficiently large $z \in U_s$ with $\im z > 0$.  Since $f = cg(1-\delta)$ we have, for sufficiently large $z \in U_s$ with $\im z > 0$, 
	$$
	\re\ff(z) = c\re \gg(z)(1-\re\delta(z)) - c\im \gg(z) \im\delta(z) > \frac c2 \re \gg(z).
	$$
	We now conclude as in Case 1 with $a:=c$ and $b:= \frac2c$.
\end{proof}

\begin{cor}
	\label{admissible_prop}
	Let $f \in \I$ be such that $\eh(f) \le 0$ and $f \preceq x$, and let $\ff$ be a complex analytic continuation of $f$ obtained from Fact \ref{ccc}.  Then for every standard power domain $V$, there exist $a \ge 0$ and standard power domains $U_1, U_2$ such that $\ff(U_1 \cap H(a)) \subseteq V$ and $\ff(V \cap H(a)) \subseteq U_2$.  In particular, the restriction of $\exp \circ (-\ff)$ to any standard power domain is bounded.
\end{cor}

\begin{proof}
	If $f \asymp x$ the conclusion follows from Lemmas \ref{comparison_mapping} and \ref{sqd_lemma}(2), so we assume that $f \prec x$.  If $f \prec x^\epsilon$ for some $\epsilon \in (0,1)$, then the conclusion follows from Corollary \ref{comparison_mapping_cor}.  The remaining case is covered by Lemma \ref{comparison_mapping_4}.
\end{proof}

Next, dominance is preserved by all pairs of germs covered by the previous corollary:

\begin{prop}
	\label{monomial_comparison}
	Let $f,g \in \H$ be such that $\eh(f), \eh(g) \le 0$ and $f, g \preceq x$, and let $a \ge 0$ and $\ff,\gg:H(a) \into \CC$ be corresponding complex analytic continuations obtained  from Fact \ref{ccc}.  Then, for any standard power domain $U \subseteq H(a)$, we have $$\exp \circ (-f) \preceq \exp \circ (-g) \quad\text{if and only if}\quad \bexp \circ (-\ff) \preceq_{U} \bexp \circ (-\gg).$$
\end{prop}

\begin{proof}
	By hypothesis and Fact \ref{0-eh_cor}(3), we have $\eh(f-g) \le 0$.  If $\exp \circ (-f) \asymp \exp \circ (-g)$, then $\lim_{x \to +\infty} (f-g)(x) = c \in \RR$ so, by Corollary \ref{unit_mapping}, we have $\lim_{|z| \to \infty} (\ff-\gg)(z) = c$ in $H(a)$, so that $\bexp \circ \ff \asymp_{H(a)} \bexp \circ \gg$.  So we assume from now on that $\exp\circ(-f) \prec \exp\circ(-g)$; then $f-g \in \I$.  
	
	Fix a standard power domain $U \subseteq H(a)$.  Since  $f-g \preceq x$ by assumption, it follows from Corollary \ref{admissible_prop} that $(\ff-\gg)(U) \subseteq V$ for some standard power domain $V$.  Since $V$ is a standard power domain, we have for $w \in V$ that $\re w \to +\infty$ as $|w| \to \infty$.  Since $\ff-\gg:U \into (\ff-\gg)(U)$ is a biholomorphism, we have for $z \in U$ that $|(\ff-\gg)(z)| \to \infty$ as $|z| \to \infty$.  Hence $|\bexp \circ (\ff-\gg)(z)| = \exp(\re(\ff-\gg)(z)) \to \infty$ as $|z| \to \infty$ in $U$, which shows that $\bexp \circ (-\ff) \prec_U \bexp \circ (-\gg)$.
\end{proof}

We set $$\H_{\le 0}^{x}:= \set{ f \in \H_{\le 0}:\ f \preceq x}$$ and $$\M_0^{x}:= \set{\exp \circ (-f):\ f \in \H_{\le 0}^{x}}.$$  Recall also that, for $h \in \left(\H^{>0}\right)^k$, we denote by $\langle h \rangle^\times$ the multiplicative $\RR$-vector space generated by $h$.

\begin{cor}
	\label{strong_expl}
	The set $\M_0^{x}$ is a multiplicative $\RR$-vector subspace of $\H^{>0}$, and if $m \in \M_0^{x}$ is bounded, then $m$ has a complex analytic continuation $\mm:H(a) \into \CC$, for some $a>0$, such that $\mm$ is bounded on every standard power domain $U \subseteq H(a)$.
\end{cor}

\begin{proof}
	That $\M_0^{x}$ is a multiplicative $\RR$-vector subspace of $\H^{>0}$ follows from the observation that $\H_{\le 0}^{x}$ is an additive $\RR$-vector subspace of $\H_{\le 0}$.  Let $m \in \M_0^{x}$ be bounded and $f \in \H_{\le 0}^{x}$ be such that $m = \exp\circ(-f)$.  Let also $a>0$ and $\ff:H(a) \into \CC$ be a complex analytic continuation of $f$ obtained from Fact \ref{ccc}, and set $\mm:= \bexp \circ (-\ff)$.  Since $0 \in \H_{\le 0}^{x}$ and $m \preceq 1 = \exp\circ(-1)$, it follows from Proposition \ref{monomial_comparison} that $\mm \preceq_U \mathbf{1}$, as required.
\end{proof}

\begin{proof}[Proof of the Admissibility Theorem]
	Let $k \in \NN$ and $f_0, \dots, f_k \in \I$ be simple such that $f_0 > \cdots > f_k$, and assume that the $f_i$ have pairwise distinct archimedean classes.  Fix $0 \le i < j \le k$.  Since $f_i > f_j$ and both are simple, we have $\lambda_i:= \level(f_i) \ge \lambda_j:= \level(f_j) = \eh(f_j)$, so that $\lambda_j - \lambda_i \le 0$.  So by Fact \ref{eh_application} with $f_j$ and $f_i$ in place of $f$ and $g$ there, we get $f_j \circ f_i^{-1} \in \H_{\le 0}^{x}$.  It follows from Corollary \ref{admissible_prop} that $(\dagger)$ holds for $f:= (f_0, \dots, f_k)$.  Moreover, by Corollary \ref{strong_expl}, every bounded $m \in M_f$ has a bounded complex analytic continuation on $U \cap H(a)$, for some $a \ge 0$ and every standard power domain $U$.  Since the $f_i$ have pairwise distinct archimedean classes, if follows that $f$ is admissible.
\end{proof}

\section{$M$-generalized power series}
\label{M-gps_section}

We work in the setting of Section \ref{results_section}; in particular, we let $M$ be a multiplicative $\RR$-vector subspace of $\hplus$ and $K$ be a commutative ring of characteristic 0 with unit $1$.  

An \textbf{$M$-generalized power series over $K$} is a series of the form $F = G(m_0, \dots, m_k)$, where $G \in \As{K}{X_0^*, \dots, X_k^*}$ has natural support and $m_0, \dots, m_k \in M$ are small; in this situation, we refer to the $m_i$ as the \textbf{generating monomials} of $F$.  

\begin{rmk}
	In general, in a representation of an $M$-generalized power series $F$ of the form $G(m_0, \dots, m_k)$ as above, the generalized series $G$ is not uniquely determined by $m_0, \dots, m_k$.  However, if $M = \la$ and we choose the latter carefully (see Lemma \ref{strong_basis_lemma} below), $G$ is indeed uniquely determined by $m_0, \dots, m_k$.
\end{rmk}

\begin{lemma}
	\label{gps_well_lemma}
	Let $G \in \As{K}{X_0^*, \dots, X_k^*}$, $m_0, \dots, m_k \in M$ be small and $n \in M$.  Then the set $$S_n^G:= \set{\alpha \in \supp(G):\ m^\alpha = n}$$ is finite.
\end{lemma}

\begin{proof}
	We show that $\Pi_{X_i}\left(S_n^G\right)$ is finite for each $i$, where $\Pi_{X_i}:\RR^{k+1} \into \RR$ denotes the projection on the $i$th coordinate.  Assume, for a contradiction, that $\Pi_{X_{i_0}}\left(S_n^G\right)$ is infinite, where $i_0 \in \{0, \dots, k\}$; without loss of generality, we may assume that $i_0 = k$.  
	
	Since $\Pi_{X_k}\left(S_n^G\right) \subseteq \Pi_{X_k}(\supp(G))$ and the latter is well ordered, there is a strictly increasing sequence $0 \le \alpha_k^0 < \alpha_k^1 < \cdots$ of elements of $\Pi_{X_k}\left(S_n^G\right)$.  For $l \in \NN$, choose $\alpha_0^l, \dots, \alpha_{k-1}^l \ge 0$ such that $$\alpha^l:= (\alpha_0^l, \dots, \alpha_k^l) \in S_n^G.$$  We claim that there exists $i(l) \in \{0, \dots, k-1\}$ such that $\alpha_{i(l)}^{l+1} < \alpha_{i(l)}^l$: otherwise, we have $$m^{\alpha^{l+1}} \le m_0^{\alpha_0^l} \cdots m_{k-1}^{\alpha_{k-1}^l} m_k^{\alpha_k^{l+1}} < m^{\alpha^l} = n,$$ a contradiction.  
	By the claim, there exist $i \in \{0, \dots, k-1\}$ and a sequence $l_p \in \NN$, for $p \in \NN$, such that $\alpha_i^{l_0} > \alpha_i^{l_1} > \cdots$, which contradicts that $\Pi_{X_i}\left(S_n^G\right)$ is well ordered.
\end{proof}

\begin{lemma}
	\label{M-gen_is_gen}
	Every $M$-generalized power series is a generalized series in $\Gs{K}{M}$.
\end{lemma}

\begin{proof}
	Let $F = G(m_0, \dots, m_k)$, where $G = \sum_{\alpha \in [0,+\infty)^{1+k}} a_\alpha X^\alpha \in \As{K}{X_0^*, \dots, X_k^*}$ has natural support and $m_0, \dots, m_k \in M$ are small.  By Lemma \ref{gps_well_lemma}, there are unique $F_n \in K$, for $n \in M$, such that $F = \sum_{n \in M} F_n n$; we need to show that the set $$T:= \set{n \in M:\ F_n \ne 0}$$ is anti-well ordered.  So let $T' \subseteq T$ be nonempty and define $$S^G_{T'}:= \bigcup_{n \in T'} S^G_n,$$ where $S^G_n$ is defined as in Lemma \ref{gps_well_lemma}.  Then $\Pi_{X_i}\left(S^G_{T'}\right)$ is well ordered  for each $i$, so there exist a nonzero $l \in \NN$ and $\alpha^1, \dots, \alpha^l \in S^G_{T'}$ such that, for every $\alpha \in S^G_{T'}$, one of the $X^{\alpha^j}$ divides $X^\alpha$.  Hence $\max\set{m^{\alpha^1}, \dots, m^{\alpha^l}}$ is the maximal element of $T'$, which proves the lemma.
\end{proof}

We denote by $\As{K}{M}^{\text{ps}}$ the subring of $\Gs{K}{M}$ of all $M$-generalized power series.  

\begin{lemma}
\label{leading_monomial_lemma}
	Assume that $M$ is an asymptotic scale, and let $F \in \As{K}{M}^{\text{ps}}$ be nonzero.  Then there are a nonzero $a \in K$ and an $E \in \As{K}{M}^{\text{ps}}$ such that $\lm(E)$ is small and $F = a \lm(F) (1-E).$
\end{lemma}

\begin{proof}
	Let $G \in \As{K}{X^*}$ have natural support and $m_0, \dots, m_k \in M$ be small such that $F = G(m)$.  Set $p:= \lm(G(m))$; changing $G$ if necessary, by Lemma \ref{gps_well_lemma}, we may assume that there is a unique minimal element $\beta \in \supp(G)$ such that $p = m^\beta$.  Since $M$ is an asymptotic scale this implies, in particular, that if $\alpha \in \supp(G)$ is minimal and $\alpha \ne \beta$, we have that $m^\alpha/p$ is small.
	
	Let now $S$ be the finite set of minimal elements of $\supp(G)$ (see \cite[Lemma 4.2(1)]{Dries:1998xr}), so that $\beta \in S$.  Then there are $U_\alpha \in \As{K}{X^*}$ with natural support (see \cite[Concluding Remark 2]{Dries:1998xr}) and nonzero $a_\alpha \in K$, for $\alpha \in S$, such that $U_\alpha(0) = 1$ for each $\alpha$ and $$G = \sum_{\alpha \in S} a_\alpha X^\alpha U_\alpha = a_\beta X^\beta U_\beta + \sum_{\beta \ne \alpha \in S} a_\alpha X^\alpha U_\alpha.$$  Let $Y = (Y_\alpha:\ \beta \ne \alpha \in S)$ be a tuple of new indeterminates and set $$U(X,Y):= U_\beta + \sum_{\beta \ne \alpha \in S} \frac{a_\alpha}{a_\beta} Y_\alpha U_\alpha,$$ a generalized power series with natural support satisfying $U(0) = 1$.  Hence $$H(X,Y):= 1 - U(X,Y)$$ is a generalized power series with natural support satisfying $\ord(H) > 0$.  On the other hand, since $m^\alpha/p$ is small for $\beta \ne \alpha \in S$, the series $E:= H(m, m')$ is an $M$-generalized power series, where $m' := (m^\alpha/p:\ \beta \ne \alpha \in S)$, and we have $$G(m) = a_\beta p H(m,m') = a_\beta p (1 - E).$$  Since $\ord(H) > 0$, the leading monomial of $E$ is small, so the lemma is proved. 
\end{proof}

\subsection{Order type of support} \label{order-type_subsection}
Recall that $m,n \in M$ are \textbf{comparable} if there exist $a,b > 0$ such that $n^{a} < m < n^{b}$.  It is straightforward to see that the comparability relation is an equivalence relation on $M$.  If $C \subseteq M$ consists of pairwise comparable germs, we say that $n \in M$ is \textbf{comparable to $C$} if $n$ is comparable to any germ in $C$.

\begin{lemma}
	\label{pairwise_comparable_lemma}
	Let $F \in \As{K}{M}^{\text{ps}}$ have generating monomials $m_0, \dots,$ $m_k$, and assume that the $m_i$ are pairwise comparable.  Then $F$ has $\langle m_0, \dots, m_k \rangle^\times$-natural support.
\end{lemma}

\begin{proof}
	Let $G \in \As{K}{X_0^*, \dots, X_k^*}$ with natural support be such that $F= G(m_0, \dots, m_k)$.  Let $p \in \langle m_0, \dots, m_k \rangle^\times$; we claim that $n \le p$ for all but finitely many $n \in \supp(F)$.  Since the $m_i$ are pairwise comparable, each $\langle m_i \rangle^\times$ is coinitial in $\langle m_0, \dots, m_k \rangle^\times$; in particular, there exist $r_1, \dots, r_k \in \RR$ such that $m_i^r < p$ for each $r > r_i$, for $i=0, \dots, k$.  Since the box $B:= [0,r_0] \times \cdots \times [0,r_k]$ is compact, the set $B \cap \supp(G)$ is finite.  But by definition of $B$, if $n \in \supp(F)$ satisfies $n \ge p$, then $n = m^\alpha$ for some $\alpha \in B$, which proves the claim.
\end{proof}

\begin{lemma}
\label{cc_order-type}
	Let $F \in \As{K}{M}^{\text{ps}}$ have generating monomials $m_0 < \dots < m_k$.  Let $l \in \NN$ and assume that $\{m_0, \dots, m_k\}$ has $l+1$ distinct comparability classes $C_0 < C_1 < \cdots < C_l$.  Then $F$ has support of reverse-order type at most $\omega^{l+1}$.
\end{lemma}

\begin{proof}
	Let $G(X) \in \Ps{R}{X^*}$ with natural support be such that $F= G(m_0, \dots, m_k)$.  
	If $l = 0$, the corollary follows from Lemma \ref{pairwise_comparable_lemma}, so we assume $l>0$ and the corollary holds for lower values of $l$.  
	
	We let $j \in \{0, \dots, k\}$ be such that $m_i$ is comparable to $C_0$ if and only if $i \le j$, and we set $M':= \langle m_{j+1}, \dots, m_k\rangle^\times$ and $M'':= \langle m_0, \dots, m_j\rangle^\times$.  Identifying $\Gs{K}{M''M'}$ with a subring of $\Gs{\Gs{K}{M''}}{M'}$, we write $$F = \tilde G(m_{j+1}, \dots, m_k),$$ where $\tilde G \in \As{\As{K}{M''}^{\text{ps}}}{(X_{j+1}^*, \dots, X_k^*)}$ has natural support.  Applying the inductive hypothesis with $\As{K}{M''}^{\text{ps}}$ in place of $K$, we obtain that the support of $\tilde G(m_{j+1}, \dots, m_k)$ has reverse-order type at most $\omega^{l}$.  On the other hand, by Lemma \ref{pairwise_comparable_lemma}, every coefficient of $\tilde G$ has $M''$-natural support, that is, support of reverse-order type at most $\omega$; it follows that the support of $F$ has reverse-order type at most $\omega^{l+1}$.
\end{proof}

\subsection{Infinite sums} \label{inf_sum_subsection}
Let $F_\nu \in \Gs{K}{M}$ for $\nu \in \NN$.  Recall that, if the sequence $(\lm(F_\nu):\ \nu \in \NN)$ is decreasing and coinitial in $M$, then the infinite sum $\sum_\nu F_\nu$ defines a series in $\Gs{K}{M}$.  
In this general context, $M$-naturality is preserved:

\begin{lemma}
\label{sum_of_natural_support}
Assume that each $F_\nu$ has $M$-natural support and the sequence $(\lm(F_\nu):\ \nu \in \NN)$ is decreasing and coinitial in $M$.  Then $\sum_\nu F_\nu$ has $M$-natural support.
\end{lemma}

\begin{proof}
	Let $n \in M$; it suffices to show that the set $$\left(\bigcup_\nu \supp(F_\nu)\right) \cap (n,+\infty)$$ is finite.  By hypothesis, we have $\lm(F_\nu) \to 0$ in $M$ as $\nu \to \infty$.  So we choose $N \in \NN$ such that $\lm(F_\nu) \le n$ for $\nu > N$; then $$\left(\bigcup_\nu \supp(F_\nu)\right) \cap (n,+\infty) = \left(\bigcup_{\nu=0}^N \supp(F_\nu)\right) \cap (n,+\infty),$$ and the latter is finite since each $F_\nu$ has natural support.	
\end{proof}

Assume now that each $F_\nu$ is an $M$-generalized power series of the form $G_\nu(m_0, \dots, m_k)$ with generating monomials $m_0 < \dots < m_k$ and $G_\nu \in \As{K}{X^*}$ of natural support.  In this situation, we want a general criterion for when $\sum_\nu F_\nu$ defines a series in $\As{K}{M}^{\text{ps}}$.  The fact that $M$ may contain more than one comparability class plays a role here: let $l \in \NN$ be such that $\{m_0, \dots, m_k\}$ has $l+1$ comparability classes $C_0 < \cdots < C_l$, and let $i_0 := 0 < i_1 < \cdots < i_l \le k < i_{l+1} := k+1$ be such that $$C_j = \{m_{i_j}, \dots, m_{i_{j+1}-1}\}$$ for $0 \le j \le l$, and set $$M_j:= \langle C_j \rangle^\times.$$
We denote by $\Pi_j:\RR^{k+1} \into \RR^{i_{j+1}-i_j}$ the projection on the coordinates $(x_{i_j}, \dots, x_{i_{j+1}-1})$, and  we let $\Pi_{M_j}:\langle m_0, \dots, m_k \rangle^\times \into M_j$ be the map defined by $$\Pi_{M_j}(m^\alpha) := (m_{i_j}, \dots, m_{i_{j+1}-1})^{\Pi_j(\alpha)}.$$  Note that, since each $G_\nu$ has natural support, each set $\Pi_{M_j}(\supp(F_\nu))$ is an anti-well ordered subset of $M_j$; so we set $$m_{\nu,j}:= \max\Pi_{M_j}(\supp(F_\nu)).$$

\begin{lemma}
	\label{inf_sum_lemma}
	Assume that $\bigcup_\nu \supp(G_\nu)$ is natural and there exists $j \in \{0, \dots, l\}$ such that the sequence $\left(m_{\nu,j}:\ \nu \in \NN\right)$ is coinitial in $M_j$.  Then
	$\sum_\nu F_\nu$ is an $M$-generalized power series with generating monomials $m_0, \dots, m_k$.
\end{lemma}

\begin{rmk}
	Note that the assumption that $\bigcup_\nu \supp(G_\nu)$ be natural is necessary: if $M = \langle \frac1{\exp}, \frac1x \rangle^\times$ and $F_\nu = \exp^{-\nu} x^{-1/\nu}$, then $M_0 = \langle \frac1{\exp} \rangle^\times$ and $m_{\nu,0} = \exp^{-\nu} \to 0$ in $M_0$ as $\nu \to \infty$, while $G_\nu = X_1^\nu X_2^{1/\nu}$ implies that $\bigcup_\nu \supp(G_\nu)$ is not natural.
\end{rmk}

\begin{proof}
	We claim that $\ord(G_\nu) \to \infty$ as $\nu \to \infty$.  Assuming this claim, it follows from \cite[Paragraph 4.6]{Dries:1998xr} that $G:= \sum_\nu G_\nu$ belongs to $\As{K}{X^*}$.  Since $G$ has natural support by hypothesis, it follows that $\sum_\nu F_\nu = F:= G(m_0, \dots, m_k) \in \As{K}{M}^{\text{ps}}$, as required.
	
	To see the claim, let $j$ be as in the hypothesis of the lemma.  For each $\nu$, choose $\beta^\nu \in \supp(G_\nu)$ such that $\ord(G_\nu) = |\beta^\nu|$.  Since $$(m_{i_j}, \dots, m_{i_{j+1}-1})^{\Pi_{j}(\beta^\nu)} = \Pi_{M_j}\left(m^{\beta^\nu}\right) \le m_{\nu,j} \to 0$$ in $M_j$ as $\nu \to \infty$, we must have $\left|\Pi_{j}(\beta^\nu)\right| \to \infty$ as $\nu \to \infty$.  Since $\beta^\nu$ only has nonnegative coordinates, it follows that $\ord(G_\nu) = |\beta^\nu| \to \infty$ as $\nu \to \infty$, as claimed.
\end{proof}

\subsection{Composition with power series} \label{composition_subsection}
For $A \subseteq [0,\infty)$ and $\nu \in \NN$, we set $$+^\nu\ A := \set{a_1 + \cdots + a_\nu:\ a_i \in A}$$ and $$B(A):= \bigcup_{\nu \ge 1} \left(+^\nu\ A\right).$$

\begin{lemma}
\label{natural_set_lemma}
	Let $A \subseteq [0,\infty)$ be natural.  Then $B(A)$ is natural.
\end{lemma}

\begin{proof}
	Assume first that $0 \notin A$, and set $a:= \min A > 0$.  Then $\min\left(+^\nu\ A\right) = \nu a$ for all $\nu \ge 1$.  So, given $b > 0$, choose $N \ge \frac ba$; then $[0,b] \cap B(A) = [0,b] \cap \bigcup_{\nu=1}^{N-1} \left(+^\nu\ A\right)$, which is finite.  This proves the lemma in this case.
		
	In general, we have $A \subseteq \{0\} \cup A'$, where $A' \subseteq (0,\infty)$ is natural.  By the previous case, it now suffices to show that $B(A) = \{0\} \cup B(A')$, which follows if we  show that 
	\begin{itemize} 
		\item[$(\ast)_\nu$]  $\ \ +^\nu\ A \subseteq \{0\} \cup B(A')$
	\end{itemize}  
	 for each $\nu \ge 1$.
	We show $(\ast)_\nu$ by induction on $\nu$: $(\ast)_1$ follows by choice of $A'$. So assume $\nu > 1$ and $(\ast)_\eta$ holds for $\eta < \nu$.  Then 
	\begin{align*}
		+^\nu\ A &= A + \left(+^{\nu-1} A\right) \\
		&\subseteq \left(\{0\} \cup B(A')\right) + \left(\{0\} \cup B(A')\right) \\
		&= \{0\} \cup B(A') \cup B(A') \cup (B(A') + B(A')) \\
		& \subseteq \{0\} \cup B(A'),
	\end{align*}
	where the second line follows from the inductive hypothesis and the fourth line follows from the observation that $B(A') + B(A') \subseteq B(A')$.  This proves $(\ast)_\nu$ and hence the lemma.	
\end{proof}

Let $F \in \As{K}{M}^{\text{ps}}$ with generating monomials $m_0,$ $\dots, m_k$ be such that $\lm(F)$ is small, and let $P = \sum_{\nu \in \NN} a_\nu T^\nu \in \As{K}{T}$ be a power series in the single indeterminate $T$.  Let also $G \in \As{K}{X^*}$ have natural support such that $F = G(m)$.  Since $\lm(F)$ is small, we have $\ord(G) > 0$, so by \cite[Paragraph 4.6]{Dries:1998xr}, the sum $P \circ G:= \sum_\nu a_\nu G_\nu$ belongs to $\As{K}{X^*}$.  We therefore define $$P \circ F := (P \circ G)\left(m\right).$$
This composition does not depend on the particular series $G$ chosen (see for instance \cite{MR1848569}), and it is associative in the following sense: if $P \in \As{K}{T}$ has positive order and $Q \in \As{K}{T}$, then $Q \circ (P \circ G) = (Q \circ P) \circ G$.  We will therefore simply write $Q \circ P \circ G$ for these compositions.

\begin{prop}
	\label{gps_comp_with_convergent}
	Let $F \in \As{K}{M}^{\text{ps}}$ with generating monomials $m_0,$ $\dots, m_k$ be such that $\lm(F)$ is small, and let $P \in \As{K}{T}$.  
	\begin{enumerate}
		\item $P \circ F$  is an $M$-generalized power series with generating monomials $m_0, \dots, m_k$.
		\item Assume in addition that $\supp(F)$ is $M$-natural and the sequence $(\lm(F)^\nu:\ \nu \in \NN)$ is coinitial in $M$.  Then $P \circ F$ has $M$-natural support.
	\end{enumerate}
\end{prop}

\begin{proof}
	(1) Let $G \in \As{K}{X^*}$ with natural support and small $m_0, \dots, m_k \in M$ be such that $F = G(m)$; we need to show that $P \circ G$ has natural support.  Since $\Pi_{X_i}(\supp(P \circ G)) \subseteq \bigcup_\nu \Pi_{X_i}(\supp(G^\nu))$, and since 
	\begin{equation*}
	\Pi_{X_i}(\supp(G^\nu)) = +^\nu\ \Pi_{X_i}(\supp(G)),
	\end{equation*}
	Lemma \ref{natural_set_lemma} implies that $\supp(P \circ G)$ is natural, as required.
	
	(2) Arguing along the lines of Lemma \ref{natural_set_lemma}, we see that each $F^\nu$ has $M$-natural support (we leave the details to the reader).  Since $\lm(F^\nu) = \lm(F)^\nu$, part (2) follows from part (1) and Lemma \ref{sum_of_natural_support}.
\end{proof}

\subsection{$M$-generalized Laurent series} \label{Laurent_subsection}
An \textbf{$M$-generalized Laurent series} is a series of the form $n F$, where $F$ is an $M$-generalized power series with generating monomials $m_0, \dots, m_k$ and $n$ is a (possibly large) element of $\langle m_0, \dots, m_k \rangle^\times$.  We denote by $\Gs{K}{M}^{\text{ls}}$ the subset of $\Gs{K}{M}$ of all $M$-generalized Laurent series.

\begin{lemma}
\label{Laurent_ring}
	$\Gs{K}{M}^{\text{ls}}$ is subring of $\Gs{K}{M}$.
\end{lemma}

\begin{proof}
	It is easy to see that the set $\Gs{K}{M}^{\text{ls}}$ is closed under multiplication.  As to closure under addition, let $G_1, G_2 \in \As{K}{X^*}$ have natural support with $X = (X_0, \dots, X_k)$, let $m_0, \dots, m_k \in M$ be small and $n_1, n_2 \in \langle m_0, \dots, m_k\rangle^\times$.  Let $\alpha^1, \alpha^2 \in \RR^{k+1}$ be such that $n_i = m^{\alpha_i}$ for $i=1,2$, where $m:= (m_0, \dots, m_k)$, and define $$\beta = (\beta_0, \dots, \beta_k) := \left(\min\set{\alpha_0^1, \alpha_0^2}, \dots, \min\set{\alpha_k^1,\alpha_k^2}\right)$$ and $n:= m^\beta$.  Note that $$H_i := X^{\alpha^i-\beta} G_i(X)$$ belongs to $\As{K}{X^*}$ and has natural support, for $i=1,2$.  Then $$n_1G_1(m) + n_2G_2(m) = n(H_1(m) + H_2(m)),$$ which shows that $\Gs{K}{M}^{\text{ls}}$ is closed under addition.
\end{proof}

\begin{prop}
\label{Laurent_field_prop}
	Assume that $K$ is a field and $M$ is an asymptotic scale.
	\begin{enumerate}
		\item Let $F \in \Gs{K}{M}^{\text{ls}}$ be nonzero.  Then there are a nonzero $a \in K$ and an $E \in \As{K}{M}^{\text{ls}}$ such that $\lm(E)$ is small and $F = a \lm(F) (1-E).$
		\item $\Gs{K}{M}^{\text{ls}}$ is the fraction field of $\As{K}{M}^{\text{ps}}$ in $\Gs{K}{M}$.
	\end{enumerate}
\end{prop}

\begin{proof}
	(1) Let $G \in \As{K}{X^*}$ have natural support, $m_0, \dots, m_k \in M$ be small and $n \in \langle m_0, \dots, m_k \rangle^\times$ be such that $F = nG(m)$.  Part(1) now follows from Lemma \ref{leading_monomial_lemma}, since $\lm(F) = n\lm(G(m))$.  
	
	(2) Let $F \in \Gs{K}{M}^{\text{ls}}$ be nonzero, and let $a$ and $E$ be for $F$ as in part (1).  We get from Proposition \ref{gps_comp_with_convergent}(1) that $P \circ E$ is an $M$-generalized power series with generating monomials $m$ and $m'$, where $P \in \As{K}{T}$ is the geometric series $P(T) = \sum_\nu T^\nu$.  It follows from the binomial formula that $$\frac1{F} = \frac1{a \lm(F)} (P \circ E),$$ which is an $M$-generalized Laurent series.
\end{proof}

\subsection{$M$-series}
\label{M-series}
Finally, the series we eventually obtain in the Construction Theorem all belong to $\Gs{\RR}{\la}$, but they are more general than $\la$-generalized Laurent series.  Once again, this has to do with the possibility that $M$ may contain more than one comparability class:  let $l \in \NN$ be such that $\{m_0, \dots, m_k\}$ has $l+1$ comparability classes $C_0 < \cdots < C_l$, and let $i_0 := 0 < i_1 < \cdots < i_l \le k < i_{l+1} := k+1$ be such that $$C_j = \{m_{i_j}, \dots, m_{i_{j+1}-1}\}$$ for $0 \le j \le l$, and set $$M_1:= \langle m_0, \dots, m_{i_1-1} \rangle^{\times} \quad\text{and}\quad M^1:= \langle m_{i_1}, \dots, m_{k} \rangle^\times.$$  Note that $$m_{j} < M^1 < \frac1{m_{j}}$$ for $j=0, \dots, i_1-1$.

\begin{df}
\label{M-series_df}
	The set $\Gs{K}{M}^{\text{s}}$ of all \textbf{$M$-series over $K$} is defined by induction on $l$:
	\begin{equation*}
		\Gs{K}{M}^{\text{s}}:= \begin{cases}
		\Gs{K}{M}^{\text{ls}} &\text{if } l=0, \\
		\Gs{\Gs{K}{M^1}^{\text{s}}}{M_1}^{\text{ls}} &\text{if } l>0.
		\end{cases}
	\end{equation*}
\end{df}

\begin{expl}
\label{M-series_expl}
	The series $\sum_{\nu \in \NN} x^\nu \exp^{-\nu}$ is an $\langle x^{-1},\exp^{-1} \rangle^{\times}$-series over $\RR$ that is not an $\langle x^{-1},\exp^{-1} \rangle^{\times}$-generalized Laurent series over $\RR$.
\end{expl}

\begin{lemma}
\label{M-series_lemma}
	If $K$ is a field and $M$ is an asymptotic scale, then $\Gs{K}{M}^{\text{s}}$ is a subfield of $\Gs{K}{M}$ containing $\Gs{K}{M}^{\text{ls}}$.
\end{lemma}

\begin{proof}
	By induction on $l$; the case $l=0$ follows from Proposition \ref{Laurent_field_prop}(3), so we assume $l>0$ and the lemma holds for lower values of $l$.  But then $\Gs{K}{M^1}^{\text{s}}$ is a subfield of $\Gs{K}{M^1}$ containing $\Gs{K}{M^1}^{\text{ls}}$, by the inductive hypothesis, so $\Gs{K}{M}^{\text{s}}$ is a subfield of $\Gs{\Gs{K}{M^1}^{\text{s}}}{M_1}$ containing $\Gs{\Gs{K}{M^1}^{\text{ls}}}{M_1}^{\text{ls}}$, and the latter contains $\Gs{K}{M}^{\text{ls}}$.  Therefore, it remains to show that $\Gs{K}{M}^{\text{s}}$ is contained in $\Gs{K}{M}$.
	
	To see this, let $F \in \Gs{K}{M}^{\text{s}}$; as an element of $\Gs{\Gs{K}{M^1}^{\text{s}}}{M_1}$, we can write $F = \sum_{m \in M_1} F_m m$, with $F_m \in \Gs{K}{M^1}^{\text{s}}$, and by Lemma \ref{pairwise_comparable_lemma}, the set $$\supp_{M_1}(F):= \set{m \in M_1:\ F_m \ne 0}$$ is $M_1$-natural.  Set $$S_1:= \set{m \in \supp_{M_1}(F):\ m < M^1 < \frac1m}$$ and $S_0:= \supp_{M_1}(F) \setminus S_1$; then $S_0$ is finite, so the series $F_0:= \sum_{m \in S_0} F_m m$ belongs to $\Gs{K}{M}$.  On the other hand, since every $m' \in M^1$ has strictly slower comparability class than every $m \in S_1$, the series $F_1:= \sum_{m \in S_1} F_m m$ also belongs to $\Gs{K}{M}$, so the claim follows.
\end{proof}

\subsection{Convergent $M$-generalized Laurent series} \label{convergent_subsection}
In this subsection, we assume that $K \subseteq \C$.  Recall \cite[Section 5]{Dries:1998xr} that if $\alpha$ is a countable ordinal and $r_\beta \ge 0$ for $\beta < \alpha$, then the sum $\sum_{\beta < \alpha} r_\beta$ (respectively, the product $\prod_{\beta < \alpha} r_\beta$) \textbf{converges to} $a \in \RR$ if, for every $\epsilon > 0$, there exists a finite set $I_\epsilon \subseteq \alpha$ such that $\big|\sum_{\beta \in I} r_\beta - a\big| < \epsilon$  (respectively, $\big|\prod_{\beta \in I} r_\beta - a\big| < \epsilon$) for every finite set $I \subseteq \alpha$ containing $I_\epsilon$.  It follows from the continuity and the morphism property of $\exp$ that $\sum_{\beta < \alpha} r_\beta$ converges to $a$ if and only if $\prod_{\beta < \alpha} \exp(r_\beta)$ converges to $\exp(a)$.
Correspondingly, a generalized power series $$G = \sum_{\alpha \in [0,+\infty)^{k+1}} a_\alpha X^\alpha \in \As{K}{X^*}$$ \textbf{converges (absolutely)} if $\sum |a_\alpha(x)||x^\alpha|$ converges uniformly for all sufficiently small $x \in [0,\infty)^{k+1}$.  In this situation, there exist $\epsilon > 0$ and a unique continuous function $g:[0,\epsilon)^{k+1} \into \RR$ such that $G(x)$ converges to $g(x)$ for $x \in [0,\epsilon)^{k+1}$.  Therefore, if $\supp(G)$ is natural, $m_0, \dots, m_k \in M$ are small and $n \in \langle m_0, \dots, m_k \rangle^\times$, then the $M$-generalized Laurent series $F := n \cdot G(m_0, \dots, m_k)$ converges to the germ $$\S_{K,M}(F):= n \cdot g(m_0, \dots, m_k) \in \C.$$
We denote by $\Gs{K}{M}^{\text{conv}}$ the set of all convergent $M$-generalized Laurent series; arguing as in the proof of Proposition \ref{Laurent_field_prop}(1), we see that $\Gs{K}{M}^{\text{conv}}$ is an $\RR$-subalgebra of $\Gs{K}{M}$.  Moreover, the map $\S_{K,M}:\Gs{K}{M}^{\text{conv}} \into \C$ is an $\RR$-algebra homomorphism; it follows that the set $$\C(K,M)^{\text{conv}}:= \set{\S_{K,M}(F):\ F \in \Gs{K}{M}^{\text{conv}}}$$ is an $\RR$-subalgebra of $\C$.

The next example is central to this paper, and it provides a way to make precise the notion of ``convergent LE-series'' hinted at in \cite[Remark 6.31]{MR1848569}.  Recall from \cite[Section 2]{Kaiser:2017aa} that $\E$ is the set of all germs in $\H$ defined by $\Lanexp$-terms (that is, without the use of $\log$) and $\M$ is the set of all germs in $\la$ defined by $\Lanexp$-terms.  We then have, in particular, that $$\H = \bigcup_{k \in \NN} \E \circ \log_k \quad\text{and}\quad \la = \bigcup_{k \in \NN} \M \circ \log_k.$$

\begin{expl}
	\label{H-expl}
	We have $\H \subseteq \C(\RR,\la)^{\text{conv}}$: to see this, let $h \in \H$, and let $k \in \NN$ and $f \in \E$ be such that $h = f \circ \log_k$.  By definition of $\E$, we have $f = \S_{\RR,\M}(F)$ for some $F \in \Gs{\RR}{\M}^{\text{conv}}$, so $h = \S_{\RR,\M\circ\log_k}(F \circ \log_k)$.  Since $\M \circ \log_k \subseteq \la$, we have $F \circ \log_k \in \Gs{\RR}{\la}^{\text{conv}}$, so that $\H \subseteq \C(\RR,\la)^{\text{conv}}$.  
	
	Note in fact that $\H$ consists, by \cite[Section 2]{Kaiser:2017aa}, of all germs $\S_{\RR,\la}(F)$ such that $F = G(m_0, \dots, m_k)$ for some small $m_0, \dots, m_k \in \la$ and a convergent Laurent series $G$ with support contained in $\ZZ^{k+1}$.  In particular, we have $\C(\RR,\la)^{\text{conv}} \nsubseteq \H$.
\end{expl}

\begin{df}\label{convergent_LE-series}
	In view of the previous example, we call a series $F \in \Gs{\RR}{\la}^{\text{conv}}$ a \textbf{convergent LE-series} (or \textbf{convergent transseries}) if $F = G(m_0, \dots, m_k)$ for some small $m_0, \dots, m_k \in \la$ and a convergent Laurent series $G$ with support contained in $\ZZ^{k+1}$.  Correspondingly, we refer to the germs in $\la$ as the \textbf{convergent LE-monomials} (or \textbf{convergent transmonomials}).  In accordance with \cite{MR1848569}, we denote by $\Gs{\RR}{x^{-1}}^{\text{LE, conv}}$ the set of all convergent LE-series.
\end{df}

\begin{lemma}
\label{subseries}
Let $F = \sum a_m m \in \Gs{K}{M}^{\text{conv}}$, let $A \subseteq \supp(F)$ and set $F_A:= \sum_{m \in A} a_m m$.  Then $F_A \in \Gs{K}{M}^{\text{conv}}$.  In particular, the set $\C(K,M)^{\text{conv}}$ is truncation closed.
\end{lemma}

\begin{proof}
	Let $m_0, \dots, m_k \in M$ be small, $n \in \langle m_0, \dots, m_k \rangle^\times$ and a convergent $G(X_0, \dots, X_k) = \sum_{\alpha \in [0,\infty)^{k+1}} b_\alpha X^\alpha \in \As{K}{X^*}$ with natural support be such that $F = n G(m_0, \dots, m_k)$.  Note that $a_p = \sum_{\alpha \in S^G_{p/n}} b_\alpha$, for $p \in M$, since the set $S^G_{p/n}$ is finite by Lemma \ref{gps_well_lemma}.  Setting $S_A := \bigcup_{p \in A} S^G_{p/n}$, the convergence of $G$ implies that $G_A:= \sum_{\alpha \in S_A} b_\alpha X^\alpha$ is also convergent; but $F_A = n G_A(m_0, \dots, m_k),$ as required.
\end{proof}

\begin{df}\label{str_slower_cc}
	Given nonzero $f,g \in \C$, we say that $f$ has \textbf{strictly slower} comparability class than $g$ if $f$ and $g$ are incomparable and either $g < f < 1/g$ or $1/g < f < g$.
\end{df}

We let $\C_M$ be the field of all $a \in \C$ such that $a$ is either 0 or $1/m < a < m$, for every large $m \in M$.  Since $M$ is a multiplicative $\RR$-vector space, every nonzero $a \in \C_M$ has strictly slower comparability class than every $m \in M$ satisfying $m \ne 1$; in particular, for every $a \in \C_M$ and $m \in M$ satisfying $m \ne 1$, the product $am$ is comparable to $m$.  

\begin{rmk}
	We always have $\RR \subseteq \C_M$.
\end{rmk}

From now on, we write $\S_M$ in place of $\S_{\C_M,M}$ and $\C(M)^{\text{conv}}$ in place of $\C(\C_M,M)^{\text{conv}}$.

\begin{prop}
\label{convergent_lemma}
\begin{enumerate}
\item For nonzero $F \in \Gs{\C_M}{M}^{\text{conv}}$, there is a nonzero $a \in \C_M$ such that  $\S_{M}(F) \asymp a \lm(F)$.
\item The set $\Gs{\C_M}{M}^{\text{conv}}$ is a field.
\item The map $\S_{M}:\Gs{\C_M}{M}^{\text{conv}} \into \C(M)^{\text{conv}}$ is a field isomorphism.
\end{enumerate}
\end{prop}

\begin{proof}
	Let $F \in \Gs{\C_M}{M}^{\text{conv}}$, and assume first that $\lm(F)$ is small.  We claim that in this case, $\S_{M}(F) \prec 1$: to see this, let $\epsilon > 0$ and choose a finite $I \subseteq \supp(F)$ such that $|\S_{M}(F)(x) - \S_{M}(F_I)(x)| < \epsilon$ for sufficiently large $x>0$.  Since $\lm(F)$ is small, the germ $\S_{M}(F_I)$ is a $\C_M$-linear combination of small monomials of $M$, so $|\S_{M}(F_I)| \prec 1$ by definition of $\C_M$.  Hence $|\S_{M}(F)(x)| \le 2\epsilon$ for all sufficiently large $x>0$, and since $\epsilon > 0$ was arbitrary, the claim is proved.
	
	Now assume that $F$ is nonzero.  By Proposition \ref{Laurent_field_prop}(1), there are a nonzero $a \in \C_M$ and an $E \in \As{\C_M}{M}^{\text{ls}}$ such that $\lm(E)$ is small and $$F = a \lm(F) (1-E).$$  Then $E$ is convergent, and since $\lm(E)$ is small, we have from the claim that $\S_{M}(E) \prec 1$.  This proves part (1); part (2) follows along the lines of the proof of Proposition \ref{Laurent_field_prop}(2), because $\C_M$ is a field.
	
	(3) It suffices to show that $\S_M$ is injective: let $F \in \Gs{\C_M}{M}^{\text{conv}}$ be nonzero.  Then $\lm(F) \ne 0$ and, by part (1), there is a nonzero $a \in \C_M$ such that $\S_M(F) \asymp a\lm(F)$; in particular, $\S_M(F) \ne 0$.
\end{proof}

In view of the previous proposition, we denote by $$S_M:\C(M)^{\text{conv}} \into \Gs{\C_M}{M}^{\text{conv}}$$ the compositional inverse of the summation map $\S_{M}$.

\begin{cor}
	\label{convergent_prop}
	\begin{enumerate}
		\item The triple $\left( \C(M)^{\text{conv}}, M, S_{M}\right)$ is a qaa field.
		\item $(\H,\la,S_\la)$ is a qaa field with $S_\la(\H) = \Gs{\RR}{x^{-1}}^{\text{LE, conv}}$.
	\end{enumerate}
\end{cor}

\begin{proof}
	(1) By Lemma \ref{subseries}, the field $\Gs{\C_M}{M}^{\text{conv}}$ is truncation closed.  Moreover, if $F \in \Gs{\C_M}{M}^{\text{conv}}$ and $m \in M$ then, by convergence and Proposition \ref{convergent_lemma}(1), there is a nonzero $a_m \in \C_M$ such that $$\S_{M}(F) - \S_{M}(F_m) = \S_M(F - F_m) \asymp a_m\lm(F-F_m) \prec m,$$ where $F_m$ is the truncation of $F$ at $m$.  
	
	(2) In the notation of Example \ref{H-expl}, we have $\S_{\RR,\la} = \S_\la \rest{\Gs{\RR}{\la}^{\text{conv}}}$.  Since $\Gs{\RR}{x^{-1}}^{\text{LE, conv}}$ is a truncation closed subset of $\Gs{\C_\la}{\la}^{\text{conv}}$, the claim follows from part (1).
\end{proof}

\subsection{Strong convergent $M$-generalized Laurent series}
Let  $U$ be a standard power domain, and assume that $M$ is a strong asymptotic scale.  Given $\phi,\psi:H(0) \into \CC$, recall that $$\phi \preceq_U \psi$$ if $|\phi(z)/\psi(z)|$ is bounded in $U$.  Correspondingly, we write $\phi \asymp_U \psi$ if both $\phi \preceq_U \psi$ and $\psi \preceq_U \phi$, and we write $\phi \prec_U \psi$ if $\phi \preceq_U \psi$ but $\psi \not\preceq_U \phi$.

\begin{df}\label{strong_convergence_def}
	We let $\C_M^U$ be the set of all $a \in \C_M$ that have a complex analytic continuation $\aa$ on $U$ such that, if $a \ne 0$, then $1/\mm \prec_U \aa \prec_U \mm$, for every large $m \in M$.  Correspondingly, we let $$\Gs{\C_M^U}{M}^{\text{conv}}_U$$ be the set of all $F = \sum a_m m \in \Gs{\C_M^U}{M}^{\text{conv}}$ such that the series $\sum |\aa_m(z)||\mm(z)|$ converges uniformly for all sufficiently large $z \in U$.  For $F \in \Gs{\C_M^U}{M}^{\text{conv}}_U$, we write $$\S_M^U(F):U \into \CC$$ for the complex analytic continuation of $\S_M(F)$ on $U$.
\end{df}

\begin{rmk}
	If $F \in \Gs{\C_M^U}{M}^{\text{conv}}_U$ and $A \subseteq \supp(F)$, then $F_A \in \Gs{\C_M^U}{M}^{\text{conv}}_U$.
\end{rmk}

\begin{expl}
	\label{real_strong_convergent}
	We have $\Gs{\RR}{M}^{\text{conv}} \subseteq \Gs{\C_M^U}{M}^{\text{conv}}_U$: to see this, let $F \in \Gs{\RR}{M}^{\text{conv}}$, and let $m_0, \dots, m_k \in M$ be small, a convergent $G \in \As{\RR}{X_0^*, \dots, X_k^*}$ with natural support and an $n \in \langle m_0, \dots, m_k\rangle^\times$ be such that $F = n G(m_0, \dots, m_k)$.  Then there are $\epsilon > 0$ and a holomorphic $g:D(\epsilon)^{k+1} \into \CC$ such that $G(z_0, \dots, z_k)$ converges uniformly to $g(z_0, \dots, z_k)$ for $(z_0, \dots, z_k) \in D(\epsilon)^{k+1}$, where $D(\epsilon):= \set{z \in \CC:\ |z| < \epsilon}$.  On the other hand, since $\mm_i \prec_U 1$ for each $i$, we have $|\mm_i(z)| < \epsilon$ for all sufficiently large $z \in U$ and each $i$, so that $G(\mm_0(z), \dots, \mm_k(z))$ converges uniformly to $g(\mm_0(z), \dots, \mm_k(z))$ for sufficiently large $z \in U$.  
\end{expl}

\begin{prop}
	\label{strong_convergent_lemma}
	\begin{enumerate}
		\item For nonzero $F \in \Gs{\C_M^U}{M}^{\text{conv}}_U$, there is a nonzero $a \in \C_M^U$ such that  $\S_{M}^U(F) \asymp_U \aa\mm$, where $m:= \lm(F)$.
		\item The set $\Gs{\C_M^U}{M}^{\text{conv}}_U$ is a field.
	\end{enumerate}
\end{prop}

\begin{proof}
	The proof is similar to the proof of Proposition \ref{convergent_lemma}, using the previous remark and the assumption that $M$ is a strong asymptotic scale; we leave the details to the reader.
\end{proof}

\section{Asymptotic and strong asymptotic expansions}
\label{asymptotic_section}

Below, we fix an arbitrary multiplicative $\RR$-vector subspace $M$ of $\,\H^{>0}$, and we assume that $M$ is an asymptotic scale.  

\begin{df}\label{asymptotic_expansion_2}
	Let $f \in \C$ and $F = \sum a_m m \in \Gs{\C_M}{M}$.  We say that $f$ has \textbf{asymptotic expansion} $F$ (at $+\infty$) if $\supp(F)$ is $M$-natural and
	\begin{equation*}\tag{$\ast$}
	f - \sum_{m \ge n} a_m m \prec n
	\end{equation*}
	for every $n \in M$.
\end{df}

\begin{expl}
	\label{convergent_is_asymptotic}
	Let $F \in \Gs{\C_M}{M}^{\text{conv}}$ have $M$-natural support.  Then the proof of Corollary \ref{convergent_prop}(1) shows, in particular, that $\S_{M}F$ has asymptotic expansion $F$.
\end{expl}

\begin{lemma}
	\label{asymptotic_ring_lemma}
	Let $f,g \in \C$ have asymptotic expansions $\sum a_m m$ and $\sum b_m m$ in $\Gs{\C_M}{M}$, respectively.  Then
	\begin{enumerate}
		\item for $n \in M$, the germ $fn$ has asymptotic expansion $\sum a_m (mn)$;
		\item $f+g$ has asymptotic expansion $\sum (a_m+b_m) m$;
		\item $fg$ has asymptotic expansion $\left(\sum a_m m\right) \left(\sum b_m m\right)$.
	\end{enumerate}
\end{lemma}

\begin{proof}
	(1) and (2) are straightforward from the definition.  For (3) we may assume, by (1), that $\sum a_m m$ and $\sum b_m m$ belong to $\Ps{\C_M}{M}$.  We fix $n \in M$, and we write $$\sum c_m m = \left(\sum a_m m\right) \left(\sum b_m m\right),$$ so that $c_m = \sum_{m_1m_2 = m} a_{m_1}b_{m_2}.$  Since
	\begin{align*}
	fg- \sum_{m \ge n} c_m m &= \left(f - \sum_{m \ge n} a_m m\right)g + \\&\quad+ \left(\sum_{m \ge n} a_m m\right)\left(g - \sum_{m \ge n} b_m m\right)+ \\&\quad+\left(\sum_{m \ge n} a_m m\right)\left(\sum_{m \ge n} b_m m\right) - \sum_{m \ge n} c_m m,
	\end{align*}
	and since $f$, $g$ and each $a_m m$ are bounded, it follows that the first and second of these four summands are $\prec n$.  As to the third and fourth summands, 
	\begin{align*}
	\left(\sum_{m \ge n} a_m m\right)&\left(\sum_{m \ge n} b_m m\right) - \sum_{m \ge n} c_m m \\&= \left(\sum_{m \ge n} a_m m\right)\left(\sum_{m \ge n} b_m m\right) - \sum_{m_1m_2 \ge n} a_{m_1}b_{m_2} m_1m_2 \\ &= \sum_{m_1,m_2 \ge n \atop m_1m_2\prec n} a_{m_1}b_{m_2} m_1m_2,
	\end{align*}
	which is $\prec n$, because the latter sum is finite.
\end{proof}

\begin{lemma}
	\label{unique_as_exp}
	Every $f \in \C$ has at most one asymptotic expansion in $\Gs{\C_M}{M}$. 
\end{lemma}

\begin{proof}
	Let $F = \sum a_m m \in \Gs{\C_M}{M}$ be an asymptotic expansion of $0$; by the previous lemma, it suffices to show that $F = 0$.  Assume for a contradiction that $F \ne 0$, and let $m_0 \in \M$ be the leading monomial of $F$, and let $m_1 \in M$ be the leading monomial of $F-m_0$, if it exists, or equal to $m_0^3$ otherwise, and set $m':= (m_0m_1)^{1/2} \in M$.  Since $M$ is an asymptotic scale, we have $m_0 \succ m' \succ m_1$ while, by definition of asymptotic expansion, we have $0 - a_{m_0} m_0 \prec m'$.  Therefore, $a_{m_0} \prec \frac{m'}{m_0}$, which contradicts $a_{m_0} \in \C_M$.
\end{proof}

\subsection*{Strong asymptotic expansions}

As mentioned in Section \ref{outline_section}, for the construction of our qaa algebras for simple tuples, we need to introduce asymptotic expansions on standard power domains.  
Thus, we fix a multiplicative $\RR$-vector subspace $M$ of $\H^{>0}$, and we assume that $M$ is a strong asymptotic scale.  As usual, we denote by $\mm$ the complex analytic continuation of $m \in M$ to any standard power domain.  Recall from Definition \ref{strong_full_asymptotics} that, for $m,n \in M$ and any standard power domain $U$, we have $m \prec n$ if and only if $\mm \prec_U \nn$. 

\begin{df}
	\label{asymptotic_def}
	Let $f\in \C$ and $F = \sum a_m m \in \Gs{\C_M}{M}$.  The germ $f$ has \textbf{strong asymptotic expansion $F$ (at $\infty$)} if there is a standard power domain $U$ such that
	\begin{renumerate}
		\item $F$ has $M$-natural support;
		\item $f$ has a complex analytic continuation $\ff$ on $U$;
		\item for each $n \in M$, there is a standard power domain $V \subseteq U$ such that $a_m \in \C_M^V$ for each $m \ge n$, and we have
		\begin{equation}
		\label{strong_asymptotics}\tag{$\ast_{f,n}$}
		\ff - \sum_{m \geq n} \aa_m \mm \prec_V \nn.
		\end{equation}
	\end{renumerate}
	In this situation, $U$ is called a \textbf{strong asymptotic expansion domain} of $f$.  
\end{df}	

\begin{expls}
	\label{A0-rmk}
	\begin{enumerate}
		\item Let $U$ be a standard power domain and $F = \sum a_m m \in \Gs{\C_M^U}{M}^{\text{conv}}_U$ have natural support.  Then $\S_M(F)$ has strong asymptotic expansion $F$ on $U$.
		\item Let $(\F_k,L,T_k)$ be the qaa field constructed in \cite[Corollary 26]{MR3744892}, where $k \in \NN$ and $M_k = M\left(\frac1{\log_{-1}}, \dots, \frac1{\log_{k-1}}\right)$.  Then every $f \in \F_k$ has strong asymptotic expansion $\tau_kf$ on some standard quadratic domain.
	\end{enumerate}
\end{expls}

Since strong asymptotic expansions are in particular asymptotic expansions, they are unique by Lemma \ref{unique_as_exp}.  Lemma \ref{asymptotic_ring_lemma} generalizes directly to strong asymptotic expansions:

\begin{lemma}
	\label{strong_asymptotic_ring_lemma}
	Let $f,g \in \C$ have strong asymptotic expansions $\sum a_m m$ and $\sum b_m m$ in $\Gs{\C_M}{M}$ in a standard power domain $U$, respectively.  Then
	\begin{enumerate}
		\item for $n \in M$, the germ $fn$ has strong asymptotic expansion $\sum a_m (mn)$ in $U$;
		\item $f+g$ has strong asymptotic expansion $\sum (a_m+b_m) m$ in $U$;
		\item $fg$ has strong asymptotic expansion $\left(\sum a_m m\right) \left(\sum b_m m\right)$ in $U$.
	\end{enumerate}
\end{lemma}

\begin{proof}
	Adapting the proof of Lemma \ref{asymptotic_ring_lemma} is straightforward and left to the reader.  
\end{proof}

The next criterion is useful for obtaining strong asymptotic expansions from infinite sums.

\begin{lemma}
	\label{inf_sum_of_strong}
	Let $f \in \C$ and $f_l \in \C$, for $l \in \NN$.  Assume that each $f_l$ has strong asymptotic expansion $F_l \in \Gs{\C_M}{M}$ in some standard power domain $U_l$ and that the sequence $(\lm(F_l):\ l \in \NN)$ is decreasing and coinitial in $M$, so that $\sum_l F_l$ is a series in $\Gs{\C_M}{M}$.  Assume also that $f$ has a complex analytic continuation $\ff$ on some standard power domain $U$ and there are standard power domains $V_l$ such  that $$\ff - \sum_{i=0}^l \ff_i \prec_{V_l} \ff_l \quad \text{ for each } l.$$  Then the series $\sum_l F_l$ is a strong asymptotic expansion of $f$ in $U$.
\end{lemma}

\begin{proof}
	From Lemma \ref{sum_of_natural_support} we know that $\sum_l F_l$ has $M$-natural support.  Let $n \in M$, and choose $N \in \NN$ such that $\lm(F_l) \prec n$ for all $l \ge N$.  Let $V$ be a standard power domain contained in $U \cap U_0 \cap \cdots \cap U_N \cap V_0 \cap \cdots \cap V_N$.  Then $\ff_N \prec_V \nn$ and $\ff - \sum_{i=0}^N \ff_i \prec_V \ff_N$ by hypothesis, so $$\ff - \sum_{i=0}^N \ff_i \prec_V \nn.$$  Increasing $N$ and shrinking $V$ if necessary, we may assume that $$\left(\sum_{i=0}^\infty F_i\right)_n = \sum_{i=0}^N (F_i)_{n}.$$  Therefore, with $\hh_n$ the complex analytic continuation of $\left(\sum F_i\right)_{n}$ on $V$ and $\hh_{i,n}$ the complex analytic continuation of $(F_i)_{n}$ on $V$, we get
	\begin{align*}
	\ff - \hh_n &= \ff - \sum_{i=0}^N \hh_{i,n} \\
	&= \left( \ff - \sum_{i=0}^N \ff_i \right) + \sum_{i=0}^N \left(\ff_i - \hh_{i,n}\right) \prec_V \nn,
	\end{align*}
	as required.
\end{proof}

To extend the notion of strong asymptotic expansion to series in $\Gs{\RR}{M}$, we proceed as in Definition \ref{qaa_df}:

\begin{df}
	\label{strong_qaa_algebra}
	Let $\K \subseteq \C$ be an $\RR$-algebra and $T:\K \into \Gs{\RR}{M}$ be an $\RR$-algebra homomorphism.  We say that the triple $(\K,M,T)$ is a \textbf{strong qaa algebra} if 
	\begin{renumerate}
		\item $T$ is injective;
		\item the image $T(\K)$ is truncation closed;
		\item for every $f \in \K$ and $n \in M$, there exists a standard power domain $U$ such that $f$ and $g_n:= T^{-1}(F_n)$ have complex analytic continuations $\ff$ and $\gg_n$ on $U$, respectively, that satisfy
		\begin{equation}\label{strong_asymptotic_relation}
		\ff - \gg_n \prec_U \nn.
		\end{equation}
	\end{renumerate}
	In this situation, we call $T(f)$ the \textbf{strong $\K$-asymptotic expansion} of $f$.
\end{df}

\begin{expl}
	\label{strong_convergent_is_strong_asymptotic}
	We set $\C(M)^{\text{conv}}_U := \S_M\left(\Gs{\C_M^U}{M}^{\text{conv}}_U\right)$ and denote by $S_M^U$ the restriction of $S_M$ to $\C(M)^{\text{conv}}_U$.  Then the triple $$\left(\C(M)^{\text{conv}}_U, M, S_M^U\right)$$ is a strong qaa field.  This is proved along the lines of the proof of Corollary \ref{convergent_prop}(1); we leave the details to the reader.  It follows from Example \ref{real_strong_convergent} that $\left(\C(\RR,M)^{\text{conv}}, M, S_{\RR,M}\right)$ is a strong qaa field.
\end{expl}

\begin{lemma}
	\label{strong_full_asymptotics}
	Let $M'$ be multiplicative $\RR$-vector subspace of $M$, and let $(\K,M',T)$ be a strong qaa algebra.  Then $(\K,M,T)$ is a strong qaa algebra. 
\end{lemma}

\begin{proof}
	Let $f \in \K$ and $n \in M$.  If $n \in M'$, then the asymptotic relation \eqref{strong_asymptotic_relation} holds by assumption on some standard power domain $U$, so assume $n \notin M'$.  If $n \le \supp(Tf)$, then $T^{-1}((Tf)_n) = f$, so the asymptotic relation \eqref{strong_asymptotic_relation} holds trivially on some standard power domain $U$.  So assume also that $n \not\le \supp(Tf)$ and choose the maximal $p \in \supp(Tf)$ such that $p<n$ (which exists because $\supp(Tf)$ is reverse well-ordered).  By assumption, writing $\gg_p$ and $\gg_n$ for the complex analytic continuations of $T^{-1}((Tf)_p)$ and $T^{-1}((Tf)_n)$, respectively, there is a standard power domain $U$ such that
	\begin{equation*}
	\pp \succ_U \ff - \gg_p = \ff - \gg_n - a\pp,
	\end{equation*}
	for some nonzero $a \in \RR$.  Since $\pp \prec_U \nn$ because $M$ is a strong asymptotic scale, the asymptotic relation \eqref{strong_asymptotic_relation} follows.
\end{proof}

\section{Construction: the admissible case}
\label{construction_section}

Let $f = (f_0, \dots, f_k) \in \I^{k+1}$ be admissible, and recall that $$M_f = \left\langle\exp \circ (-f_0), \dots, \exp \circ (-f_k)\right\rangle^\times.$$  To describe the construction of our qaa algebra $\left(\K_f, M_f, T_f\right)$, we introduce the following notation: for $i \in \{0, \dots, k\}$, we set 
\begin{equation*}
	f^{\langle i \rangle} := \left(x, f_{k-i+1} \circ f_{k-i}^{-1}, \dots, f_k \circ f_{k-i}^{-1}\right).
\end{equation*}
It is straightforward to see that, for $0 \le i,j \le k$ such that $i+j \le k$, we have 
\begin{equation}
\label{sequence_combinatorics}
	\left(f^{\langle i \rangle}\right)^{\langle j \rangle} = f^{\langle i+j \rangle}.
\end{equation}
Moreover, we recall that since $f$ is admissible, so is each $f^{\langle i \rangle}$; in particular, each $M_{f^{\langle i \rangle}}$ is a strong asymptotic scale.

As outlined in the introduction, we first construct the qaa fields $\left(\K_{f^{\langle i \rangle}}, M_{f^{\langle i \rangle}}, T_{f^{\langle i \rangle}}\right)$ by induction on $i = 0, \dots, k$.

To simplify notation in our description of this construction below, and since $f$ is fixed in this construction, we temporarily replace the subscript ``${f^{\langle i \rangle}}$'' by the subscript ``$i$''; thus, we write $M_i = M_{f,i}$ in place of $M_{f^{\langle i \rangle}}$.

\begin{rmk}
	We have $M_0 = M_x = \langle \exp\circ(-x) \rangle^\times$ is the space denoted by $E$ in \cite{MR3744892}.
\end{rmk}

\subsection*{The initial Ilyashenko field}
Since 
$M_0$ is a strong asymptotic scale with strong basis $\exp^{-x}$ we define $\A_0 = \A_{f,0}$ to be, as in \cite{MR3744892}, the set of all germs $g \in \C$ that have an $M_0$-generalized power series $T_0(g) = T_{f,0}(g) \in \As{\RR}{M_0}^{\text{ps}}$ as strong asymptotic expansion.  

\begin{lemma}
	\label{algebra_cor}
	$\A_0$ is an $\RR$-algebra, and the map $T_0:\A_0 \into \As{\RR}{M_0}^{\text{ps}}$ is an injective $\RR$-algebra homomorphism.  
\end{lemma}

\begin{proof}
	By Lemma \ref{strong_asymptotic_ring_lemma}, the set $\A_0$ is an $\RR$-algebra and the map $T_0:\A_0 \into \As{\RR}{M_0}^{\text{ps}}$ is an $\RR$-algebra homomorphism; its kernel is trivial by the Uniqueness Principle.
\end{proof}

\begin{cor}
	\label{closure_cor}
	The triple $\left(\A_0,M_0,T_0\right)$ is a strong qaa algebra and, for every $F \in \As{\RR}{M_0}^{\text{conv}}$ with $M_0$-natural support, the sum $\S_{M_0}(F)$ belongs to $\A_0$ and satisfies $T_0(\S_{M_0}(F)) = F$.  	
\end{cor}

\begin{proof}
	For bounded $m \in M_0$ the complex analytic continuation $\mm$ to any standard power domain is bounded, so $m$ belongs to $\A_0$ with $T_0 m = m$.  Since the support of $T_0 g$, for $g \in \A_0$, is $M_0$-natural and since $g$ is bounded, every truncation of $T_0 g$ is an $\RR$-linear combination of bounded germs in $M_0$ and therefore belongs to $\A_0$ as well.
	
	Finally, let $F \in \As{\RR}{M_0}^{\text{conv}}$ have $M_0$-natural support.  By Example  \ref{real_strong_convergent}, the series $F$ belongs to $\Gs{\C_{M_0}^U}{M_0}^{\text{conv}}_U$ for every standard power domain $U$.  So $\S_{M_0}(F)$ has strong asymptotic expansion $F$ on every standard power domain $U$ by Example \ref{A0-rmk}(1).
\end{proof}

For  $g \in \A_0$, we set $\lm(g):= \lm\left(T_0 g\right).$

\begin{prop}
	\label{construction_lemma_1}
	Let $g \in \A_0$ be nonzero and have strong asymptotic expansion domain $U$.
	\begin{enumerate}
		\item There exist unique nonzero $a \in \RR$ and small $\epsilon \in \A_0$ such that $$g = a \lm(g) (1 - \epsilon).$$
		\item Assume that $g$ is small, let $P \in \Ps{R}{T}$ be convergent and let $p$ be the real analytic germ at 0 defined by $P$.  Then $p \circ g$ belongs to $\A_0$, has strong asymptotic expansion domain $U$ and satisfies $T_0(p \circ g) = P \circ (T_0 g)$.
		\item We have $g \asymp \lm(g)$, and the germ $\frac g{a \lm(g)}$ is a unit in $\A_0$.
	\end{enumerate}
\end{prop}

\begin{proof}
	(1) By Lemma \ref{leading_monomial_lemma}, there are a nonzero $a \in \RR$ and an $E \in \As{\RR}{M_0}^{\text{ps}}$ such that $\lm(E)$ is small and $T_0(g) = a \lm(g) (1-E)$.  By Lemma \ref{strong_asymptotic_ring_lemma}, the germ $$\epsilon := \frac{g - a\lm(g)}{a\lm(g)}$$ belongs to $\A_0$, has strong asymptotic expansion domain $U$ and satisfies $T_0(\epsilon) = E$.
	
	(2) Let $\gg$ be the complex analytic continuation of $g$ on $U$ and $\pp$ be the holomorphic function defined in a neighbourhood of 0 by $P$.  By Condition $(\ast_{g,1})$ of Definition \ref{asymptotic_def} and after replacing $U$ with $H(a) \cap U$ for some sufficiently large $a>0$, the function $\pp \circ \gg$ is a bounded, complex analytic continuation of $p \circ g$ on $U$.  Moreover, say $P(X) = \sum a_\nu X^\nu \in \Ps{R}{X}$; since $P(z) - \sum_{\nu=0}^n a_\nu z^\nu = o(z^n)$ at 0 in $\CC$ by absolute convergence, it follows that $$\pp \circ \gg - \sum_{\nu=0}^n a_\nu \gg^\nu \prec_U \gg^n.$$  From Lemma \ref{strong_asymptotic_ring_lemma}, it follows that $a_n g^n \in \A_0$ has strong asymptotic expansion domain $U$ and satisfies $$T_0(a_\nu g^\nu) = a_\nu (T_0g)^\nu,$$ for each $\nu$.  Since $\lm(g)$ is small and all germs in $M_0$ are pairwise comparable, and since $\lm(g^\nu) = (\lm g)^\nu$ for each $\nu$, the sequence  $$\left(\lm\left(g^\nu\right):\ \nu \in \NN\right)$$ is coinitial in $M_0$.  Part (2) now follows from Lemma \ref{inf_sum_of_strong} and Proposition \ref{gps_comp_with_convergent}.
	
	(3) Let $a$ and $\epsilon$ be for $g$ as in part (1).  By the binomial theorem, we have $\frac1{1-\epsilon} = p \circ \epsilon$, where $p$ is the germ at 0 defined by the geometric series $P = \sum\nu T^\nu$.  Part (3) now follows from part (2).
\end{proof}

It follows from Proposition \ref{construction_lemma_1}(3) that the set $\set{g \in \A_0:\ g \prec 1}$ is a maximal ideal of $\A_0$.  So we let $\K_0 = \K_{f,0}$ be the fraction field of $\A_0$; since $\Gs{\RR}{M_0}^{\text{ls}}$ is the fraction field of $\As{\RR}{M_0}^{\text{ps}}$ in $\Gs{\RR}{M_0}$, the injective $\RR$-algebra homomorphism $T_0$ extends uniquely to a field homomorphism $T_0:\K_0 \into \Gs{\RR}{M_0}^{\text{ls}}$.

\begin{cor}
	\label{construction_cor_1}
	\begin{enumerate}
		\item Let $g \in \K_0$.  Then $g$ has strong asymptotic expansion $T_0\left(g\right)$, and there exist unique $a \in \RR$ and small $\epsilon \in \A_0$ such that $g = a \lm(T_0g) (1+\epsilon)$.
		\item The triple $\left(\K_0,M_0,T_0\right)$ is a strong qaa field extending the strong qaa field $\left(\C(\RR,M_0)^{\text{conv}}, M_0, S_{\RR,M_0}\right)$.
	\end{enumerate}
\end{cor}

In view of Corollary \ref{construction_cor_1}(1) we let, for $g \in \K_0$, $$\lm(g):= \lm(T_0g)$$ be the \textbf{leading monomial} of $g$.

\begin{proof}[Proof of Corollary \ref{construction_cor_1}]
	(1) Say $g = h_1/h_2$, for some $h_1,h_2 \in \A_0$ with $h_2 \ne 0$; we may also assume that $h_1 \ne 0$.  By Lemma \ref{construction_lemma_1}(1) there are $a_1, a_2 \in \RR\setminus\{0\}$ and small $\epsilon_1, \epsilon_2 \in \A_0$ such that $h_i = a_i \lm(h_i)(1-\epsilon_i)$.  In particular, $$g = \frac{a_1 \lm(h_1)}{a_2 \lm(h_2)}(1-\epsilon_1)\frac1{1-\epsilon_2}.$$  Part (1) now follows from Lemma \ref{construction_lemma_1}(3).
	
	(2) Note that the series in $T_0\left(\K_0\right)$ have $M_0$-natural support and each monomial in $M_0$ belongs to $\K_0$; it follows that $\left(\K_0,M_0,T_0\right)$ is a strong qaa field.  The rest follows from Example \ref{strong_convergent_is_strong_asymptotic}.
\end{proof}

\subsection*{Iteration}
Assume now that $i>0$ and we have constructed a qaa field $\left(\K_{i-1}, M_{i-1}, T_{i-1}\right)$ extending a qaa algebra $\left(\A_{i-1}, M_{i-1}, T_{i-1}\right)$ such that the following hold:
\begin{renumerate}
	\item $\K_{i-1}$ is the fraction field of $\A_{i-1}$ and $T_{i-1}(\K_{i-1}) \subseteq \Gs{\RR}{M_{i-1}}^{\text{s}}$;
	\item for each $g \in \K_{i-1}$, the quotient $\frac g{\lm(g)}$ belongs to $\A_{i-1}$ and has a bounded, complex analytic continuation to some standard power domain;
	\item $\left(\K_{i-1}, M_{i-1}, T_{i-1}\right)$ extends $\left(\C(\RR,M_{i-1})^{\text{conv}}, M_{i-1}, S_{\RR,M_{i-1}}\right)$.
\end{renumerate}

First, we set
\begin{equation*}
\K'_{i} = \K_{f,i}' := \K_{i-1} \circ \left( f_{k-i+1} \circ f_{k-i}^{-1} \right)
\end{equation*}
and
\begin{equation*}
M'_{i} = M_{f,i}' := M_{i-1} \circ \left( f_{k-i+1} \circ f_{k-i}^{-1} \right),
\end{equation*}
and we define $T'_{i} = T_{f,i}':\K'_{i} \into \Gs{\RR}{M'_i}$ by $$T'_{i}\left(g \circ \left( f_{k-i+1} \circ f_{k-i}^{-1} \right)\right) := \left(T_{i-1} g\right) \circ \left( f_{k-i+1} \circ f_{k-i}^{-1} \right).$$  Note that $M'_i$ is a divisible multiplicative subgroup of $M_i$.

\begin{cor}
	\label{field_closure_cor}
	$\left(\K'_{i},M_i,T'_{i}\right)$ is a strong qaa field extending the strong qaa field $\left(\C(\RR,M'_{i})^{\text{conv}}, M_{i}', S_{\RR,M_{i}'}\right)$ and such that $T'_{i}(\K'_{i}) \subseteq \Gs{\RR}{M'_{i}}^{\text{s}}$.
\end{cor}

\begin{proof}
	Since $f$ is admissible, the germ $\left( f_{k-i+1} \circ f_{k-i}^{-1} \right)$ has a complex analytic continuation on every standard power domain that maps standard power domains into standard power domains.  Therefore the triple $\left(\K'_{i},M'_i,T'_{i}\right)$ is a strong qaa field.  Since $M'_i$ is a divisible subgroup of $M_i$, and since $\C(\RR,M'_{i})^{\text{conv}} = \C(\RR,M_{i-1})^{\text{conv}} \circ \left( f_{k-i+1} \circ f_{k-i}^{-1} \right)$ and $$S_{\RR,M_{i}'}\left(g \circ \left( f_{k-i+1} \circ f_{k-i}^{-1} \right)\right) = \left(S_{\RR,M_{i-1}} g\right) \circ \left( f_{k-i+1} \circ f_{k-i}^{-1} \right),$$ for $g \in \C(\RR,M_{i-1})^{\text{conv}}$, the corollary follows from Lemma \ref{strong_full_asymptotics}.
\end{proof}

\begin{nrmk}
	\label{rel_growth_rmk}
	Let $g \in \K'_{i}$; by condition (ii) above, we have $g \preceq \lm(g)$.  Since  $f_j \circ f_{i}^{-1}$ has comparability class strictly slower than $x = f_{i} \circ f_{i}^{-1}$, for $j > i$, it follows that $1/m \prec g \prec m$ for every large $m \in M_0$, that is, $\K'_i \subseteq \C_{M_0}$.  Also, since $\K'_i$ is a field, we get that $g \prec m$ for some small $m \in M_0$ if and only if $g = 0$.
	
	Moreover, let $U$ be a strong asyptotic expansion domain of $g$.  Then the argument of the previous paragraph, with $\prec_U$ in place of $\prec$, shows that $g \in \C_{M_0}^U$.
\end{nrmk}

By the previous remark, we let $\A_{i} = \A_{f,i}$ be the set of all $g \in \C$ that have a strong asymptotic expansion $\tau_i(g) = \tau_{f,i}(g) \in \As{\K'_i}{M_0}^{\text{ps}}$ on some standard power domain $U$ (this expansion is then unique by Remark \ref{rel_growth_rmk} and Lemma \ref{unique_as_exp}).  

Arguing as in Lemma \ref{algebra_cor}, we see that  $\A_{i}$ is an $\RR$-algebra and the map $\tau_i:\A_{i} \into \As{\K'_{i}}{M_0}^{\text{ps}}$ is an $\RR$-algebra homomorphism.  Moreover, it follows from the Uniqueness Principle that this map is injective.  
For $g \in \A_{i}$ with $\tau_i g = \sum_{m \in M_0} a_m m$, we now define $$T_{i} g = T_{f,i} g:= \sum_{m \in M_0} (T'_{i} a_m) m \in \Gs{\RR}{M_i}^{\text{s}}.$$ 
For completeness' sake, we also set $\tau_0:= T_0$.

\begin{prop}
	\label{closure_cor_2}
	\begin{enumerate}
		\item The triple $\left(\A_{i},M_i,T_{i}\right)$ is a strong qaa algebra that  extends $\left(\K'_{i},M'_i,T'_{i}\right)$ and such that $T_i(\A_i) \subseteq \Gs{\RR}{M_i}^{\text{s}}$.
		\item Let $U$ be a standard power domain and $F \in \As{\left(\K_i'\cap \C_{M_0}^U\right)}{M_0}^{\text{conv}}_U$ with $M_0$-natural support.  Then the sum $\S_{M_0}(F)$ belongs to $\A_i$ and satisfies $\tau_i(\S_{M_0}(F)) = F$.  
	\end{enumerate}
\end{prop}

\begin{proof}
	(1) The map $\sigma_i:\As{\K'_{i}}{M_0}^{\text{ps}} \into \Gs{\RR}{M_i}$ defined by $$\sigma_i\left( \sum_{m \in M_0} a_m m \right) := \sum_{m \in M_0} (T'_{i} a_m) m$$ is an $\RR$-algebra homomorphism, and it is injective because $T'_{i}$ is injective.  Since $T_{i} = \sigma_i \circ \tau_i$, it follows that $T_{i}$ is an injective $\RR$-algebra homomorphism.
	Let now $g \in \A_i$ be such that $$T_i g = \sum_{m \in M_i} a_m m \quad\text{and}\quad \tau_i g = \sum_{m \in M_0} b_m m,$$ and let $n \in M_i$; we show there exists $h \in \A_i$ such that $T_i h = (T_i g)_n$.  Let $n_0 \in M_0$ and $n' \in M'_i$ be such that $n = n'n_0$ and
	\begin{equation*}
	\left(T_i g\right)_n = \sum_{m \in M_0,\ m > n_0} T_i'(b_{m}) m + \left(T'_i a_{n_0}\right)_{n'} n_0,
	\end{equation*}
	and let $U$ be a strong asymptotic expansion domain of $g$.
	Note that each $b_m m$ has a bounded complex analytic continuation on $U$, since $b_m = 0$ for all large $m \in M_0$ by assumption.  Since $$\sigma_i^{-1}\left(\sum_{m \in M_0,\ m > n_0} T_i'(b_{m}) m\right) = \sum_{m \in M_0,\ m > n_0} b_{m} m$$ has finite support in $\As{\K'_i}{M_0}$, it follows that $$g_1:= \sum_{m \in M_0,\ m > n_0} b_{m} m$$ belongs to $\A_i$ and satisfies $\tau_i g_1 = g_1$ and $T_i g_1 =  \sum_{\pi_0(m) > n_0} a_m m$, where $\pi_0:M_i \into M_0$ is the (multiplicative) linear projection on $M_0$.  On the other hand, by the inductive hypothesis, there exists $h_1 \in \K'_i$ such that $T'_i h_1 = \left(T'_i a_{n_0}\right)_{n'}$.  Hence $h_1 n_0 \in \A_i$ and, by definition of $T_i$, we obtain $T_i(h_1 n_0) = \left(T'_i a_{n_0}\right)_{n'} n_0$.  Therefore, we can take $h:= g_1+h_1 n_0$.
	
	Finally, after shrinking $U$ if necessary, we may assume that $U$ is also a strong asymptotic expansion domain of $h$; we now claim that $\gg-\hh \prec_U \nn$, which then proves the proposition.  By the inductive hypothesis, we have $\aa_{n_0} - \hh_1 \prec_U \nn'$; therefore, 
	\begin{equation} \label{asym_1}
	\aa_{n_0} \nn_0 - \hh_1 \nn_0 \prec_U \nn.
	\end{equation} 
	On the other hand, let $n_1:= \min\set{m \in M_0:\ m > n_0 \text{ and } a_m \ne 0}$.  Then, by hypothesis, we have 
	\begin{equation} \label{asym_2}
	\gg - \gg_1 - \aa_{n_0} \nn_0 \prec_U \sqrt{\nn_0\nn_1}.
	\end{equation}
	Since $\sqrt{\nn_0\nn_1} \prec_U \nn$, part (1) follows.
	
	Part (2) follows from Example \ref{A0-rmk}(1) and the definition of $\A_i$. 
\end{proof}

For  $g \in \A_i$, we set $\lm(g):= \lm\left(T_i g\right) \in M_i$ and $\lm_0(g):= \lm\left(\tau_i g\right) \in M_0$.

\begin{prop}
	\label{construction_lemma}
	Let $g \in \A_i$ be nonzero and have strong asymptotic expansion domain $U$.
	\begin{enumerate}
	\item There exist unique nonzero $h \in \K'_i$ and $\epsilon \in \A_i$ such that $\lm_0(\epsilon)$ is small and $$g = h \lm_0(g) (1 - \epsilon).$$
	\item Assume that $\lm_0(g)$ is small, let $P \in \Ps{R}{T}$ be convergent and let $p$ be the real analytic germ at 0 defined by $P$.  Then $p \circ g$ belongs to $\A_i$, has strong asymptotic expansion domain $U$ and satisfies $T_i(p \circ g) = P \circ (T_i g)$ and $\tau_i(p \circ g) = P \circ (\tau_i g)$.
	\item The germ $\frac g{\lm_0(g)}$ is a unit in $\A_i$.
	\end{enumerate}
\end{prop}

\begin{proof}
	(1) By Lemma \ref{leading_monomial_lemma}, there are a nonzero $h \in \K'_i$ and an $E \in \As{\K'_i}{M_0}^{\text{ps}}$ such that $\lm(E)$ is small and $\tau_i(g) = h \lm_0(g) (1-E)$.  By Lemma \ref{strong_asymptotic_ring_lemma}, the germ $$\epsilon := \frac{g - h\lm_0(g)}{h\lm_0(g)}$$ belongs to $\A_i$, has strong asymptotic expansion domain $U$ and satisfies $\tau_i(\epsilon) = E$.  
	
	(2) Let $\gg$ be the complex analytic continuation of $g$ on $U$ and $\pp$ be the holomorphic function defined in a neighbourhood of 0 by $P$.  By Condition $(\ast_{g,1})$ of Definition \ref{asymptotic_def} and after replacing $U$ with $H(a) \cap U$ for some sufficiently large $a>0$, the function $\pp \circ \gg$ is a bounded, complex analytic continuation of $p \circ g$ on $U$.  Moreover, say $P(X) = \sum a_\nu X^\nu \in \Ps{R}{X}$; since $P(z) - \sum_{\nu=0}^n a_\nu z^\nu = o(z^n)$ at 0 in $\CC$ by absolute convergence, it follows that $$\pp \circ \gg - \sum_{\nu=0}^n a_\nu \gg^\nu \prec_U \gg^n.$$  From Lemma \ref{strong_asymptotic_ring_lemma}, it follows that $a_n g^n \in \A_i$ has strong asymptotic expansion domain $U$ and satisfies $$\tau_i(a_\nu g^\nu) = a_\nu (\tau_i g)^\nu,$$ for each $\nu$.  Since $\lm_0(g)$ is small and all germs in $M_0$ are pairwise comparable, and since $\lm_0(g^\nu) = (\lm_0 g)^\nu$ for each $\nu$, the sequence  $$\left(\lm_0\left(g^\nu\right):\ \nu \in \NN\right)$$ is coinitial in $M_0$.  Part (2) now follows from Lemma \ref{inf_sum_of_strong} and Proposition \ref{gps_comp_with_convergent}, except for the statement $T_i(P \circ g) = P \circ T_i(g)$.  However, since for each bounded $m \in M_0$ there exists $N_m \in \NN$ such that  $$(P \circ \tau_i(g))_{m} = \sum_{n=0}^{N_m} a_n \left(\tau_i(g)^n\right)_{m},$$ it follows that $\sigma_i(P \circ \tau_i(g)) = P \circ \sigma_i(\tau_i(g))$.  
	
	(3) Let $h$ and $\epsilon$ be for $g$ as in part (1).  By the binomial theorem, we have $\frac1{1-\epsilon} = p \circ \epsilon$, where $p$ is the germ at 0 defined by the geometric series $P = \sum\nu T^\nu$.  Part (3) now follows from part (2).
\end{proof}

As in the construction of $\K_0$, we now let $\K_{i} = \K_{f,i}$ be the fraction field of $\A_{i}$ and extend $\tau_i$ and $T_{i}$ correspondingly.

\begin{cor}
	\label{construction_cor}
	\begin{enumerate}
		\item Let $g \in \K_i$.  Then $g$ has strong asymptotic expansion $\tau_i\left(g\right)$, and there exist unique $h \in \K'_i$ and $\epsilon \in \A_i$ such that $\lm_0(\epsilon)$ is small and $$g = h \lm_0(\tau_i g)(1+\epsilon).$$  In particular, $g \in \A_i$ if and only if $\,\lm_0(\tau_ig)$ is bounded.
		\item The triple $\left(\K_i,M_i,T_i\right)$ is a strong qaa field  satisfying $T_i(\A_i) \subseteq \Gs{\RR}{M_i}^{\text{s}}$ and extending $\left(\C(\RR,M_i)^{\text{conv}}, M_i, S_{\RR,M_i}\right)$.
	\end{enumerate}
\end{cor}

In view of Corollary \ref{construction_cor}(1) we set, for $g \in \K_i$, $$\lm_0(g):= \lm_0(\tau_i g).$$

\begin{proof}[Proof of Corollary \ref{construction_cor}]
	(1) Say $g = h'/h$, for some $h',h \in \A_i$ with $h \ne 0$.  By Proposition \ref{construction_lemma}(3), the quotient $h/\lm_0(h)$ is a unit in $\A_i$, so part (1) follows from Lemma \ref{strong_asymptotic_ring_lemma}.
	
	(2) The map $T_i$ is injective, because the restriction of $T_i$ to $\A_i$ is.  Also, by part (1), each $g \in \K_i$ is of the form $g = \lm_0(g) h$ with $h \in \A_i$.  Since $\left(\A_i,M_i,T_i\right)$ is a strong qaa algebra, it follows that $\left(\K_i,M_i,T_i\right)$ is a strong qaa field.  Moreover, let $F \in \As{\RR}{M_i}^{\text{conv}}$; considering $F$ as a series in $\As{\Gs{\RR}{M_i'}}{M_0}$, we write $F = \sum_{m \in M_0} F_m m$.  By Corollary \ref{field_closure_cor}, each $\S_{M_i'}(F_m)$ belongs to $\K_i'$ and satisfies $T_i'(\S_{M_i'}(F_m)) = F_m$.  Since $\S_{M_i'}(F_m)$ converges uniformly on every standard power domain $U$, it also belongs to $\C_{M_i'}^U$ for every $U$.  Hence $$\S_{M_i}(F) = \S_{M_0}\left(\sum_{m \in M_0} \left(\S_{M_i'}(F_m) m\right)\right)$$ belongs to $\K_i$, by Proposition \ref{closure_cor_2}(2), and satisfies $$\tau_i\left(\S_{M_i}(F)\right) = \sum_{m \in M_0} \S_{M_i'}(F_m) m;$$ in particular, we have $T_i\left(\S_{M_i}(F)\right) = F$, as required.
\end{proof}

\subsection*{The final step}
Finally, having constructed the qaa field $\left(\K_k,M_k,T_k\right)$, we set $$\K_f := \K_k \circ f_0$$ and, for $g \in \K_f$, $$T_f(g):= T_k\left(g \circ f_0^{-1}\right) \circ f_0.$$  Arguing as in the proof of Corollary \ref{field_closure_cor}, we obtain:

\begin{cor}
	\label{final_construction_cor}
	$\left(\K_f,M_f,T_f\right)$ is a strong qaa field satisfying $T_f(\K_f) \subseteq \Gs{\RR}{M_f}^{\text{s}}$ and extending $\left(\C(\RR,M_f)^{\text{conv}}, M_f, S_{\RR,M_f}\right)$. \qed
\end{cor}

\begin{nrmk}
\label{sub_tuple_rmk}
	Let $f' = (f'_0, \dots, f'_{k-1}) := (f_1, \dots, f_k)$.  Then $$(f')^{\langle k-1 \rangle} = f^{\langle k-1 \rangle};$$ since $\K_{f^{\langle k \rangle}}' = \K_{f^{\langle k-1 \rangle}} \circ \left(f_1 \circ f_0^{-1}\right) \subseteq \K_{f^{\langle k \rangle}}$ by definition, it follows that $$\K_{f'} = \K_{f^{\langle k-1 \rangle}} \circ f_1 \subseteq \K_{f^{\langle k \rangle}} \circ f_0 = \K_f.$$
\end{nrmk}

It remains to show that the previous remark holds with arbitrary subtuples of $f$ in place of $f'$.  Let $l \le k$ and $\iota:\{0, \dots, l\} \into \{0, \dots, k\}$ be strictly increasing, and set $f_\iota:= (f_{\iota(0)}, \dots, f_{\iota(l)})$.  Then $f_\iota$ is admissible and we have:

\begin{prop}
\label{sub_tuple_prop}
	$(\K_f,M_f,T_f)$ extends $(\K_{f_\iota}, M_{f_\iota}, T_{f_\iota})$.
\end{prop}

\begin{proof}
	The proposition is trivial if $k=0$, so we assume $k>0$ and the proposition holds for lower values of $k$.  If $\iota(0) > 0$, then $f_\iota$ is also a subtuple of $f':= (f_1, \dots, f_k)$, so the proposition follows from the inductive hypothesis and Remark \ref{sub_tuple_rmk} in this case.  So we also assume $\iota(0) = 0$ and set $$f'_\iota:= (f_{\iota(1)}, \dots, f_{\iota(l)}),$$ a subtuple of $f'$.  By the inductive hypothesis, $\left(\K_{f'}, M_{f'}, T_{f'}\right)$ extends $\left(\K_{f_\iota'}, M_{f_\iota'}, T_{f_\iota'}\right)$, that is, $\left(\K_{f^{\langle k \rangle}}', M_{f^{\langle k \rangle}}', T_{f^{\langle k \rangle}}'\right)$ extends $\left(\K_{f_\iota^{\langle l \rangle}}', M_{f_\iota^{\langle l \rangle}}', T_{f_\iota^{\langle l \rangle}}'\right)$.
	Thus, every strong asymptotic expansion in $\Gs{\K_{f_\iota^{\langle l \rangle}}'}{M_x}$ is a strong asymptotic expansion in $\Gs{\K_{f^{\langle k \rangle}}'}{M_x}$, so that $\left(\K_{f^{\langle k \rangle}}, M_{f^{\langle k \rangle}}, T_{f^{\langle k \rangle}}\right)$ extends $\left(\K_{f_\iota^{\langle l \rangle}}, M_{f_\iota^{\langle l \rangle}}, T_{f_\iota^{\langle l \rangle}}\right)$, which proves the proposition.
\end{proof}

\section{Construction: the general case}
\label{general_constr_section}

Let $k \in \NN$ and $f = (f_0, \dots, f_k) \in \U^{k+1}$ be such that each $f_i$ is infinitely increasing and $f_0 > \cdots > f_k$; we aim to construct $(\K_f, M_f, T_f)$ in a canonical way based on the construction for the admissible case.

\begin{lemma}
\label{strong_basis_lemma}
	Let $g$ be an admissible subtuple of $f$.  Then there is an admissible tuple $\tilde g$ of infinitely increasing germs in $\U$ such that $g \subseteq \tilde g$ and $M_{\tilde g} = M_f$.
\end{lemma}

\begin{proof}
	By the Admissibility Corollary, it suffices to find a basis $\tilde g = (\tilde g_0, \dots, \tilde g_k)$ of the additive $\RR$-vector space $\langle f \rangle \subseteq \U$ such that $\tilde g \supseteq g$ and the archimedean classes of the $\tilde g_i$ are pairwise distinct.  We do this by induction on $k$; if $k=0$, there is nothing to do, so we assume $k>0$ and the claim holds for lower values of $k$.
	
	If $f_i \asymp f_0$, then $\lm(f_i) = \lm(f_0)$, because each archimedean class of $\H$ is represented by exactly one principal monomial in $\la$.  We now set $$\tilde g_0:= \begin{cases}   f_0 &\text{if } g_0 \not\asymp f_0,\\ g_0 &\text{if } g_0 \asymp f_0.\end{cases}$$  Either way, we have $\tilde g_0 \in \U$ and, for each $i$ such that $f_i \asymp \tilde g_0$, there exists a nonzero $c_i \in \RR$ such that $f_i':= f_i - c_i\tilde g_0 \prec \tilde g_0$.  Note that $f_i' \in \U$ as well; thus, after replacing each $f_i \asymp \tilde g_0$ if necessary by a corresponding $f_i' \in \U$, we may assume that $ 0 < f_i \prec \tilde g_0$ for $i>0$.  Note that this process, in particular, does not change any of the $g_j$.  Now set $f':= (f_1, \dots, f_k)$; by the inductive hypothesis, there is a basis $\tilde g' = (\tilde g_1, \dots, \tilde g_k)$ of $\langle f' \rangle$ such that $\tilde g' \supseteq g\setminus\{\tilde g_0\}$ and the archimedean classes of the $\tilde g_i$ are pairwise distinct.  Now take $\tilde g:= (\tilde g_0, \tilde g_1, \dots, \tilde g_k)$.
\end{proof}

The lemma suggests that we set $(\K_f,M_f,T_f):= (\K_{\tilde g}, M_{\tilde g}, T_{\tilde g})$, for some admissible tuple $\tilde g \in \U^{k+1}$ such that $M_{\tilde g} = M_f$.  Before we can do so, we need to verify that this definition does not depend on the particular $\tilde g$ chosen.  So we let $g^1 = (g^1_0, \dots, g^1_k) \in \U^{k+1}$ and $g^2 = (g^2_0, \dots, g^2_{k'}) \in \U^{k'+1}$ be two admissible tuples satisfying $M_{g^1} = M_{g^2} = M_f$.  Since both $g^1$ and $g^2$ are bases of $\langle f \rangle$, we must have $k' = k$.

\begin{lemma}
\label{arch_classes_lemma}
	We have $g^1_i \asymp g^2_i$ for each $i$.
\end{lemma}

\begin{proof}
	Let $i \in \{0, \dots, k\}$, and let $a_i \in \RR$ be such that $g^2_i = a_{i,0} g^1_0 + \cdots + a_{i,k} g^1_k$.  Since the archimedean classes of the $g^1_i$ are pairwise distinct and $g^1_0 > \cdots > g^1_k$, we have $g^2_i \asymp g^1_{j_i}$, where $j_i := \min\{j: a_{i,j} \ne 0\}$.  We now prove the claim by induction on $i$: since $g^2_0$ represents the maximal archimedean class represented by $M_f$, it follows that $g^2_0 \asymp g^1_0$.
	
	Assume now $i > 0$ and $g^2_j \asymp g^1_j$ for $j=0, \dots, i-1$.  Then $g^2_j \prec g^1_{i-1}$ for $j=i, \dots, k$.  Therefore we have, in particular, that $j_i \ge i$, and since $g^2_i$ represents the $(i+1)$st largest archimedean class represented by $M_f$, it follows that $j_i = i$, that is, $g^2_i \asymp g^1_i$, as claimed.
\end{proof}

For $i=0, \dots, k$, consider the subspace $$H_i:= \set{m \in \langle f \rangle:\ -g^1_{k-i} \preceq m \preceq g^1_{k-i}};$$  Lemma \ref{arch_classes_lemma} implies that $H_i$ only depends on $\langle f \rangle$, not on the particular admissible tuple $g^1$.  Moreover, since $\langle g^1_0, \dots, g^1_{k-i-1} \rangle \cap H_i = \emptyset$ by definition of $H_i$ and $g^1_{k-i}, \dots, g^1_k \in H_i$, we have  
\begin{equation}
	\langle f \rangle = \langle g^1_0, \dots, g^1_{k-i-1} \rangle \oplus H_i,
\end{equation} 
that is, 
\begin{equation}
	H_i = \langle g^1_{k-i}, \dots, g^1_k \rangle.
\end{equation}  
The same argument with $g^2$ in place of $g^1$ shows that $H_i = \langle g^2_{k-i}, \dots, g^2_k \rangle$ as well; so we obtain the following:

\begin{lemma}
\label{sub_arch_classes_lemma}
	We have $\langle g^2_{k-i}, \dots, g^2_k \rangle = \langle g^1_{k-i}, \dots, g^1_k \rangle$ for each $i$. \qed
\end{lemma}

\begin{cor}
\label{iterate_classes_cor}
	We have $\langle (g^2)^{\langle i \rangle} \rangle = \langle (g^1)^{\langle i \rangle} \rangle \circ \left(g^1_{k-i} \circ (g^2_{k-i})^{-1}\right)$ for each $i$.
\end{cor}

\begin{proof}
	Since $\langle (g^j)^{\langle i \rangle} \rangle = \langle g^j_{k-i}, \dots, g^j_k \rangle \circ (g^j_{k-i})^{-1}$ by definition, for $j=1,2$, the corollary follows from Lemma \ref{sub_arch_classes_lemma}.
\end{proof}

\begin{nrmk}
\label{the_powers_case}
	Let $c \in \C$ and $r>0$.  Since any strong asymptotic expansion of $c$ in $\Gs{\RR}{\langle \exp^{-x} \rangle^\times}$ is also a strong asymptotic expansion of $c$ in $\Gs{\RR}{\langle \exp^{-rx} \rangle^\times}$, for $r > 0$, we have $(\K_{x}, M_x, T_x) = (\K_{rx}, M_{rx}, T_{rx})$.
\end{nrmk}

\begin{prop}
\label{constr_indep}
	We have $(\K_{g^1}, M_{g^1}, T_{g^1}) = (\K_{g^2}, M_{g^2}, T_{g^2})$.
\end{prop}

\begin{proof}
	We proceed by induction on $k$. 
	
	\subsection*{Case $k=0$:} then $g^2_0 = r g^1_0$ for some $r>0$, so that $(g^2_0)^{-1} = (g^1_0)^{-1}/r$.  Hence, by Remark \ref{the_powers_case}, 
	\begin{align*}
	(\K_{g^2}, M_{g^2}, T_{g^2}) &= (\K_x,M_x,T_x) \circ g^2_0 \\ &= (\K_x,M_x,T_x) \circ (r g^1_0) \\ &= (\K_{rx},M_{rx},T_{rx}) \circ g^1_0 \\ &= (\K_x,M_x,T_x) \circ g^1_0 \\ &= (\K_{g^1}, M_{g^1}, T_{g^1}) ,
	\end{align*}
	as claimed.
	
	\subsection*{Case $k>0$:} Write $(g^1)':= (g^1_1, \dots, g^1_k)$ and $(g^2)':= (g^2_1, \dots, g^2_k)$, and assume that $(\K_{(g^1)'}, M_{(g^1)'}, T_{(g^1)'}) = (\K_{(g^2)'}, M_{(g^2)'}, T_{(g^2)'})$.  Let $h \in \K_{g^1}$; then $h \circ (g^1_0)^{-1} \in \K_{(g^1)^{\langle k \rangle}}$ by construction and, say,  $$\tau_{(g^1)^{\langle k \rangle}}\left(h \circ (g^1_0)^{-1}\right) = \sum_{r \in \RR} a_r \exp(-rx).$$  Therefore, for each $r$, we have $$a_r \in \K_{(g^1)^{\langle k \rangle}}' = \K_{(g^1)^{\langle k-1 \rangle}} \circ \left(g^1_1 \circ (g^1_0)^{-1}\right) = \K_{(g^1)'} \circ  (g^1_0)^{-1}.$$  It follows from the inductive hypothesis that $a_r \circ g^1_0 \in \K_{(g^2)'} = \K_{(g^2)^{\langle k-1 \rangle}} \circ g^2_1$, so that $$a_r \circ \left(g^1_0 \circ (g^2_0)^{-1}\right) \in \K_{(g^2)^{\langle k \rangle}}'.$$  
	
	Now recall that both $g^1_0$ and $g^2_0$ are simple and, by Lemma \ref{arch_classes_lemma}, we have $\level(g^1_0) = \level(g^2_0)$.  From Application 1.3 of \cite{Kaiser:2017aa}, we get that $\eh\left(g^1_0 \circ (g^2_0)^{-1}\right) \le 0$ so, by Corollary \ref{admissible_prop}, the germ $g^1_0 \circ (g^2_0)^{-1}$ has a complex analytic continuation on a right half-plane $H(a)$, with $a>0$, that maps standard power domains into standard power domains.  Therefore, the fact that $h \circ (g^1_0)^{-1}$ has strong asymptotic expansion $\sum a_r \exp^{-rx}$ implies that $h \circ (g^2_0)^{-1} = \left(h \circ (g^1_0)^{-1}\right) \circ \left(g^1_0 \circ (g^2_0)^{-1}\right)$ has strong asymptotic expansion $$\sum \left(a_r \circ \left(g^1_0 \circ (g^2_0)^{-1}\right)\right) \cdot \exp\left(-r\left(g^1_0 \circ (g^2_0)^{-1}\right)\right).$$
	
	Next, note that $x \in \langle (g^1)^{\langle k \rangle}\rangle$ so, from Corollary \ref{iterate_classes_cor} with $i = k$, we get $g^1_0 \circ (g^2_0)^{-1} \in \langle (g^2)^{\langle k \rangle}\rangle$.  Since $g^1_0 \circ (g^2_0)^{-1} \asymp x$ and $g^2_i \circ (g^2_0)^{-1} \prec x$ for $i > 0$, it follows that there are a nonzero $s \in \RR$ and an $\varphi \in \langle (g^2)' \rangle \circ (g^2_0)^{-1}$ such that $$g^1_0 \circ (g^2_0)^{-1} = sx + \varphi.$$ Note, in particular, that $\exp(-\varphi) \in M'_{(g^2)^{\langle k \rangle}}$ by definition of the latter, so that $\exp(-\varphi) \in \K'_{(g^2)^{\langle k \rangle}}$.  Therefore, setting $$b_r:= \left(a_r \circ \left(g^1_0 \circ (g^2_0)^{-1}\right)\right) \cdot \exp(-r\varphi) \in \K_{(g^2)^{\langle k \rangle}}',$$ we get 
	$$\sum \left(a_r \circ \left(g^1_0 \circ (g^2_0)^{-1}\right)\right) \cdot \exp\left(-r\left(g^1_0 \circ (g^2_0)^{-1}\right)\right) = \sum b_r \exp(-rx).$$  Moreover, since $\varphi \prec x$, it follows that the latter is also a strong asymptotic expansion of $h \circ (g^2_0)^{-1}$, showing that $h \circ (g^2_0)^{-1} \in \K_{(g^2)^{\langle k \rangle}}$ satisfying $$T_{(g^2)^{\langle k \rangle}}\left(h \circ (g^2_0)^{-1}\right) = T_{(g^1)^{\langle k \rangle}}\left(h \circ (g^1_0)^{-1}\right) \circ \left(g^1_0 \circ (g^2_0)^{-1}\right).$$  Therefore, $h \in \K_{g^2}$ and satisfies $T_{g^1}(h) = T_{g^2}(h)$.  Exchanging the roles of $g^1$ and $g^2$ in the argument above, we obtain the proposition.
\end{proof}

We can now give the

\begin{proof}[Proof of the Construction Theorem]
	Let $h$ be a finite tuple of small germs in $\la$.
	For part (1), we define $$(\K_h,\langle h \rangle^\times,T_h):= (\K_f,M_f,T_f)$$ for any admissible tuple $f \in \U^{k+1}$ such that $M_f = \langle h \rangle^\times$.  By Proposition \ref{constr_indep}, this definition only depends on $h$, not on the particular $f$ used for the construction.
	
	For part (2), let $h'$ be a subtuple of $h$.  By Proposition \ref{sub_tuple_prop}, it suffices to show that there exist $0 \le l \le k \in \NN$, an admissible tuple $f \in \U^{k+1}$ and a strictly increasing $\iota:\{0, \dots, l\} \into \{0, \dots, k\}$ such that $\langle h \rangle^\times = M_f$ and $\langle h' \rangle^\times = M_{f_\iota}$, where $f_\iota:= (f_{\iota(0)}, \dots, f_{\iota(l)})$.  To do this, we first obtain, by Lemma \ref{strong_basis_lemma}, an $l \in \NN$ and an admissible tuple $f' \in \U^{l+1}$ such that $\langle h' \rangle^\times = M_{f'}$.  Next, we choose a finite tuple $f''$ of germs in $\U$ such that $M_{f' \cup f''} = \langle h \rangle^\times$; this is possible, because $\langle h \rangle^\times$ is finite dimensional.  Now again by Lemma \ref{strong_basis_lemma}, there is an admissible tuple $f$ of germs in $\U$ extending $f'$ such that $M_f = \langle h \rangle^\times$.
\end{proof}

\section{Proof of the Closure Theorem}
\label{closure_section}

Since $(\K,\la,T)$ is the direct limit of all $(\K_h, \langle h \rangle^\times, T_h)$, for finite tuples $h$ of small germs in $\la$ with respect to the inclusion ordering, part (1) of the Closure Theorem follows.  For part (2), note that $\log_k \in \la$ for $k \in \ZZ$.  For part (3), let $\varphi \in \H$.  By \cite[Section 2]{Kaiser:2017aa}, there is a finite tuple $h$ of small germs in $\la$ such that $\varphi \in \C(\RR,\langle h \rangle^\times)^{\text{conv}}$.
By Lemma \ref{strong_basis_lemma}, there is an admissible tuple $f$ of infinitely increasing germs in $\U$ such that $\langle h \rangle^\times = M_f$, so that $\varphi \in \C(\RR,M_f)^{\text{conv}}$.  Hence $\varphi \in \K_f$ by Corollary \ref{final_construction_cor}, and part (3) now follows from part (1).

It remains to prove part (4) of the Closure Theorem.

\subsection*{Closure under differentiation} \label{closure-differentiation}

For closure under differentiation, it suffices to show the following: let $f = (f_0, \dots, f_k) \in \U^{k+1}$ be admissible with $f_0 > \cdots > f_k$.  Then \textit{there exists an admissible $g = (g_0, \dots, g_l) \in \U^{l+1}$ such that $f \subseteq g$ and, for every $h \in \K_f$, we have $h' \in \K_g$.}  To see how we obtain this, let $i \in \{1, \dots, k\}$ and consider the inductive step of constructing $\A_i$ from $\K_{i-1}$: say $h \in \A_i$ has strong asymptotic expansion $\tau_i(h) = \sum_{r \ge 0} h_r \exp^{-r}$, with $$h_r = g_r \circ \left(f_{k-i+1} \circ f_{k-i}^{-1}\right)$$ and $g_r \in \K_{i-1}$.  Then $h'$ has the expected strong asymptotic expansion, as the next two lemmas show; the first of these is a version of L'H\^opital's rule for holomorphic maps on standard power domains.

\begin{lemma}
	\label{derivative_o}
	Let $1 > \epsilon' \ge \epsilon > 0$ and $C,D>0$, and let $\phi:U^\epsilon_C \into \CC$ be holomorphic.  Assume that either $\epsilon' > \epsilon$, or that $\epsilon' = \epsilon$ and $D > C$.
	\begin{enumerate}
		\item Let $r \in \RR$ be such that $\phi \prec_{U^\epsilon_C} \bexp^{-r}$.  Then $\phi' \prec_{U^{\epsilon'}_D} \bexp^{-r}$.
		\item If $\phi$ is bounded in $U^\epsilon_C$, then $\phi'$ is bounded in $U^{\epsilon'}_D$.
	\end{enumerate}
\end{lemma}

\begin{proof}
	By Lemma \ref{sqd_facts}(1,2), there is $R>0$ such that $D(z,2) \subseteq U^\epsilon_C$ for every $z \in U^{\epsilon'}_D$ with $|z| > R$.  The proof now proceeds exactly like the proof of \cite[Lemma 5.1]{MR3744892}.
\end{proof}

We now set $$\C^1:= \set{f \in \C:\ f \text{ is differentiable}},$$  and for $F = \sum f_r \exp^{-r} \in \Gs{\C^1}{M_0}$, we define $$F':= \sum (f'_r-rf_r) \exp^{-r} \in \Gs{\C}{M_0}.$$

\begin{lemma}
	\label{derivative_of_strong_asymptotic_expansion}
	Let $f \in \C$ have strong asymptotic expansion $F \in \Gs{\C^1}{M_0}$.  Then $f$ is differentiable and has strong asymptotic expansion $F'$.
\end{lemma}

\begin{proof}
	Write $F = \sum_{r \in \RR} f_r \exp^{-r}$; note that $F$ has $M_0$-natural support by hypothesis.
	Let $1 > \epsilon > 0$ and $C>0$ be such that $U_C^\epsilon$ is a domain of strong asymptotic expansion of $f$ with corresponding analytic continuation $\ff$, and let $\epsilon' \in [\epsilon,1)$ and $D>C$.  Then $\ff':U_D^\epsilon \into \CC$ is a complex analytic continuation of $f'$.  Moreover, since $F$ has natural support we have, for $r  \in \RR$, that
	\begin{equation*}
	\ff' - \sum_{s \le r} (\ff_s'-s\ff_s)\bexp^{-s} = \left(\ff - \sum_{s \le r} \ff_s\bexp^{-s}\right)' \prec_{U^{\epsilon'}_D} \bexp^{-r},
	\end{equation*}
	by Lemma \ref{derivative_o}(1) and Condition $(\ast)_{f,\exp^{-r}}$ of Definition \ref{asymptotic_def}.  This shows that $f'$ has strong asymptotic expansion $F'$ in $U_D^{\epsilon'}$, as claimed.
\end{proof}

Returning to our discussion of closure under differentiation, it follows from the previous lemma that $h'$ has strong asymptotic expansion $$\tau_i(h)' = \sum_{r \ge 0} (h'_r-rh_r) \exp^{-r},$$  and we have 
\begin{equation*}
	h_r' = \left(g_r' \circ \left(f_{k-i+1} \circ f_{k-i}^{-1}\right)\right) \cdot \left(f_{k-i+1} \circ f_{k-i}^{-1}\right)'
\end{equation*}
and
\begin{align*}
	\left(f_{k-i+1} \circ f_{k-i}^{-1}\right)' &= \left(f_{k-i+1}' \circ f_{k-i}^{-1}\right) \cdot \left(f_{k-i}^{-1}\right)' \\ 
	&= \left(f_{k-i+1}' \circ f_{k-i}^{-1}\right) \cdot \frac1{f_{k-i}' \circ f_{k-i}^{-1}} \\
	&= \frac{f_{k-i+1}'}{f_{k-i}'} \circ f_{k-i}^{-1}.
\end{align*}
We do not know whether $f_{k-i}$ belongs to $\K_f$,  because the corresponding monomial in $\K_f$ is $\exp \circ \left(-f_{k-i}\right)$, not $f_{k-i}$ itself.  We also do not know if $f_{k-i}'$ belongs to $\K_f$.  However, each $f_{k-i}$ and each $f_{k-i}'$ does belong to $\H$, so they are given by convergent LE-series.  The next proposition shows that there is a good such representation for our purposes:

\begin{prop}
	\label{opt_rep_prop}
	There are $l \ge k$, $g_0 > \cdots > g_l$ in $\U$ and a strictly increasing $\iota:\{0, \dots, k\} \into \{0, \dots, l\}$ such that $\iota(k) = l$ and $g:= (g_0, \dots, g_l)$ is admissible and, for each $i = 0, \dots, k$, we have $g_{\iota(i)} = f_i$ and, if $i \in \{0, \dots, k-1\}$ and $j \in \{i+1, \dots, k\}$, the germ $f_{j}'/f_{i}'$ is given by a convergent LE-series in $\Gs{\RR}{M_{(g_{\iota(i)}, \dots, g_l)}}$.
\end{prop}

The proof of Proposition \ref{opt_rep_prop} relies on the next three lemmas, which build on the description of $\H$ in  \cite[Section 2]{Kaiser:2017aa}.  Recall that $$\U = \bigcup_{k \in \NN} \E^\infty \circ \log_k,$$ where $$\E^\infty = \oplus_{n \in \NN} \left(\P\E_n^\infty \cup \{0\}\right)$$ and each $\P\E_n^\infty$ is the set of all purely infinite $\Lanexp$-germs of exponential height $n$.  Note that, for $f \in \H$, we have $$\M(f) = \supp S_\la(f)$$ is the set of principal monomials of $f$ as defined in \cite[Section 2]{Kaiser:2017aa}.

For $h \in \H$ we denote by $\compclass(h)$ the \textbf{comparability class} of $h$, that is, the set of all $g \in \H$ that are comparable to $h$ (see Example \ref{natural_expls}(3)).  The set of all comparability classes of germs in $\H$ is linearly ordered by the obvious ordering induced from $<$, and we also denote it by $<$.  Corresponding to Definition \ref{str_slower_cc}, for $f,g \in \H$ we say that $f$ \textbf{has faster comparability class than $g$} (or that $g$ \textbf{has slower comparability class than $f$}) if either 
$$f \in \I \quad\text{and}\quad \compclass(1/f) \le \compclass(g) \le \compclass(f)$$ or $$1/f \in \I \quad\text{and}\quad \compclass(f) \le \compclass(g) \le \compclass(1/f).$$

\begin{rmk}
	Proposition \ref{opt_rep_prop} implies that the germ $f_{k-i+1}'/f_{k-i}'$ can be written as a convergent LE-series with monomials whose comparability class is slower than that of $\exp \circ \left(-f_{k-i}\right)$.  This is not true in general for the germ $f_{k-i}$ itself: take, for instance, $f_{k-i} = \log \circ \log$.  Then $f_{k-i}' = \frac1{x\log}$, which is a germ in $\la$ comparable to $1/x$.  However, $\exp \circ (-\log \circ \log) = 1/\log$ has slower comparability class than $1/x$.  (The proof of Proposition \ref{opt_rep_prop} below does show, however, that it holds with $f_{k-i}'$ in place of $f_{k-i+1}'/f_{k-i}'$ if $f_{k-i}$ belongs to $\M$, that is, if $f_{k-i}'$ is an $\Lanexp$-monomial.)
\end{rmk}

\begin{lemma}
	\label{H-lemma_1}
	Let $n \in \NN$ and $f \in \P\E_n^\infty$.  Then $f' \in \P\E_n^\infty \oplus \RR$ and $\compclass(f') \le \compclass(f)$.
\end{lemma}

\begin{proof}
	By induction on $n$; the case $n=0$ follows from the observation that $\M(f) \subseteq \{x^k:\ k \in \NN\setminus\{0\}\}$, so that $\M(f') \subseteq \{x^k:\ k \in \NN\}$.  So we assume $n>0$ and the lemma holds for lower values of $n$.  Then every $m \in \M(f)$ is of the form $\exp \circ g$ for some $g \in \P\E^\infty_{n-1}$, so by the inductive hypothesis we have $$m' = m \cdot g'$$ with $g' \in \P\E_{n-1}^\infty \oplus \RR$.  It follows from \cite[Remark 2.2(2)]{Kaiser:2017aa} that every $m \in \M(f')$ belongs to $\P\E_n^\infty$ and satisfies $\compclass(m) \le \compclass(\lm(f)) = \compclass(f)$.
\end{proof}

\begin{lemma}
	\label{H-lemma_2}
	Let $g,h \in \U$ be such that $h < g$, and let $m \in \M\left(h'/g'\right)$.  Then $m$ has slower comparability class than $g$.
\end{lemma}

\begin{proof}
	Choose $k \in \NN$ such that $H:= h \circ \exp_k \in \E$ and $G:= g \circ \exp_k \in \E$.  Then $h,g \in \E^\infty$, so by Lemma \ref{H-lemma_1} we have $H',G' \in \E^\infty \oplus \RR$.  In particular, for every $n \in \M(H') \cup \M(G')$, we have $$\compclass(1) \le \compclass(n) \le \compclass(G).$$  Since every $M \in \M\left(H'/G'\right)$ belongs to the multiplicative group generated by $\M(H') \cup \M(G')$, it follows that every such $M$ has slower comparability class than $G$.  Finally, since $h' = (H' \circ \log_k) \cdot (\log_k)'$ and $g' = (G' \circ \log_k) \cdot (\log_k)'$, we have $\frac{h'}{g'} = \frac{H'}{G'} \circ \log_k$, and the lemma follows.
\end{proof}

\begin{lemma}
	\label{H-lemma_3}
	Let $h \in \H$ and $g \in \I$ be such that every $m \in \M(h)$ has slower comparability class than $g$.  Then there exist $k \in \NN$, small $m_0, \dots, m_k \in \la$, a convergent Laurent series $G(X_0, \dots, X_k)$ over $\RR$ such that $h = G(m_0, \dots, m_k)$ and each $m_i$ has slower comparability class than $g$.
\end{lemma}

\begin{proof}
	By Lemma \ref{strong_basis_lemma}, there exist an admissible $f = (f_0, \dots, f_k) \in \U^{k+1}$ and a convergent Laurent series $G(X_0, \dots, X_k)$ over $\RR$ such that $h = G(m_0, \dots, m_k),$  where $m_i:= \exp \circ (-f_i)$ for each $i$.  It follows from the admissibility of $f$ that $\compclass(m_0) < \cdots < \compclass(m_k) < \compclass(1),$ so that every $m \in \M(h)$ has slower comparability class than $m_0$.  Moreover, we may assume that $k$ in minimal, that is, there exists $\tilde m \in \M(h)$ such that $m_0$ has slower comparability class than $\tilde m$.  Since $\tilde m$ has slower comparability class than $g$ by assumption, it follows that $m_0$ has slower comparability class than $g$, which proves the lemma.
\end{proof}

\begin{proof}[Proof of Proposition \ref{opt_rep_prop}]
	Let $i \in \{0, \dots, k-1\}$.  By Lemma \ref{H-lemma_2}, with $g = f_i > f_{i+1} = h$, we have that every principal monomial of $f_{i+1}'/f_i'$ has slower comparability class than $f_i$.  So by Lemma \ref{H-lemma_3}, with $g = f_i$ and $h = f_{i+1}'/f_i'$, there exist $k_i \in \NN$, small $m_{i,0}, \dots, m_{i,k_i} \in \la$ and a convergent Laurent series $G(X_{i,0}, \dots, X_{i,k_i})$ over $\RR$ such that $$\frac{f_{i+1}'}{f_i'} = G(m_{i,0}, \dots, m_{i,k_i})$$ and each $m_{i,j}$ has slower comparability class than $f_i$.  For $j = 0, \dots, k_i$ set $$f_{i,j}:= -\log \circ m_{i,j},$$ which belongs to $\U$ by definition of $\la$.  Then $f_{i,j} \preceq \log \circ f_i \prec f_i$ for each $j$; in particular, we have $$f_i > f_{i,0} > \cdots > f_{i,k_i}.$$  Now apply Lemma \ref{strong_basis_lemma} to obtain an admissible subtuple $g$ of the tuple consisting of all $f_i$ and all $f_{i,j}$, in such a way that $f \subseteq g$.
\end{proof}

It remains to show how Proposition \ref{opt_rep_prop} implies closure under differentiation for $\K$.  The tricky part here is that, while formal differentiation makes sense for LE-series in general, the construction of $\K$ uses ``shifted'' LE-series (based on the monomials $M_{f,i}$, which are not in $\la$ in general), for which no such formal differentiation exists.  To make the connection between analytic and formal differentiation, we will need to introduce some auxiliary notions of differentiation for these ``shifted'' series.

Note that in the situation of Proposition \ref{opt_rep_prop}, by the Construction Theorem \ref{Construction_thm}(2) and the definition of $\iota$, we have $\K_{f,i} \subseteq \K_{g,\eta(i)}$, where we define $\eta:\{0, \dots, k\} \into \{0, \dots, l\}$ by $$\eta(i):= l - \iota(k-i);$$  in particular, we have $\eta(0) = l - \iota(k) = 0$ and $\eta(k) = l$.

\begin{prop}
	\label{derivative_prop}
	Let $l \ge k$, $g = (g_0, \dots, g_l)$ and $\iota:\{0, \dots, k\} \into \{0, \dots, l\}$ be for $f$ as in Proposition \ref{opt_rep_prop}.  Then, for $i \in \{0, \dots, k\}$ and $h \in \K_{f,i}$, we have $h' \in \K_{g,\eta(i)}$ and  $\tau_{g,\eta(i)}(h') = (\tau_{f,i} h)'.$
\end{prop}

\begin{proof}
	By induction on $i$; say $\tau_{f,i}(h) = \sum h_r \exp^{-r}$ with $h_r \in \K_{f,i}'$.  If $i=0$, then $(\tau_{f,i} h)' \in \Gs{\RR}{M_{0}}$ because the coefficients of $\tau_{f,i} h$ are real numbers.  If $i>0$, then $h_r = g_r \circ \left(f_{k-i+1} \circ f_{k-i}^{-1}\right)$ for each $r$, with $g_r \in \K_{f,i-1} \subseteq \K_{g,\eta(i-1)}$.  By the inductive hypothesis, we have $g_r' \in \K_{g,\eta(i-1)}$, so that $$g_r' \circ \left(f_{k-i+1} \circ f_{k-i}^{-1}\right) = g_r' \circ \left(g_{\iota(k-i+1)} \circ g_{\iota(k-i+1)-1}^{-1}\right) \circ \cdots \circ \left(g_{\iota(k-i)+1} \circ g_{\iota(k-i)}^{-1}\right)$$ belongs to $\K_{g,\eta(i)}$.  Also, by Corollary \ref{construction_cor}(2) and Proposition \ref{opt_rep_prop}, we have $f_{k-i+1}'/f_{k-i}' \in \K_{g,\eta(i)}$.  Since $$g^{\langle \eta(i) \rangle} = (g_{\iota(k-i)}, \dots, g_l) \circ f_{k-i}^{-1},$$  it follows from the observations before Proposition \ref{opt_rep_prop} that $(\tau_{f,i} h)' \in \Gs{\K'_{g,\eta(i)}}{M_0}$.  Finally, by Lemma \ref{derivative_of_strong_asymptotic_expansion} the series $(\tau_{f,i} h)'$ is a strong asymptotic expansion of $h'$.
\end{proof}

%

Next, by Corollary \ref{convergent_prop}(2), every $h \in \H$ is given by the unique convergent LE-series $S_\la(h)$.  Moreover, if $m_0, \dots, m_k \in \H$, $r_0, \dots, r_k \in \RR$ and $m = m_0^{r_0} \cdots m_k^{r_k}$, then by the rules of differentiation we have $$m' = \sum_{i=0}^k r_i \cdot m \cdot \frac {m_i'}{m_i} = \sum_{i=0}^k r_i \cdot m \cdot (\log m_i)'.$$ 
Assuming in addition that $\{m_0, \dots, m_k\}$ is multiplicatively $\RR$-linearly independent, the exponents $r_0, \dots, r_k$ are uniquely determined by $m$.  In this situation, we set $\partial_i m:= r_i \cdot m$; so the formula for the derivative becomes $$m' = \sum_{i=0}^k \partial_i m \cdot (\log m_i)'.$$
Thus, for $F = \sum_m a_m m \in \Gs{\C^1}{\langle m_0, \dots, m_k\rangle^\times}$ we set $$\delta F:= \sum_m a_m' m$$ and $$\partial_i F:= \sum_m a_m \partial_i m \quad\text{for } i = 0, \dots, k.$$ These are generalized series in $\Gs{\C}{\langle m_0, \dots, m_k\rangle^\times}$, and we define a derivative $$F^\dagger := \delta F + \sum_{i=0}^k \partial_i F \cdot (\log m_i)',$$ an element of the additive vector space $$\Gs{\C}{\langle m_0, \dots, m_k\rangle^\times}\left\langle(\log m_0)', \dots, (\log m_k)'\right\rangle$$ over $\Gs{\C}{\langle m_0, \dots, m_k\rangle^\times}$.

\begin{expls}\label{dagger_derivative_expls}
	\begin{enumerate}
		\item For $F \in \Gs{\C^1}{M_0}$, we have $F' = F^\dagger$.
		\item For $X = (X_0, \dots, X_k)$, a generalized power series $G(X) = \sum_{r \in [0,\infty)^{k+1}} a_r X^r$ over $\RR$ and $i=0, \dots, k$ we write, as in \cite{Dries:1998xr}, $$\partial_i G(X):= \sum_{r \in [0,\infty)^{k+1}} r_ia_r X^r.$$  Then, if each $m_i$ above is small and $\{m_0, \dots, m_k\}$ is multiplicatively $\RR$-linearly independent, we have $$\partial_i(G(m_0, \dots, m_k)) = (\partial_iG)(m_0, \dots, m_k),$$ so that $G(m_0, \dots, m_k)^\dagger = \sum_{i=0}^k \partial_i(G(m_0, \dots, m_k)) \cdot (\log m_i)',$ as one would expect from the chain rule for differentiation.
	\end{enumerate}
\end{expls}

\begin{lemma}
	\label{inductive_dagger_lemma}
	Let $m_0, \dots, m_k \in \H$ be such that $\{m_0, \dots, m_k\}$ is multiplicatively $\RR$-linearly independent.  Let $F \in \Gs{\RR}{\langle m_0, \dots, m_k\rangle^\times}$, and write $F = \sum_{r \in \RR} F_r m_k^r$, where each $F_r$ belongs to $\Gs{\RR}{\langle m_0, \dots, m_{k-1}\rangle^\times}$.  Then $$F^\dagger = \sum_{r \in \RR} \left(r(\log m_k)' F_r + F_r^\dagger\right) m_k^r.$$
\end{lemma}

\begin{proof}
	Note that $\delta F = 0 = \delta F_r$, for each $r \in \RR$.  Also, $$\partial_k F = \sum_{r \in \RR} rF_r m_k^r$$ while, for $i = 0, \dots, k-1$, $$\partial_i F = \sum_{r \in \RR} (\partial_i F_r) m_k^r.$$  Therefore, by definition of $F^\dagger$, we have
	\begin{align*}
		F^\dagger &= \sum_{r \in \RR} r(\log m_k)'F_rm_k^r + \sum_{i=0}^{k-1} \left(\sum_{r \in \RR} (\partial_i F_r) m_k^r\right) (\log m_i)' \\
		&= \sum_{r \in \RR} \left(r(\log m_k)' F_r + \sum_{k=0}^{k-1} (\partial_i F_r) (\log m_i)'\right) m_k^r \\
		&= \sum_{r \in \RR} \left(r(\log m_k)' F_r + F_r^\dagger\right) m_k^r,
	\end{align*}
	as claimed.
\end{proof}

\begin{lemma}
	\label{dagger_derivative_lemma}
	Let $m_0, \dots, m_k \in \H$ be such that $\{m_0, \dots, m_k\}$ is multiplicatively $\RR$-linearly independent.  Let $h \in \I$ and $F \in \Gs{\C^1}{M}$, where $M:= \langle m_0, \dots, m_k\rangle^\times$.  Then the series $(F \circ h)^\dagger$ is an element of the vector space $\Gs{\C}{M \circ h}\left\langle(\log (m_0\circ h))', \dots, (\log (m_k\circ h))'\right\rangle$,  and we have  $$(F \circ h)^\dagger = \left(F^\dagger \circ h\right) \cdot h'.$$
\end{lemma}

\begin{proof}
	For $m = \langle m_0, \dots, m_k \rangle^\times$, we have $(m \circ h)' = (m' \circ h) \cdot h'$ by the chain rule; hence $\partial_i (m \circ h) = (\partial_i m) \circ h.$  It follows that $$\partial_i (F \circ h) = (\partial_i F) \circ h.$$  On the other hand, we have $F \circ h \in \Gs{\C^1}{M \circ h}$ and $$\delta (F \circ h) = (\delta F \circ h) \cdot h'.$$  The lemma now follows by direct computation.
\end{proof}

The reason for introducing the derivative $(\cdot)^\dagger$ is that the series $T_{f,i}h$, for $h \in \K_{f,i}$, are not in $\Gs{\RR}{\la}$ in general.  However, for a generalized series $F \in \Gs{\RR}{\langle m_0, \dots, m_k\rangle^\times}$ with $m_0, \dots, m_k \in \la$ small and $\{m_0, \dots, m_k\}$ multiplicatively $\RR$-linearly independent., we have $\delta F = 0$, and we define $$F' := \sum_{i=0}^k \partial_i F \cdot S_\la((\log m_i)'),$$ which is a series in $\Gs{\RR}{\la}$.  In this situation, it is straightforward to verify that if $F$ is an $\la$-series (respectively, an $\la$-generalized Laurent series), then $F'$ is an $\la$-series (respectively, an $\la$-generalized Laurent series); we leave the details to the reader.  Moreover, $F'$ is equal to the derivative of $F$ as defined in \cite{MR1848569} (this follows from the  usual properties established in \cite{MR1848569} for the latter).

The crucial link between the two derivatives, in the situation of the admissible tuple $f = (f_0, \dots, f_k)$, is given as follows:  writing $m_i:= \exp \circ (-f_i)$, we have $(\log m_i)' = -f_i'$, for each $i$.  Thus, for $F \in \Gs{\RR}{M_f}$, we have $F^\dagger \in \Gs{\RR}{M_f}\langle f_0', \dots, f_k'\rangle$.  Therefore, we let $\delta_{f}: \Gs{\RR}{M_{f}}\left\langle f_0', \dots, f_k'\right\rangle \into \Gs{\RR}{\la}$ be the unique $\Gs{\RR}{M_f}$-vector space homomorphism that is the identity on $\Gs{\RR}{M_f}$ and satisfies $$\delta_{f}\left(f_j'\right) := S_\la\left(f_j'\right), \quad\text{ for } j=0, \dots, k;$$ then for $F \in \Gs{\RR}{M_f}$, we have $$F' = \delta_f\left(F^\dagger\right).$$
Note that, given $i \in \{0, \dots, k\}$ and $F \in \Gs{\RR}{M_{(f_{k-i}, \dots, f_k)}}$, we have $F'/S_\la(f_{k-i}') \in \Gs{\RR}{M_{(g_{\iota(k-i)}, \dots, g_l)}}$.  Therefore, we lift $\delta_f$ to the unique $\Gs{\RR}{M_{f^{\langle i \rangle}}}$-vector space homomorphism $$\delta_{f,i}:\Gs{\RR}{M_{f^{\langle i \rangle}}}\langle1, \dots, (f_k \circ f_{k-i}^{-1})'\rangle \into \Gs{\RR}{M_{g^{\langle \eta(i) \rangle}}}$$ that is the identity on $\Gs{\RR}{M_{f^{\langle i \rangle }}}$ and satisfies, for $j \ge k-i$, $$\delta_{f,i}\left(\left(f_j \circ f_{k-i}^{-1}\right)'\right) = S_\la\left(\frac{f_j'}{f_{k-i}'}\right) \circ f_{k-i}^{-1};$$  in particular, the map $\delta_{f,0}$ is the identity map on $\Gs{\RR}{M_0}$.  Note that, by Proposition \ref{opt_rep_prop}, we have 
\begin{equation}
	\label{delta_f_i_label}
	\delta_{f,i}\left(\left(f_j \circ f_{k-i}^{-1}\right)'\right) = T_{g,\eta(i)}\left(\left(f_j \circ f_{k-i}^{-1}\right)'\right).
\end{equation}
The next lemma exhibits the relationships between $\delta_f$ and the various $\delta_{f,i}$.  To simplify notations, we set $n_i:= f_{k-i+1} \circ f_{k-i}^{-1}$.

\begin{lemma}
	\label{delta_f_i_lemma}
	\begin{enumerate}
		\item Let $i \in \{1, \dots, k\}$ and $$G = \sum_{j=k-i+1}^k G_j \cdot \left(f_j \circ f_{k-i+1}^{-1}\right)'$$ belong to $\Gs{\RR}{M_{f,i-1}}\left\langle 1, \dots, \left(f_k \circ f_{k-i+1}^{-1}\right)' \right\rangle$.  Then $\left(G \circ n_i\right) \cdot n_i'$ belongs to $\Gs{\RR}{M_{f,i}}\left\langle 1, \dots, \left(f_k \circ f_{k-i}^{-1}\right)' \right\rangle$ and satisfies $$\delta_{f,i}\left((G \circ n_i) \cdot n_i'\right) = (\delta_{f,i-1}(G) \circ n_i) \cdot \delta_{f,i}(n_i').$$
		\item Let $$G = \sum_{j=0}^k G_j \cdot f_j'$$ belong to $\Gs{\RR}{M_{f}}\left\langle f_0', \dots, f_k' \right\rangle$.  Then the series $\frac{G}{f_0'} \circ f_0^{-1}$ belongs to $\Gs{\RR}{M_{f,k}}\left\langle 1, \dots, \left(f_k \circ f_{0}^{-1}\right)' \right\rangle$ and satisfies $$\delta_{f,k}\left(\frac{G}{f_0'} \circ f_0^{-1}\right) = \frac{\delta_f(G)}{S_\la(f_0')} \circ f_0^{-1}.$$
	\end{enumerate}
\end{lemma}

\begin{proof}
	(1) We have
	\begin{align*}
		G \circ n_i &= \sum_{j=k-i+1}^k (G_j \circ n_i) \cdot \left(\frac{f_j'}{f_{k-i+1}'} \circ f_{k-i}^{-1}\right) \\
		&= \left[\sum_{j=k-i+1}^k (G_j \circ n_i) \cdot \left(\frac{f_j'}{f_{k-i}'} \circ f_{k-i}^{-1}\right)\right] \cdot \left(\frac{f_{k-i}'}{f_{k-i+1}'} \circ f_{k-i}^{-1}\right) \\
		&= \frac1{n_i'} \cdot \left[\sum_{j=k-i+1}^k (G_j \circ n_i) \cdot \left(f_j \circ f_{k-i}^{-1}\right)'\right].
	\end{align*}
	Since each $G_j \circ n_i$ belongs to $\Gs{\RR}{M_{f,i}}$, it follows that $(G \circ n_i) \cdot n_i'$ belongs to $\Gs{\RR}{M_{f,i}}\left\langle 1, \dots, \left(f_k \circ f_{k-i}^{-1}\right)' \right\rangle$.  By definition of $\delta_{f,i}$ and $\delta_{f,i-1}$, we have
	\begin{align*}
		&\delta_{f,i}((G \circ n_i) \cdot n_i') \\
		&= \sum_{j=k-i+1}^k (G_j \circ n_i) \cdot  \left( S_\la\left(\frac{f_j'}{f_{k-i}'}\right)  \circ f_{k-i}^{-1} \right) \\
		&= \sum_{j=k-i+1}^k (G_j \circ n_i) \cdot  \left( S_\la\left(\frac{f_j'}{f_{k-i+1}'}\right)  \circ f_{k-i+1} \circ n_i \right) \cdot \left(S_\la\left(\frac{f_{k-i+1}'}{f_{k-i}'}\right) \circ f_{k-i}^{-1}\right) \\
		&= (\delta_{f,i-1}(G) \circ n_i) \cdot \delta_{f,i}(n_i'),
	\end{align*}
	as claimed.  The proof of (2) is similar, if easier, and it is left to the reader.
\end{proof}

\begin{lemma}
	\label{series_dagger_derivation}
	Let $i \in \{0, \dots, k\}$ and $h \in \K_{f,i}$.  Then $h' \in \K_{g,\eta(i)}$ and $$T_{g,\eta(i)}(h') = \delta_{f,i} \left(\left(T_{f,i} h\right)^\dagger\right).$$
\end{lemma}

\begin{proof}
	By induction on $i$.  Assume first that $i=0$; since $\eta(0) = 0$ and $\delta_{f,0}$ is the identity map, we have
	\begin{align*}
		\left(T_{f,0} h\right)^\dagger &= \left(\tau_{f,0} h\right)^\dagger & \\
		&= \left(\tau_{f,0} h\right)' &\text{by Example \ref{dagger_derivative_expls}(1)} \\
		&= \tau_{g,\eta(0)}(h') &\text{by Proposition \ref{derivative_prop}} \\
		&= T_{g,0}(h') &\text{because } T_{g,0} = \tau_{g,0}.
	\end{align*}	
	
	Therefore, we assume $i>0$ and the lemma holds for lower values of $i$.  Let $$\tau_{f,i}(h) = \sum (h_r \circ n_i) \exp^{-r}$$ with $h_r \in \K_{f,i-1}$.  Then
	\begin{align*}
		\left(T_{f,i} h\right)^\dagger
		&= \left( \sum \left(T_{f,i-1}(h_r) \circ n_i\right) \exp^{-r}\right)^\dagger & \\
		&= \sum \left(-r T_{f,i-1}(h_r) \circ n_i + \left(T_{f,i-1}(h_r) \circ n_i\right)^\dagger\right) \exp^{-r}  \\ 
		&= \sum \left(-r T_{f,i}(h_r \circ n_i) + \left((T_{f,i-1} h_r)^\dagger \circ n_i\right) \cdot n_i'\right) \exp^{-r},
	\end{align*}
	where the second equality follows from Lemma \ref{inductive_dagger_lemma} and the third from Lemma \ref{dagger_derivative_lemma}.  Now note that
	\begin{align*}
		T_{g,\eta(i)}((h_r \circ n_i)') &= T_{g,\eta(i)}((h_r'\circ n_i) \cdot n_i') &\\
		&= \left(T_{g,\eta(i-1)}(h_r') \circ n_i\right) \cdot \delta_{f,i}(n_i') &\text{by } \eqref{delta_f_i_label} \\
		&= \left(\delta_{f,i-1}\left(T_{f,i-1}(h_r)^\dagger\right) \circ n_i\right) \cdot \delta_{f,i}(n_i') &\text{by ind. hyp.}\\
		&= \delta_{f,i} \left(\left(T_{f,i-1}(h_r)^\dagger \circ n_i\right) \cdot n_i' \right) &\text{by Lemma \ref{delta_f_i_lemma}(1)}.
	\end{align*}
	
	Since $h_r \circ n_i \in \K_{f,i} \subseteq \K_{g,\eta(i)}$ and $T_{g,\eta(i)}$ agrees with $T_{f,i}$ on $\K_{f,i}$, and since $\delta_{f,i}$ is the identity on $\K_{f,i}$, it follows that
	\begin{align*}
		\delta_{f,i}\left(T_{f,i} h)^\dagger\right)
		&= \sum \left(-r T_{f,i}(h_r \circ n_i) + \delta_{f,i}\left((T_{f,i-1} h_r)^\dagger \circ n_i\right) \cdot n_i'\right) \exp^{-r} \\
		&= \sum \left(-r T_{g,\eta(i)}(h_r \circ n_i) + T_{g,\eta(i)}((h_r \circ n_i)')\right) \exp^{-r} \\
		&= \sum T_{g,\eta(i)} \left(-r h_r \circ n_i + (h_r \circ n_i)'\right) \exp^{-r} \\
		&= T_{g,\eta(i)}(h'),
	\end{align*}
	where the last equality follows from Proposition \ref{derivative_prop} and the definition of $T_{g,\eta(i)}$.
\end{proof}

\begin{cor}
	\label{series_derivation}
	Let $h \in \K_{f}$.  Then $\frac{h'}{f_0'} \in \K_g$ and $$T_{g}\left(\frac{h'}{f_0'}\right) = \frac{\left(T_{f}h\right)'}{S_\la\left(f_0'\right)}.$$
\end{cor}

\begin{proof}
	Since $h \in \K_f$, we have $h \circ f_0^{-1} \in \K_{f,k}$.  So by Lemma \ref{series_dagger_derivation}, the germ $\left(h \circ f_0^{-1}\right)'$ belongs to $\K_{g,\eta(k)} = \K_{g,l}$ and satisfies 
	\begin{equation}\label{dagger_eq}
		T_{g,l}\left(\left(h \circ f_0^{-1}\right)'\right) = \delta_{f,k}\left(\left(T_{f,k}\left(h \circ f_0^{-1}\right)\right)^\dagger\right).
	\end{equation} 
	But $\left(h \circ f_0^{-1}\right)' = \frac{h'}{f_0'} \circ f_0^{-1}$ and $g_0 = f_0$, so that $\frac{h'}{f_0'} \in \K_g$.  Moreover, by definition of $\delta_f$, we have
	\begin{align*}
		\frac{\left(T_f h\right)'}{S_\la(f_0')} \circ f_0^{-1}
		&= \frac{\delta_f\left(\left(T_f h)\right)^\dagger\right)}{S_\la(f_0')} \circ f_0^{-1} \\
		&= \delta_{f,k}\left(\frac{\left(T_fh\right)^\dagger}{f_0'} \circ f_0^{-1}\right) &\text{by Lemma \ref{delta_f_i_lemma}(2)} \\
		&= \delta_{f,k}\left(\frac{\left(T_fh \circ f_0^{-1} \circ f_0\right)^\dagger}{f_0'} \circ f_0^{-1}\right) \\
		&= \delta_{f,k} \left(T_{f,k}\left(\left(h \circ f_0^{-1}\right)\right)^\dagger\right) &\text{by Lemma \ref{dagger_derivative_lemma}},
	\end{align*}
	where we recall that $T_f h = T_{f,k}\left(h \circ f_0^{-1}\right) \circ f_0$ by definition of $T_f$.  On the other hand, by definition of $T_g$ we have, since $g_0 = f_0$, that
	\begin{align*}
		T_g\left(\frac{h'}{f_0'}\right) \circ f_0^{-1} 
		&= T_{g,l}\left(\frac{h'}{f_0'} \circ f_0^{-1}\right) \\
		&= T_{g,l}\left(\left(h \circ f_0^{-1}\right)'\right) \\
		&= \delta_{f,k} \left(T_{f,k}\left(\left(h \circ f_0^{-1}\right)\right)^\dagger\right) &\text{by Lemma \ref{series_dagger_derivation}},
	\end{align*}
	which proves the corollary.
\end{proof}

\begin{cor}
	\label{derivative_cor}
	$\K$ is closed under differentiation and for $h \in \K$, we have $T(h') = (Th)'$.
\end{cor}

\begin{proof}
	This corollary follows from the previous corollary, because the Ilyashenko field $(\K,\la,T)$ extends $(\H,\la,S_\la)$ and each $(\K_f,M_f,T_f)$, with $f$ any admissible tuple of germs in $\la$.
\end{proof}

\section{Concluding remarks} \label{final}

In the spirit of \'Ecalle's field of ``fonctions analysables'' \cite{MR1094378} and the concluding remark of van den Dries et al. \cite{Dries:1997jl}, one might ask whether the Hardy field $\K$ constructed here is closed under composition and integration.  More generally, we wonder how far $\K$ is from being a model of the recently described theory of the field $\mathbb T$ of transseries, see Aschenbrenner et al. \cite{MR3585498}.

At this point, we do not know if $\K$ is closed under composition, not even under composition on the right by infinitely increasing germs in $\H$.  
%
%
The latter appears to be related to whether or not one can canonically assign a subalgebra $\K_h$ of $\K$ to every tuple $h$ of infinitely increasing germs in $\H$.  (This is done here, by construction, for every tuple $f$ of infinitely increasing germs in $\la$ only.)

Full closure under composition seems unlikely, as $\K$ is unlikely to be closed even under left composition by $\exp$: the general asymptotic expansion of a germ $h \in \K$ is a \textit{divergent} LE-series with \textit{convergent} LE-monomials.  Under left composition with $\exp$, the infinite part of such an asymptotic expansion becomes an LE-monomial that is \textit{not} a \textit{convergent} LE-monomial.  Thus, $\exp \circ h$ would have an asymptotic LE-series expansion some of whose LE-monomials are not convergent; in particular, $\exp \circ h$ cannot belong to $\K$.

Finally, we do not know if $\K$ is closed under integration.  However, based on Camacho's Ph.D. thesis \cite{camacho2016}, we suspect that the Liouville closure of a Hardy qaa field is again a Hardy qaa field; this would provide a canonical way of closing $\K$ under exponentiation and integration.  A more elaborate way to extend $\K$ is that of showing that $\K$ generates an o-minimal expansion $\RR_\K$ of the real field, as we hope to do in a future paper.  The Hardy field of all germs at $+\infty$ of unary functions definable in its pfaffian closure $\P(\RR_\K)$ as in \cite{Speissegger:1999nt} would then give a Hardy field extension of $\K$ closed under composition (and hence exponentiation) and integration.

\end{document}